\documentclass[10pt,reqno]{amsart} % please use amsart at 11pt

\usepackage{amssymb,latexsym,  
amscd, 
amsthm, 
graphicx, epsfig
}
\usepackage{cite} % to get Refs. [1,2,3] typeset as [1--3] automatically

\usepackage[height=190mm,width=130mm]{geometry} % this is the journal
\theoremstyle{plain}

\newtheorem{lemma}{Lemma}
\newtheorem{corollary}{Corollary}
\newtheorem{proposition}{Proposition}

\theoremstyle{definition}
\newtheorem{definition}{Definition}

\newtheorem{remark}{Remark}

\newcommand{\Z}{\mathbb{Z}}
\newcommand{\W}{\mathcal{W}} 
\newcommand{\U}{\mathcal{U}}

\newcommand{\Oo}{\mathcal{O}}
\newcommand{\C}{\mathbb{C}}

\numberwithin{equation}{section} % to get equations numbered
\newcommand{\nn}{\nonumber \\}

 \newcommand{\res}{\mbox{\rm Res}}

\renewcommand{\hom}{\mbox{\rm Hom}}

\newcommand{\wt}{\mbox{\rm wt}\ }

\newcommand{\N}{\mathbb{N}}

\newcommand{\F}{\mathcal{F}}

\newcommand{\V}{\mathcal{V}}
\newcommand{\one}{\mathbf{1}}

\begin{document}
\title[Product-type classes for vertex algebra cohomology] 
{Product-type classes for vertex algebra cohomology of foliations on   
 complex curves} 
%%
 % please provide
                                % an abbreviated title 
%%\newcommand{\half}{\frac{1}{2}}
%%%%%%%%%%%%%%%%%%%%%%%%%%%%%%%%%%%%%%%%%%%%%%%%%%%%%%%%%%%%%%%%%%%%%%%%%%%%%%%
\author{A. Zuevsky}
\address{Institute of Mathematics \\ Czech Academy of Sciences\\ Praha}

\email{zuevsky@yahoo.com}

%%\thanks{The first author was supported in part by the XXX Fund, Grant No.~YYYYY.}

%%\author{B. Author}

%%\address{Institute of Something\\ Some Academy\\ P.~O. Box 0001 \\ 12346
 %% City\\ Country}

%%\curraddr{Other Institution\\ Some Other Department\\ P.~O. Box 0002 \\ 12347
%%  City\\ Other Country}

%%\email{bauthor@email.net}

%%%%%%%%%%%%%%%%%%%%%%%%%%%%%%%%%%%%%%%%%%%%%%%%%%%%%%%%%%%%%%%%%%%%%%%%%
% You may repeat \author \address as often as necessary                 %
%%%%%%%%%%%%%%%%%%%%%%%%%%%%%%%%%%%%%%%%%%%%%%%%%%%%%%%%%%%%%%%%%%%%%%%%%
%% 
\begin{abstract}
We define a product of pairs of double complex spaces $C^n_m(V, \W, \F)$ 
for gradig-restricted vertex algebra cohomology of 
codimension one foliation on a complex curve. 
We introduce a vertex algebra counterpart of the classical %product-type 
cohomological class using the orthogonality conditions on elements of 
double complex spaces with respect to the product we introduced.  
\end{abstract}

\vskip12pt  % insert '\vskip12pt' while using '\twocolumn' command
%\vskip28pt % if there is no keywords

\maketitle
%%
%%%%%%%%%%%%%%%%%%%%%%%%%%%%%%%%%%%%%%%%%%%%%%%%%%%%%%%%%%%%%%%%%%%%%%%%%%%%%%%%%%%%%%%%%%%
%%%%%%%%%%%%%%%%%%%%%%%%%%%%%%%%%%%%%%%%%%%%%%%%%%%%%%%%%%%%%%%%%%%%%%%%%%%%%%%%%%%%%%%%%%5
\section{Introduction}
\label{introduction}
The theory of foliations involves a bunch of approaches \cite{Bott, BS, BH, CM, 
 Fuks1, Fuks2, GF1, GF2, GF3, Losik} and many others. 
In certain cases it is useful to express cohomology in terms of connections and 
use the language of connections in order to study leave spaces of foliations.   
Connections appear in conformal field theory \cite{FMS, BZF} in definitions of many notions and formulas. 
Vertex algebras, generalizations of ordinary Lie algebras, are essential in conformal field theory. 
The theory of vertex algebra characters is a rapidly developing 
field of studies. Algebraic nature of methods applied in this field helps to understand and compute  
the structure of vertex algebra characters. 
On the other hand, the geometric side of vertex algebra characters is in associating their formal parameters with 
local coordinates on a complex variety.   
Depending on geometry, one can obtain various consequences for a vertex algebra and its space 
of characters, and vice-versa, one can study geometrical property of a manifold by using 
algebraic nature of a vertex algebra attached. 
%%%%%
In this paper we use the vertex algebra cohomology theory \cite{Huang} in order to construct cohomological 
invariants for codimension one foliations %on complex curves. 
%%

%%%%%%%%%%%%%%%%%%%%%%%%%%%%%%%%%%%%%%%%%%%%%%%%%%%%%%%%%%%%%%%%%%%%%%%%%%%%%%%%%%%%%%%%%%%%%%
%%%%%%%%%%%%%%%%%%%%%%%%%%%%%%%%%%%%%%%%%%%%%%%%%%%%%%%%%%%%%%%%%%%%%%%%%%%%%%%%%%%%%%%%%%%%%%
%% 
There exist a few approches to definition and computation of cohomologies of vertex operator algebras. 
\cite{Huang, Li}.  
Taking into account the above definitions and construction, 
we aim to consideration of a characteristic classes theory for arbitrary codimension regular and singular 
 foliations vertex operator algebras. 
Losik \cite{Losik} defines a smooth structure on the leaf space
$M/\mathcal{F}$ of a foliation $\mathcal {F}$ of codimension $n$ on
a smooth manifold $\mathcal M$ that allows to apply to $M/\mathcal{F}$ the
same techniques as to smooth manifolds.
In \cite{Losik} characteristic classes for a
foliation as  elements of the cohomology of certain bundles over the
leaf space $M/\mathcal{F}$ are defined.  
It would be interesting also to develope intrinsic (i.e., purely coordinate independent) theory of a smooth manifold 
foliation cohomology involving vertex algebra bundles \cite{BZF}.
Similar to Losik's theory, we use bundles correlation functions over a foliated space. 
The idea of studies of cohomology of certain bundles on a smooth manifold $\mathcal M$ and making connection to 
a cohomology of $\mathcal M$ has first appeared in \cite{BS}. 
 This can have a relation with Losik's work \cite{Losik} proposing a new framework for singular spaces and 
 characteristic classes. 
In applications, one would be interested in applying techniques of this paper to case of higher-dimensional 
manifolds of codimension one \cite{BG, BGG}. In particular, the question of higher non-vanishing invariants, as 
well as the problem of distinguishing of compact and non-compact leaves for the Reeb foliation of the full torus, 
are also of high importance. 
It would be important to establish connection to chiral de-Rham complex on a smooth manifold introduced in \cite{MSV}. 
After a modification, one is able to introduce a vertex algebra cohomology of smooth manifolds on a similar basis as in 
this paper.  
One can mention a possibility to derive differential equations \cite{Huang0} 
for characters on 
separate leaves of foliation. Such equations are derived for various genuses and can be used in frames 
of Vinogradov theory \cite{vinogradov}. 
The structure of foliation (in our sense) can be also studied from the automorphic function theory point of view. 
Since on separate leaves one proves automorphic properties of characters, on can think about "global" automorphic 
properties for the whole foliation. 
%%%%%

%%
%%%%%%%%%%%%%%%%%%%%%%%%%%%%%%%%%%%%%%%%%%%%%%%%%%%%%%%%%%%%%%%%%%%%%%%%%%%%%%%%%%%%%%%%%%%%%%%%%%%%%% 
The plan of the paper is the following. 
Section \ref{transversal} contains a description of the transversal basis  
and vertex algebra interpretation of the local geometry for a foliation on a smooth manifold. 
In Section \ref{product} we define a product of elements of two $\W_{z_1, \ldots, z_n}$-spaces. 
Spaces for double chain-cochain complexes associated to a vertex algebra are introduced in 
Section \ref{spaces}.
Product of double complex spaces is defined in Section \ref{productc}.
 In Section \ref{coboundary} coboundary operators are defined.
 It is shown that combining with the double complex 
spaces they determine chain-cochain double complex. 
The vertex algebra cohomology of a foliation on a smooth complex curve is introduced. 
A relation to Crainic-Moerdijk cohomology found. 
Properties of the product defined on double complex spaces are studied in Section \ref{pyhva}. 
Section \ref{gv} describes product-type cohomological invariants for a codimension one foliation 
on a smooth complex curve.  
In Appendixes we provide the material required for the construction of the vertex algebra cohomology 
of foliations.
In Appendix \ref{grading} we recall the notion of a quasi-conformal grading-restricted vertex algebra 
and its modules.
 In Appendix \ref{valued} the space of $\W$-valued rational sections of a vertex algebra bundle is described.  
In Appendix \ref{properties} properties of matrix elements for elements of the space $\W$ are listed. 
In Appendix  \ref{composable} maps composable with a number of vertex operators are defined. 
In Appendix \ref{connections} we describe the approach to cohomology in terms of connections. 
In Appendix \ref{sphere} geometrical procedure of sewing two Riemann spheres to form another 
Riemann sphere is recalled. 
Appendix \ref{rasto} contains proofs of Proposition \ref{katas}, Proposition \ref{pupa}  
Lemma \ref{functionformpropcor}, Lemma \ref{tarusa}. 
Appendix \ref{proof} contains proofs of Lemmas \ref{empty}, \ref{nezu}, \ref{subset}, and Proposition \ref{nezc}. 

A consideration of foliations of smooth manifold of arbitrary dimension will be given in \cite{preZ}.

%%%%%%%%%%%%%%%%%%%%%%%%%%%%%%%%%%%%%%%%%%%%%%%%%%%%%%%%%%%%%%%%%%%%%%%%%%%%%%%%%%%%%%%%%%%%%%%%%%%%%%%%
%%%%%%%%%%%%%%%%%%%%%%%%%%%%%%%%%%%%%%%%%%%%%%%%%%%%%%%%%%%%%%%%%%%%%%%%%%%%%%%%%%%%%%%%%%%%%%%%%%%%%%%%
%%
\section{Transversal basis description of foliations and vertex algebra interpretation}
%%
%
%%%%%%%%%%%%%%%%%%%%%%%%%%%%%%%%%%%%%%%%%%%%%%%%%%%%%%%%%%%%%%%%%%%%%%%%%%%%%%%%%%%%%%%%%%%%%%%%%%%%%%%%
%%
\label{transversal}
%%%
%%%%%%%%%%%%%%%%%%%%%%%%%%%%%%%%%%%%%%%%%%%%%%%%%%%%%%%%%%%%%%%%%%%%%%%%%%%%%%%%%%%%%%%%%%%%%
%%%%%%%%%%%%%%%%%%%%%%%%%%%%%%%%%%%%%%%%%%%%%%%%%%%%%%%%%%%%%%%%%%%%%%%%%%%%%%%%%%%%%%%%%%%%%
%%
In this Section we recall \cite{CM} the notion of the basis of transversal sections for a foliation, and 
provide its vertex algebra setup. 
%%
%%%%%%%%%%%%%%%%%%%%%%%%%%%%%%%%%%%%%%%%%%%%%%%%%%%%%%%%%%%%%%%%%%%%%%%%%%%%%%%%%%%%%%%%%%%%%%%%%%%%%%%%%%%%%
%%%%%%%%%%%%%%%%%%%%%%%%%%%%%%%%%%%%%%%%%%%%%%%%%%%%%%%%%%%%%%%%%%%%%%%%%%%%%%%%%%%%%%%%%%%%%%%%%%%%%%%%%%%%%
\subsection{Basis of transversal sections for a foliation} 
\label{basis}
This subsection reminds \cite{CM} the notion of basis of transversal sections for a codimension one foliation. 
 Let $\mathcal M$ be a complex curve %manifold of dimension $n$,
 equipped with a foliation $\F$ of codimension one. % $q$.
\begin{definition}
A transversal section of $\F$ is an embedded %$q$
 one-dimensional 
 submanifold $U\subset M$ which  
is everywhere transverse to the leaves of foliation.  
\end{definition}
\begin{definition}
If $\alpha$ is a path between two points $p_1$ and $p_2$ on the same leaf, 
and  
$U_1$ and $U_2$ are transversal sections through $p_1$ and $p_2$, 
 then $\alpha$ defines a transport
 along the leaves from a neighborhood of $p_1$ in $U_1$ to a neighborhood of $p_2$ in $U_2$.  
i.e., 
%hence 
%%
a germ of a diffeomorphism 
\[
hol\; (\alpha): (U_1, p_1)\hookrightarrow  (U_2, p_2),
\] 
which is called the holonomy of the path $\alpha$. 
\end{definition}
Two homotopic paths always define the same holonomy. 
%%
%%%%%%%%%%%%%%%%%%%%%%%%%%%%%%%%%%%%%%%%%%%%%%%%%%%%%%%%%%%%%%%%%%%%%%%%%%%%%%%%%%
\begin{definition}
If the above transport along $\alpha$ is defined in all of $U_1$ and embeds $U_1$ into $U_2$, this embedding
\[
h: U_1\hookrightarrow U_2, 
\]
is called the {holonomy embedding}.  
\end{definition}
A composition of paths %also
 induces an 
operation of composition on %those 
holonomy embeddings. 
%%
%%%%%%%%%%%%%%%%%%%%%%%%%%%%%%%%%%%%%%%%%%%%%%%%%%%%%%%%%%%%%%%%%%%%%%%%%%%%%%%%%%%%%%%%%%%%%%%%%
%
Transversal sections $U$ through $p$ as above should be thought of 
as neighborhoods of the leaf through $p$ in
the leaf space.
%%
%%%%%%%%%%%%%%%%%%%%%%%%%%%%%%%%%%%%%%%%%%%%%%%%%%%%%%%%%%%%%%%%%%%%%%%%%%%%5
Then we have 
\begin{definition}
A {transversal basis} for $\mathcal M/\F$ as a 
family $\U$ of transversal sections $U\subset \mathcal M$ 
with the property that, if $U_p$ is any transversal section through a given point $p\in \mathcal M$, there exists a 
holonomy embedding 
\[
h: U\hookrightarrow U_p, 
\]
 with $U\in \U$ and $p\in h(U)$.
\end{definition}
%%
%%%%%%%%%%%%%%%%%%%%%%%%%%%%%%%%%%%%%%%%%%%%%%%%%%%%%%%%%%%%55
%
%%
A transversal section is a 
one-dimensional disk given by a chart of the foliation. 
Accordingly, we can construct 
a transversal basis $\U$ out of a basis $\widetilde{\U}$ of $\mathcal M$ by domains of foliation charts
\[
 \phi_{U}: \widetilde{U}\tilde{\hookrightarrow } \mathbb{R}%^{n-1}
\times U,
\] 
$\widetilde{U}\in \widetilde{\U}$,
 with $U=\mathbb{R}$. 
%%

%%%%%%%%%%%%%%%%%%%%%%%%%%%%%%%%%%%%%%%%%%%%%%%%%%%%%%%%%%%%%%%%%%%%%%%%%%%%%%%%%%%%%%%%%%%%%%%%%%%%%%%
%%%%%%%%%%%%%%%%%%%%%%%%%%%%%%%%%%%%%%%%%%%%%%%%%%%%%%%%%%%%%%%%%%%%%%%%%%%%%%%%%%%%%%%%%%%%%%%%%%%%5
\subsection{Vertex algebra interpretation of the transversal basis}
Let $\U$ be a family of transversal sections of $\F$.  
We consider a $(n,k)$-set of points, $n \ge 1$, $k \ge 1$,    
\begin{equation}
\label{points}
\left(p_1, \ldots, p_n; p'_1, \ldots, p'_k \right), 
\end{equation}
on a smooth complex curve 
$\mathcal M$. 
Let us denote the set of corresponding local coordinates for $n+k$ points on $\mathcal M$ as 
\[
\left(c_1(p_1), \ldots, c_n(p_n); c'_1(p'_a), \ldots, c'_k(p'_k) \right). 
\]
In what follows we consider points \eqref{points} as points on either the leaf space $\mathcal M/\F$ of $\F$,
 or on transversal sections $U_j$ of the transversal basis $\U$. 
Since $\mathcal M/\F$ is not in general a manifold, one has to be careful in considerations of chains of local coordinates 
along its leaves \cite{IZ, Losik}.
 For association of formal parameters of mappings with parameters of vertex operators 
taken at points of $\mathcal M/\F$ we will use in what follows
 either local coordinates on $\mathcal M$ or local  
coordinates on 
sections $U$ of a transversal basis $\U$ which are submanifolds of $\mathcal M$ of dimension equal to codimenion of 
foliation $\F$.  
In case of extremely singular foliations when it is not possible to use local coordinates of $\mathcal M$ in order to  
parameterize 
a point on $\mathcal M/\F$ we still be able to use local coordinates
 on transversal sections passing through this point on  
$\mathcal M/\F$.
In addition to that, note that the complexes considered below are constructed in such a way that one can always 
use coordinates on transversal sections only, avoiding any possible problems with localization of coordinates on leaves 
of $\mathcal M/\F$. 
%%

%%%%%%%%%%%%%%%%%%%%%%%%%%%%%%%%%%%%%%%%%%%%%%%%%%%%%%%%%%%%%%%%%%%%%%%%%%%%%%%%%%%%%%%%%%%5
For a $(n, k)$-set of a grading-restricted vertex algebra $V$ elements  
\begin{equation}
\label{vectors}
\left(v_1, \ldots, v_n; v'_1, \ldots, v'_k\right),   
\end{equation}
we consider linear maps 
\begin{equation} 
\label{maps}
\Phi: V^{\otimes n} \rightarrow \W_{z_1, \ldots, z_n},   %{c_1(p_1), \dots, c_{n}(p_n)},  
\end{equation}
 (see Appendix \ref{valued} for the definition of a 
$\W_{z_{1}, \dots, z_{n}}$ space), 
\begin{equation}
\label{elementw}
\Phi \left(dz_1^{\wt v_1} \otimes v_1, c_1(p_1); \ldots; dz_n^{\wt v_n} \otimes v_n,  c_n(p_1) \right),    
\end{equation}
where we identify, as it is usual in the theory of characters for vertex operator algebras on curves 
\cite{Y, Zhu, TUY, H2},   
 $n$ formal parameters   
$z_{1}, \ldots , z_{n}$ of $\W_{z_{1}, \ldots , z_{n}}$, 
with local coordinates $c_i(p_i)$ in vicinities of points $p_i$, $0 \le i \le n$,  
 on $\mathcal M$. 
%%
%%%%%%%%%%%%%%%%%%%%%%%%%%%%%%%%%%%%%%%%%%%%%%%%%%%%%%%%%%%%%%%%%%%%%%%%%%%%%%%%%%%%%%%%%%%%%%%%%%%%%%%%%%%%%%%%%
%%
%%
 Elements $\Phi\in \W_{c_1(p_1), \ldots, c_n(p_n)}$ can be seen as 
  coordinate-independent $\overline{W}$-valued rational sections 
 of a vertex algebra bundle \cite{BZF} generalization.
Note that, according to \cite{BZF}, 
 they can be treated as 
\[
\left({\rm Aut}\; \Oo^{(1)}\right)^{\times n}_{z_1, \ldots, z_n} 
={\rm Aut}_{z_1}\; \Oo^{(1)} \times \ldots \times {\rm Aut}_{z_n}\; \Oo^{(1)},
\]
-torsors  
%%
%(see Appendix \ref{her} for explanations) 
%%
 of product of groups of independent coordinate transformations.    
The construction of vertex algebra cohomology of foliations in terms of connections is parallel to ideas of \cite{BS}. 
 Such a construction will be explained elsewhere \cite{preZ}. 
%%
%%%%%%%%%%%%%%%%%%%%%%%%%%%%%%%%%%%%%%%%%%%%%%%%%%%%%%%%%%%%%%%%%%%%%%%%%%%%%%%%%%%%%%%%%%%%%%%%%%%%%%%%%%%%%%%%%%

%%%%%%%%%%%%%%%%%%%%%%%%%%%%%%%%%%%%%%%%%%%%%%%%%%%%%%%%%%%%%%%%%%%%%%%%%%%%%%%%%%%%%%%%%%%%%%%%%%%%%%%%%%%%%
%%
In what follows, according to definitions of Appendix \ref{valued}, 
when we write an element $\Phi$ of the space $\W_{z_1, \ldots, z_n}$, we actually have in mind 
  corresponding matrix element $\langle w', \Phi\rangle$ that 
  absolutely converges (on certain domain) to 
a rational form-valued function 
\begin{equation}
\label{def}
\langle w', \Phi\rangle = R(\langle w', \Phi\rangle).
\end{equation}
%%
%%%%%%%%%%%%%%%%%%%%%%%%%%%%%%%%%%%%%%%%%%%%%%%%%%%%%%%%%%%%%%%%%%%%%%%%%%%%
% 
%%
%
%%
In notations, we will keep tensor products of vertex algebra elements with wight-powers of $z$-differentials 
when it is inevitable only. 
%%

%% 
%%%%%%%%%%%%%%%%%%%%%%%%%%%%%%%%%%%%%%%%%%%%%%%%%%%%%%%%%%%%%%%%%%%%%%%%%%%%%%%%%%%%%%%%%%%%%5
%
In Appendix \ref{proof} we prove, that 
 for arbitrary $v_i$, $v'_j \in V$, $1 \le i \le n$, $1 \le j \le k$, 
 points $p'_i$ with local coordinates $c_i(p'_i)$ on transversal sections $U_i\in \U$ of $\mathcal M/\F$, 
an element \eqref{elementw} 
 as well as the vertex operator 
 \begin{equation}
\label{poper}
\omega_W\left(dc_1({p'_i})^{\wt(v'_i)} \otimes v'_{i},  c_1({p'_1})\right)
= Y\left(  d(c_1(p'_i))^{\wt(v'_i)} \otimes v'_{i},  c_1({p'_i})\right), 
\end{equation}
 are invariant with respect to the action of 
$\left({\rm Aut}\; \Oo^{(1)}\right)^{\times n}_{z_1, \ldots, z_n}$. 
In \eqref{poper} we mean usual vertex operator (as defined in Appendix \ref{grading}) not affecting the tensor product 
with corresponding differential. 
We assume that the maps \eqref{maps} are composable 
(according to Definition \eqref{composabilitydef} of Appendix \ref{composable}), %cf.\cite{Huang}),
 with $k$   
vertex operators $\omega_W(v'_i, c_i(p'_i))$, $1 \le i \le k$ 
with $k$ vertex algebra elements from \eqref{vectors}, and 
formal parameters associated with local coordinates on  
$k$ transversal sections of $\mathcal M/\F$, of $k$ points from the set \eqref{points}. 

The composability of a map $\Phi$ with a number of vertex operators
consists of two conditions on $\Phi$.
 The first requires %the
 existence of positive 
integers $N^n_m(v_i, v_j)$ depending on vertex algebra elements $v_i$ and $v_j$ only, while the second condition 
restricts orders of poles of corresponding sums \eqref{Inm} %psi-i1} 
 and \eqref{Jnm}. %psi-i2}. 
Taking into account these conditions, we will see that 
the construction of spaces \eqref{ourbicomplex} %does 
depends on the choice 
of vertex algebra elements \eqref{vectors}. 
%%

%%%%%%%%%%%%%%%%%%%%%%%%%%%%%%%%%%%%%%%%%%%%%%%%%%%%%%%%%%%%%%%%%%%%%%%%%%%%%%%%%%%%%%%%%%%%%%%%%%%%%%
%%%%%%%%%%%%%%%%%%%%%%%%%%%%%%%%%%%%%%%%%%%%%%%%%%%%%%%%%%%%%%%%%%%%%%%%%%%%%%%%%%%%%%%%%%%%%%%%%%%%%%
In Section \ref{spaces} we construct spaces for double complexes  
associated to a grading-restricted vertex algebra and 
 defined for
codimension one foliations on complex curves.  
In order to keep control on the construction, we consider a section $U_j$ of a transversal basis $\U$,
and mappings $\Phi$ that belong to the space $\W_{c(p_1), \ldots, c(p_n)}$, 
depending on points $p_1, \ldots, p_n$ of intersection of $U_j$ with leaves of $\mathcal M/\F$ of $\F$.
It is assumed that local coordinates $c(p_1), \ldots, c(p_n)$ of points $p_i$, $1 \le i \le n$, 
are taken on $\mathcal M$ along these leaves of $\mathcal M/\F$. 
We then  consider a collection of $k$ sections of $\U$.   
In order to define vertex algebra cohomology of $\mathcal M/\F$, 
mappings $\Phi$ are supposed to be composable with a number of vertex operators 
 with a number of vertex algebra elements, and 
formal parameteres identified with local coordinates of points $p'_1, \ldots, p'_k$ on each of 
$k$ transversal sections $U_j$, $1 \le j \le k$.    
%%

%%%%%%%%%%%%%%%%%%%%%%%%%%%%%%%%%%%%%%%%%%%%%%%%%%%%%%%%%%%%%%%%%%%%%%%%%%%%%%%%%%%%%%%%%%%%
%%%%%%%%%%%%%%%%%%%%%%%%%%%%%%%%%%%%%%%%%%%%%%%%%%%%%%%%%%%%%%%%%%%%%%%%%%%%%%%%%%%%%%%%%%%%
\section{Product of $\W_{z_1, \ldots, z_n} 
$-spaces}
\label{product}
%%%%%%%%%%%%%%%%%%%%%%%%%%%%%%%%%%%%%%%%%%%%%%%%%%%%%%%%%%%%%%%%%%%%%%%%%%%%%%%%%%
%%%%%%%%%%%%%%%%%%%%%%%%%%%%%%%%%%%%%%%%%%%%%%%%%%%%%%%%%%%%%%%%%%%%%%%%%%%%%%%%%%
\subsection{Geometrical interpretation and definition of the $\epsilon$-
product for $\W_{z_1, \ldots, z_n}$-valued rational forms}
Recall Definition \ref{wspace} of $\W_{z_1, \ldots, z_n}$-spaces (Appendix \ref{valued}).  
The structure of $\W_{z_1, \ldots, z_n}$-spaces %\overline{W}$ 
is quite complicated and it is a problem to introduce a product of elements of such spaces algebraically. 
In order to define an appropriate product of two $\W_{z_1, \ldots, z_n}$-spaces we first have to interpret 
it geometrically. 
 $\W_{z_1, \ldots, z_n}$-spaces must be associated with certain model spaces. 
Then a geometric product of such model spaces should be defined, and, finally,
an algebraic product of $\W$-spaces should be introduces. 
For two  
 $\W_{x_1, \ldots, x_k}$- and 
$\W_{y_{1}, \ldots,  y_{n}}$-spaces we first associate formal complex parameters  
in  sets 
$(x_1, \ldots, x_k)$ and $(y_{1}, \ldots, y_n)$ 
to parameters of two auxiliary %auxilirary  
spaces. 
 Then we describe a geometric procedure to 
form resulting model space %$\mathcal X$ 
by combining two original model spaces.  %$\mathcal X_a$. 
Formal parameters of algebraic product 
$\W_{z_1, \ldots, z_{k+n}}$ of $\W_{x_1, \ldots, x_k}$  and  
$\W_{y_{1}, \ldots,  y_{n}}$ should be then identified with  
parameters of resulting auxiliary space.  %$\mathcal X$.
%%%% 
%%%%%%%%%%%%%%%%%%%%%%%%%%%%%%%%%%%%%%%%%%%%%%%%%%%%%%%%%%%%%%%%%%%%%%%%%%%%%%%%%%%%%%%%%%%%%%%%%%%%%%%%%%%%
Note that  
according to our assumption, $(x_1, \ldots, x_k) \in F_k\C$, and $(y_{1}, \ldots, y_{n}) \in F_{n}\C$, i.e., 
belong to corresponding configuration space (Definition \ref{confug}, Appendix \ref{valued}).  
 As it follows from the definition of $F_n\C$, 
 coincidence of two 
formal parameters is excluded from $F_{k+n}\C$. 
%%%
%%%%%%%%%%%%%%%%%%%%%%%%%%%%%%%%%%%%%%%%%%%%%%%%%%%%%%%%%%%%%%%%%%%%%%%%%%%%%%%%%%%%%%%
In general, it might happens that some $r$ formal parameters of $(x_1, \ldots, x_k)$ 
coincide with formal parameters of 
$(y_{1}, \ldots, y_{n})$, i.e., $x_{i_l}=y_{j_l}$, $1\le i_l, j_l \le r$. 
%%  

%%%%%%%%%%%%%%%%%%%%%%%%%%%%%%%%%%%%%%%%%%%%%%%%%%%%%%%%%%%%%%%%%%%%%%%%%%%%%%%%%%%%%%%%%%%%%%%%%%%%%5
%
%%
 In Definition \ref{wprodu} %(cf. Subsection \ref{genus_two_n-point}) 
of the product of $\W_{x_1, \ldots, x_k}$ and $\W_{y_{1}, \ldots, y_{n}}$  below  
we leave only one of two coinciding formal parameters,  
 i.e., we exclude one formal parameter from each coinciding pair. 
We require that the set of formal parameters
\begin{eqnarray}
\label{zsto}
(z_1, \ldots, z_{k+n-r}) =(x_1,  \ldots, {x}_{i_1}, \ldots, {x}_{i_r}, \ldots , x_k; y_1,  
 \ldots, \widehat {y}_{j_1},  \ldots, \widehat{y}_{j_r}, \ldots, y_n )
\end{eqnarray} 
would belong to $F_{k+n-r}\C$  
where $\; \widehat{.} \; $ denotes the exclusion of corresponding formal parameter for 
 $x_{i_l}=y_{j_l}$, $1 \le l \le r$.  
We denote this operation of formal parameters exclusion by 
$\widehat{R}\;\Phi(x_1, \ldots, x_k; y_{1}, \ldots, y_{n}; \epsilon)$. 
%%
%%%%%%%%%%%%%%%%%%%%%%%%%%%%%%%%%%%%%%%%%%%%%%%%%%%%%%%%%%%%%%%%%%%%%%%%%55
Thus,  
we require that the set of formal parameters $(z_1, \ldots, z_{k+n-r})$ 
for the resulting product %model space %$\mathcal X$ 
 would belong to $F_{k+n-r}\C$. %, where 
Note that instead of exclusion given by the right hand side of \eqref{zsto}, 
we could equivalently omit elements from $(x_1, \ldots, x_k)$ coinsiding with some elements 
of $(y_1, \ldots, y_n)$. 
%%

%%%%%%%%%%%%%%%%%%%%%%%%%%%%%%%%%%%%%%%%%%%%%%%%%%%%%%%%%%%%%%%%%%%%%%%%%%%%%%%%%%%%%%%%%%%%%%%%%%%
%% 
%% 
%%%%%%%%%%%%%%%%%%%%%%%%%%%%%%%%%%%%%%%%%%%%%%%%%%%%%%%%%%%%%%%%%%%%%%%%%%%%%%%%%%%%%%%%%%%
%%%%%%%%%%%%%%%%%%%%%%%%%%%%%%%%%%%%%%%%%%%%%%%%%%%%%%%%%%%%%%%%%%%%%%%%%%%%%%%%%%%%%%%%%%%
Recall the notion of intertwining operator \eqref{wprop} given in Appendix \ref{grading}. 
Let us first give a formal algebraic definition of a product of $\W$-spaces. 
%% 
%%
%%%%%%%%%%%%%%%%%%%%%%%%%%%%%%%%%%%%%%%%%%%%%%%%%%%%%%%%%%%%%%%%%%%%%%%%%%%%%%%%%%%%%%%%%%%%%%%%%%%%%%
%%%%%%%%%%%%%%%%%%%%%%%%%%%%%%%%%%%%%%%%%%%%%%%%%%%%%%%%%%%%%%%%%%%%%%%%%%%%%%%%%%%%%%%%%%%%%%%%%%%%%%
\begin{definition}
\label{wprodu}
For 
$\Phi(v_1, x_1; \ldots; v_{k}, x_k)  \in  \W_{x_1, \ldots, x_k}$, and 
$\Psi(v'_{1}, y_{1}; \ldots; v'_{n}, y_{n})  \in   \W_{y_1, \ldots, y_n}$ %$n'=n-k-1$, 
the $\epsilon$-product 
\begin{eqnarray}
\label{gendef} 
&& \Phi(v_1, x_1; \ldots; v_{k}, x_k) \cdot_\epsilon \Psi(v'_{1}, y_{1}; \ldots; v'_{n}, y_{n})  
\nn
&&
\qquad \qquad \mapsto 
\widehat{R} \; \Theta\left( v_1, x_1; \ldots; v_{k}, x_k; v'_{1}, y_{1}; \ldots; v'_{n}, y_{n}; \epsilon\right), 
%%
%%\nn
%&&
%%
\end{eqnarray}
 is defined by the form via \eqref{def} 
\begin{eqnarray}
\label{Z2n_pt_eps1q1}  
&& \langle w',  \widehat{R} \; \Theta (v_1, x_1; \ldots; v_{k}, x_k; v'_{1}, y_{1}; \ldots; v'_{n}, y_{n}; \epsilon) \rangle 
\nn 
& & \quad  =  
\sum_{l \in \Z }  \epsilon^l  
 \sum_{u\in V_l} 
 \langle w', Y^{W}_{WV}\left(  
\Phi (v_{1}, x_{1};  \ldots; v_{k}, x_{k}), \zeta_1\right)\; u \rangle  
\nn
& &
\quad    
 \langle w', Y^{W}_{WV}\left( 
\Psi(v'_{1}, y_{1}; \ldots; {v'}_{i_1}, \widehat{y}_{i_1}; 
\ldots; \ldots; {v'}_{j_r}, \widehat{y}_{j_r}; \ldots;  v'_{n}, y_{n})
 , \zeta_{2}\right) \; \overline{u} \rangle,      
\end{eqnarray}
parametrized by  
$\zeta_1$, $\zeta_2  \in \C$, where all monomials $(x_{i_l} - y_{j_l})$, 
are exluded 
for coinciding $x_{i_l} =y_{j_l}$, $1 \le l \le r$,  
 from 
\eqref{Z2n_pt_eps1q1}. 
%% 
%%
%%
%%%%%%%%%%%%%%%%%%%%%%%%%%%%%%%%%%%%%%%%%%%%%%%%%%%%%%%
%%
The sum 
is taken over any $V_{l}$-basis $\{u\}$,  
where $\overline{u}$ is the dual of $u$ with respect to a non-degenerate bilinear form %the Li--Z metric
 $\langle .\ , . \rangle_\lambda$, \eqref{eq: inv bil form} over $V$, (see Appendix \ref{grading}). 
\end{definition}
%%
%%
%%%%%%%%%%%%%%%%%%%%%%%%%%%%%%%%%%%%%%%%%%%%%%%%%%%%%%%%%%%%%%%%%%%%%%%%%%%%%%%%%%%%%%%%%
%%
By the standard reasoning \cite{FHL, Zhu}, %it is clear that %Definition 
 \eqref{Z2n_pt_eps1q1} does not depend on the choice of a basis of $u \in V_l$, $l \in \Z$.   
%%
%
%%
%%
%%%%%%%%%%%%%%%%%%%%%%%%%%%%%%%%%%%%%%%%%%%%%%%%%%%%%%%%%%%%%%%%%%%%%%%%%%%%%%%%%%%%%%%%%%%%%%%%%%%%%%%%
%%
In the case when multiplied forms $\Phi$ and $\Psi$ do not contain $V$-elements, i.e., for $\Phi$, $\Psi \in \W$, 
  \eqref{Z2n_pt_eps1q1} defines the product $\Phi \cdot_\epsilon \Psi$
\begin{eqnarray}
\label{Z2_part}
\langle w', \Theta(\epsilon)\rangle
= \sum_{l \in \Z }  \epsilon^l 
 \sum_{u\in V_l } 
\langle w', Y^{W}_{WV}\left(  
\Phi, \zeta_1\right) \; u \rangle 
\langle w', Y^{W}_{WV}\left(
\Psi,   
\zeta_2 \right) \; \overline{u} \rangle.   
\end{eqnarray}  
As we will see,  
 Definition \ref{wprodu} is also supported by 
Lemma \ref{functionformpropcor}.  
%
%%
%%%%%%%%%%%%%%%%%%%%%%%%%%%%%%%%%%%%%%%%%%%%%%%%%%%%%%%%%%%%%%%%%%%%%%%%%%%%%%%%%%%%%%%%%%%%%%%%%%%%%%%%%%%
\begin{remark}
Note that due to \eqref{wprop}, in 
Definition \ref{wprodu},  %and in \eqref{Z2n_pt_eps1q1} in particular, 
it is assumed that 
$\Phi (v_{1}, x_{1};  \ldots $; $ v_{k}, x_{k})$ and $\Psi(v'_{1}, y_1; \ldots $; $ v'_{n}, y_{n})$
 are composable with the $V$-module $W$ vertex operators 
$Y_W(u, -\zeta_1)$ and $Y_W(\overline{u}, -\zeta_2)$ correspondingly 
 (cf. Appendix \ref{composable} for the definition 
of composability).   
The product \eqref{Z2n_pt_eps1q1} is actually defined by sum of 
products of matrix elements of ordinary $V$-module $W$ vertex operators 
acting on $\W_{z_1, \ldots, z_n}$ elements.
%%
%%%%%%%%%%%%%%%%%%%%%%%%%%%%%%%%%%%%%%%%%%%%%%%%%%%%%%%%%%%%%%%%%%%%%%%%%%%%%%
% 
%%
Vertex algebra elements 
$u \in V$ and $\overline{u} \in V'$ 
are connected  by \eqref{dubay}, and $\zeta_1$ and $\zeta_2$ appear in a relation to each other. 
%%
%%
%%%%%%%%%%%%%%%%%%%%%%%%%%%%%%%%%%%%%%%%%%%%%%%%%%%%%%%%%%%%%%%%%%%%%%%%%%%%%%%%%%%%%%%%%%%%%%%%%%%%%%%%%%%%%%%%%%%
%%
The form 
of the product defined above is natural in terms of the theory of chacaters for vertex operator algebras 
\cite{TUY, FMS, Zhu}.   
\end{remark}
%%%%%%%%%%%%%%%%%%%%%%%%%%%%%%%%%%%%%%%%%%%%%%%%%%%%%%%%%%%%%%%%%%%%%%%%%%%%%%%%%%%%%%%%%%%%%%%%%%%%%%%%%%
%

%%%%%%%%%%%%%%%%%%%%%%%%%%%%%%%%%%%%%%%%%%%%%%%%%%%%%%%%%%%%%%%%%%%%%%%%%%%%%%%%%%%%%%%%%%%%%%%%%%%%%%%%%%%
%%%%%%%%%%%%%%%%%%%%%%%%%%%%%%%%%%%%%%%%%%%%%%%%%%%%%%%%%%%%%%%%%%%%%%%%%%%%%%%%%%%%%%%%%%%%%%%%%%%%%%%%%%%
\subsection{Convergence 
of the $\epsilon$-product}
%%
%% 
%%%%%%%%%%%%%%%%%%%%%%%%%%%%%%%%%%%%%%%%%%%%%%%%%%%%%%%%%%%%%%%%%%%%%%%%%%%%%%%%%%%
%%%%%%%%%%%%%%%%%%%%%%%%%%%%%%%%%%%%%%%%%%%%%%%%%%%%%%%%%%%%%%%%%%%%%%%%%%%%%%%%%%%
In order to prove convergence of the product \eqref{Z2n_pt_eps1q1} of elements of two spaces $\W_{x_1, \ldots, x_k}$ 
and $\W_{y_1, \ldots, y_n}$  of rational $\W$-valued forms,   
we have to use a geometrical interpretation \cite{H2, Y}.
Recall that a $\W_{z_1, \ldots, z_n}$-space is defined by means of matrix elements of the form \eqref{def}. 
For a vertex algebra $V$, this corresponds \cite{FHL}  to matrix element of a number of $V$-vertex operators
with formal parameters identified with local coordinates on a Riemann sphere. 
%%
%%%%%%%%%%%%%%%%%%%%%%%%%%%%%%%%%%%%%%%%%%%%%%%%%%%%%%%%%%%%%%%%%%%%%%%%%%%%%%%
% 
%
Geometrically, each space $\W_{z_1, \ldots, z_n}$ can be also associated to a Riemann sphere
 with a few marked points, %punctures 
and local coordinates 
 vanishing at these points %punctures 
\cite{H2}. 
An additional point %nother  of those points 
is identified to annulus center used in order 
to sew the sphere with another sphere.    
The product \eqref{Z2n_pt_eps1q1} has then a geometric interpretation.    
The resulting model space is then a Riemann sphere formed as a result of sewing procedure.  
In Appendix \ref{sphere} we describe explicitly the geometrical procedure of sewing of two spheres \cite{Y}.  
%% 

%%%%%%%%%%%%%%%%%%%%%%%%%%%%%%%%%%%%%%%%%%%%%%%%%%%%%%%%%%%%%%%%%%%%%%%%%%%%%%%%%%%%%%%%%%%%%%%%%%%%%%%%
%% sovpadayuschie tochki 
%%
Let us identify (as in \cite{H2, Y, Zhu, TUY, FMS, BZF}) two sets $(x_1, \ldots, x_k)$ and $(y_1, \ldots, y_n)$ of 
complex formal parameters,
with local
coordinates of two sets of points on the first and the second Riemann spheres correspondingly.   
Complex parameters $\zeta_1$ and $\zeta_2$ of \eqref{Z2n_pt_eps1q1} play then the roles of 
 coordinates \eqref{disk} of 
the annuluses \eqref{zhopki}. 
On identification of annuluses $\mathcal A_a$ and $\mathcal A_{\overline{a}}$,  
 $r$ coinciding coordinates may occur. %This takes into account case of coinciding formal parameters.
%%

%%%%%%%%%%%%%%%%%%%%%%%%%%%%%%%%%%%%%%%%%%%%%%%%%%%%%%%%%%%%%%%%%%%%%%%%%%%%%%%%%%
%%%%%%%%%%%%%%%%%%%%%%%%%%%%%%%%%%%%%%%%%%%%%%%%%%%%%%%%%%%%%%%%%%%%%%%%%%%%%%%%%%
%%
The product \eqref{Z2n_pt_eps1q1}
describes a $\W$-valued rational differential form defined 
on a sphere formed %obtained 
as a result of geometrical sewing \cite{Y} of two initial spheres. 
Since two initial spaces $\W_{x_1, \ldots, x_k}$ and $\W_{y_1, \ldots, y_n}$ 
are defined through  rational-valued forms 
expressed by 
 matrix elements of the form \eqref{def},  
it is then proved (see Proposition \ref{derga} below), that the resulting product defines a 
$\W_{
x_1, \ldots, x_k; y_1, \ldots, y_n
}$-valued rational 
form by means of an absolute convergent matrix element on the resulting sphere. 
The complex sewing parameter parametrizes the module space of sewin spheres as well as 
the product of 
$\W$-spaces.  
%
%
%%%%%%%%%%%%%%%%%%%%%%%%%%%%%%%%%%%%%%%%%%%%%%%%%%%%%%%%%%%%%%%%%%%%%%%%%%%%%%%%%%%%%%%
\begin{proposition} 
\label{derga} 
The product \eqref{Z2n_pt_eps1q1} of elements of the spaces $\W_{x_1, \ldots, x_k}$ and $\W_{y_{1}, \ldots, y_n}$ 
 corresponds to a rational form 
absolutely converging in $\epsilon$ %to a  %ordinary
with only possible poles at $x_i=x_j$, $y_{i'}=y_{j'}$, and    
 $x_i=y_j$, $1 \le i, i' \le k$, $1 \le j, j' \le n$.   %$i\ne j$.  
\end{proposition}
Proof of this proposition is given in \cite{LZ}. 
%%%%%%%%%%%%%%%%%%%%%%%%%%%%%%%%%%%%%%%%%%%%%%%%%%%%%%%%%%%%%%%%%%%% 

%% 
Next, we formulate 
%%
%%%%%%%%%%%%%%%%%%%%%%%%%%%%%%%%%%%%%%%%%%%%%%%%%%%%%%%%%%%%%%%%%%%%%%%%%%%%%%%%%%%%%
%%%%%%%%%%%%%%%%%%%%%%%%%%%%%%%%%%%%%%%%%%%%%%%%%%%%%%%%%%%%%%%%%%%%%%%%%%%%%%%%%%%%%
\begin{definition}
\label{sprod}
We define the
action of an element $\sigma \in S_{k+n-r}$ on the product of 
$\Phi  (v_{1}, x_{1};  \ldots; v_{k}, x_{k}) \in \W_{x_1, \ldots, x_k}$, and 
$\Psi  (v'_{1}, y_{1}; \ldots; v'_{n}, y_{n}) \in \W_{y_1, \ldots, y_n}$, as
%%
%
%%
%%%%%%%%%%%%%%%%%%%%%%%%%%%%%%%%%%%%%%%%%%%%%%%%%%%%%%%%%%%%%%%%%%%%%%%%%%%%%%%%%%%%%%%%%%%%%%%%%%%%%%%
\begin{eqnarray}
\label{Z2n_pt_epsss}
&& \langle w',  \sigma(\widehat{R}\;\Theta) (v_{1}, x_{1};  \ldots; v_{k}, x_{k}; v'_{1}, y_{1}; \ldots; v'_{n}, y_{n}; 
\epsilon) \rangle 
\nn
&&
\qquad =\langle w',  \Theta (\widetilde{v}_{\sigma(1)}, z_{\sigma(1)};  \ldots; \widetilde{v}_{\sigma(k+n-r)}, 
 z_{\sigma(k+n-r)};  
\epsilon) \rangle 
%%%%%%%%%
\nn 
& & \qquad =  
 \sum_{l \in \Z }  \epsilon^l 
 \sum_{u\in V_l } 
 \langle w', Y^{W}_{WV}\left(  
\Phi (\widetilde{v}_{\sigma(1)}, z_{\sigma(1)};  \ldots; \widetilde{v}_{\sigma(k)}, z_{\sigma(k)}), \zeta_1\right)\; u \rangle  
\nn
& &
\qquad   
 \langle w', Y^{W}_{WV}\left( 
\Psi
(\widetilde{v}_{\sigma(k+1)}, z_{\sigma(k+1)}; \ldots; \widetilde{v}_{\sigma(k+n-r)}, z_{\sigma(k+n-r)}) , \zeta_{2}\right) \; \overline{u} \rangle, 
\end{eqnarray}
where by $(\widetilde{v}_{\sigma(1)}, \ldots, \widetilde{v}_{\sigma(k+n-r)})$ we denote a permutation of 
vertex algebra elements 
\begin{equation}
\label{notari}
(\widetilde{v}_{1}, \ldots, \widetilde{v}_{k+n-r})
=(v_1, \ldots; v_k; \ldots, \widehat{v}'_{j_1}, \ldots, \widehat{v}'_{j_r},  \ldots  ). 
\end{equation}
\end{definition}
%%
%%topa
%%%%%%%%%%%%%%%%%%%%%%%%%%%%%%%%%%%%%%%%%%%%%%%%%%%%%%%%%%%%%%%%%%%%%%%%%%%%%%%%%%%%%%%%
 Next, 
 we have 
%%
%%%%%%%%%%%%%%%%%%%%%%%%%%%%%%%%%%%%%%%%%%%%%%%%%%%%%%%%%%%%%%%%%%%%%%%%%%%%%%%%%%%%%%%%%
\begin{lemma}
\label{tarusa}
The product \eqref{Z2n_pt_eps1q1} satisfy 
 \eqref{shushu} for $\sigma \in S_{k+n-r}$, i.e., 
\begin{equation*}
%%
%% \label{shushu1} 
%%
\sum_{\sigma\in J_{k+n-r; s}^{-1}}(-1)^{|\sigma|}
%%
 %\left(
\widehat{R} \; \Theta\left(v_{\sigma(1)}, x_{\sigma(1)}; \ldots; v_{\sigma(k)}, x_{\sigma(k)};  
v'_{\sigma(1)},  y_{\sigma(1)};  \ldots; v'_{\sigma(n)},  y_{\sigma(n)}; \epsilon \right) 
=0. 
\end{equation*}
\end{lemma}
The proof of this lemma is contained in Appendix \ref{rasto}.
 
%%
%%%%%%%%%%%%%%%%%%%%%%%%%%%%%%%%%%%%%%%%%%%%%%%%%%%%%%%%%%%%%%%%%%%%%%%%%%%%%%%%%%%%%%%%%%%%%%%%%%%%%%%%%%
%%%%%%%%%%%%%%%%%%%%%%%%%%%%%%%%%%%%%%%%%%%%%%%%%%%%%%%%%%%%%%%%%%%%%%%%%%%%%%%%%%%%%%%%%%%%%%%%%%%%%%%%%%
Next we prove the existence of appropriate $\W_{z_1, \ldots, z_{k+n-r}}$% ; y_{1}, \ldots, y_n}$
-valued rational form 
corresponding to the absolute convergent rational form 
$\mathcal R(z_1, \ldots, z_{k+n-r})$ %; y_{1}, \ldots, y_n; \epsilon)$ 
defining the $\epsilon$-product of elements of 
the spaces 
$\W_{x_1, \ldots, x_k}$ and $\W_{y_{1}, \ldots, y_n}$. %-valued rational forms.  
% 
%%
%%%%%%%%%%%%%%%%%%%%%%%%%%%%%%%%%%%%%%%%%%%%%%%%%%%%%%%%%%%%%%%%%%%%%%%%%%%%%%%%%%%%%%%%%%%%%%%%%
\begin{lemma}
\label{baskal}
For all choices of elements of the spaces $\W_{x_1, \ldots, x_k}$ and $\W_{y_{1}, \ldots, y_n}$ %\eqref{her}, 
there exists an element 
${\widehat R }\;
\Theta(v_1, x_1; \ldots; v_{k}, x_{k};
 v'_1, y_1; \ldots; v'_{n}, y_{n} $; $ \epsilon) 
\in \W_{z_1, \ldots, z_{k+n-r}}$%; y_{1}, \ldots, y_n %z_{1}, \ldots, z_{k+n}
%%
%}$
-valued rational form  
such that the product \eqref{Z2n_pt_eps1q1} 
converges to 
\[
R(x_1, \ldots, x_{k}; y_1, \ldots, y_{n}; \epsilon) 
=\langle w', {\widehat R }\;\Theta (v_1, x_1; \ldots; v_{k}, x_{k}; v'_1, y_1; \ldots; v'_{n}, y_{n}; \epsilon)\rangle.
\] 
\end{lemma}
%% 
%%%%%%%%%%%%%%%%%%%%%%%%%%%%%%%%%%%%%%%%%%%%%%%%%%%%%%%%%%%%%%%%%%%%%%%%%%%%%%%%%5
%%
\begin{proof}
In the proof of Proposition \ref{derga} we showed the absolute
 convergence of the product \eqref{Z2n_pt_eps1q1} to a rational form %some 
$R(x_1, \ldots, x_{k}; y_1, \ldots, y_{n}; \epsilon)$. The lemma follows from completeness of 
$\overline{W}_{x_1, \ldots, x_k; y_{1}, \ldots, y_n}$ and density of the space of rational 
differential forms. 
\end{proof}
%%
%
%%%%%%%%%%%%%%%%%%%%%%%%%%%%%%%%%%%%%%%%%%%%%%%%%%%%%%%%%%%%%%%%%%%%%%%%%%%%%%%%%%%%%%%%%%%%%%%%%%%%%%%%%%%%%%%%%
%%%%%%%%%%%%%%%%%%%%%%%%%%%%%%%%%%%%%%%%%%%%%%%%%%%%%%%%%%%%%%%%%%%%%%%%%%%%%%%%%%%%%%%%%%%%%%%%%%%%%%%%%%%%%%%
 We formulate
%%
%%%%%%%%%%%%%%%%%%%%%%%%%%%%%%%%%%%%%%%%%%%%%%%%%%%%%%%%%%%%%%%%%%%%%%%%%%%%%%%%%%%%%%%%%%%%%%%%%%%%%%
\begin{definition}
%%%%%%%%%%%%%%%%%%%%%%%%%%%%%%%%%%%%%%%%%%%%%%%%%%%%%%%%%%%%%%%%%%%%%%%%%%%%%%%%%%%%%%%%%%%%%%%%%%%%%%%
%%
For $\Phi(v_1, x_1; \ldots; v_k, x_k) \in \W_{x_1, \ldots, x_k}$ and 
$\Psi(v'_1, y_1; \ldots; v'_n, y_k) \in \W_{y_1, \ldots, y_n}$, with $r$ coinciding formal parameters 
$x_{i_q}=y_{j_q}$, $1 \le q \le r$,   
we define the action of $\partial_{s}=\partial_{z_{s}}={\partial}/{\partial_{z_{s}}}$, $1\le s \le k+n-r$,
$1 \le i \le k$,  $1 \le j \le n$
the differentiation of
$\widehat {R} \; \Theta(v_{1},  x_{1};  \ldots;  v_k, x_{k}$; $  v'_1, y_1; \ldots;  v'_{n},  y_{n}; \epsilon)$ 
with respect to the $s$-th entry of 
 $(z_{1},  \ldots,   z_{k+n-r})$,  %;   y_1, \ldots,   y_{n})$
 as 
follows 
%%%%%%%%%%%%%%%%%%%%%%%%%%%%%%%%%%%%%%%%%%%%%%%%%%%%%%%%%%%%%%%%%%%%%%%%%%%%%%%%%%%%%%%%%%%%%%%%555
%%
\begin{eqnarray}
\label{Z2n_pt_eps1qdef}
 & &
\langle w', 
\partial_{
s} 
\widehat{R} \; \Theta 
 ( v_{1},  x_{1};  \ldots;  v_k, x_{k};   v'_1, y_1; \ldots;  v'_{n},  y_{n}; \epsilon)
 \rangle 
%%
%
%%%%%%%%%%%%%%%%%%%%%%%%%%%%%%%%%%%%%%%%%%%%%%%%%%%%%%%%%%%%%%%%%%%%%%%%%%%%%%%%%%%%5
\nn
 & &  = \sum_{l\in \mathbb{Z}
} \epsilon^{l} \sum_{u\in V_l}   
\langle w',   \partial^{\delta_{s,i}}_{x_i}   
Y^{W}_{WV}\left( 
   \Phi ( v_{1},  x_{1};  \ldots; v_{k}, x_{k}), \zeta_1 \right) \; u \rangle  
%%
%\notag 
%%
\nn
& &
\qquad   \qquad %\cdot 
\langle w',   \partial^{\delta_{s,j} - \delta_{i_q, j_q}}_{y_j} 
Y^{W}_{WV}\left( 
\Psi
(v'_{1},  y_{1}; \ldots;  v'_{n},  y_{n}),  \zeta_2 \right) \overline{u} \rangle.     
\end{eqnarray}
\end{definition}
%%
%%%%%%%%%%%%%%%%%%%%%%%%%%%%%%%%%%%%%%%%%%%%%%%%%%%%%%%%%%%%%%%%%%%%%%%%%%%%%%%%%%%%%%%%%%%%%%%%%%%%%%%%%%%%
%%
\begin{remark}
As we see in the last expressions, 
%%
%\eqref{her} 
%%
  the $L_V(0)$-conjugation property \eqref{loconj} for the product \eqref{Z2n_pt_eps1q1} includes the action 
of $z^{L_V(0)}$-operator on complex parameters $\zeta_a$, $a=1$, $2$.  
\end{remark}
%%
%%%%%%%%%%%%%%%%%%%%%%%%%%%%%%%%%%%%%%%%%%%%%%%%%%%%%%%%%%%%%%%%%%%%%%%%%%%%%%%%%%%%%%%%%%%55
%%
\begin{proposition}
\label{katas}
The product \eqref{Z2n_pt_eps1q1} satisfies the 
$L_V(-1)$-derivative \eqref{lder1} and $L_V(0)$-conjugation \eqref{loconj} properties. 
\end{proposition}
Proof of this propositions in Appendix \ref{rasto}. 

%%%%%%%%%%%%%%%%%%%%%%%%%%%%%%%%%%%%%%%%%%%%%%%%%%%%%%%%%%%%%%%%%%%%%%%%%%%%%%%%%%%%%%%%%%%%%%
Summing up the results of Proposition \eqref{derga},  
Lemma \eqref{tarusa}, Lemma \eqref{baskal}, and Proposition \eqref{katas},
 we obtain the main statement of this section:  
%%
%%%%%%%%%%%%%%%%%%%%%%%%%%%%%%%%%%%%%%%%%%%%%%%%%%%%%%%%%%%%%%%%%%%%%%%%%%%%%%%%%%%%%%%%%%%%%%%
\begin{proposition}
\label{glavna}
The product \eqref{Z2n_pt_eps1q1} provides a map 
\[
\cdot_\epsilon :  \W_{x_1, \ldots, x_k} \times \W_{y_1, \ldots, y_n} 
\rightarrow \W_{z_1, \ldots, z_{k+n-r}}.   
\] 
\hfill $\square$
\end{proposition}
%%
%%%%%%%%%%%%%%%%%%%%%%%%%%%%%%%%%%%%%%%%%%%%%%%%%%%%%%%%%%%%%%%%%%%%%%%%%%%%%%%%%%%%
We then have 
%%
%%%%%%%%%%%%%%%%%%%%%%%%%%%%%%%%%%%%%%%%%%%%%%%%%%%%%%%%%%%%%%%%%%%%%%%%%%%%%%%%%%%
%%%%%%%%%%%%%%%%%%%%%%%%%%%%%%%%%%%%%%%%%%%%%%%%%%%%%%%%%%%%%%%%%%%%%%%%%%%%%%
\begin{definition}
For fixed sets $(v_1, \ldots, v_k)$, $(v'_1, \ldots, v'_n) \in V$,  
$(x_1, \ldots, x_k)\in \C$, $(y_1, \ldots, y_n$) $\in \C$,   
 we call the set of all $\W_{z_1, \ldots, z_{k+n-r}}$ 
-valued rational forms 
$\widehat{R} \; \Theta(
v_1, x_1; \ldots;  v_k, x_{k} $ ; $ v'_1, y_1; \ldots;  v'_{n},  y_{n}; 
\epsilon)$ defined  by \eqref{Z2n_pt_eps1q1} 
with the parameter $\epsilon$ exhausting all possible values,    
 the complete product of the spaces $\W_{x_1, \ldots, x_k}$ and  $\W_{y_{1}, \ldots, y_n}$.  
\end{definition}
%%
%%
%%%%%%%%%%%%%%%%%%%%%%%%%%%%%%%%%%%%%%%%%%%%%%%%%%%%%%%%%%%%%%%%%%%%%%%%%%%%%%%%%%%%%%%%%%%%%%%%%%%%%%%%%%%%
%%%%%%%%%%%%%%%%%%%%%%%%%%%%%%%%%%%%%%%%%%%%%%%%%%%%%%%%%%%%%%%%%%%%%%%%%%%%%%%%%%%%%%%%%%%%%%%%%%%%%%%%%%%%
\subsection{Properties of the $\W_{z_1, \ldots, z_n}$-product}
\label{properties}
In this subsection we study %the
 properties of 
the product 
$\widehat{R} \;\Theta(v_1, x_1$; $ \ldots $; $  v_k, x_{k} $; $ v'_1, y_1; \ldots $; $  v'_{n},  y_{n} $; $ \epsilon )$ of 
\eqref{Z2n_pt_eps1q1}. 
We  have 
%%
%%%%%%%%%%%%%%%%%%%%%%%%%%%%%%%%%%%%%%%%%%%%%%%%%%%%%%%%%%%%%%%%%%%%%%%%%%%%%%%%%%%%%%%%%%%%%%%%%%%%%%%%%%%%%
%%%%%%%%%%%%%%%%%%%%%%%%%%%%%%%%%%%%%%%%%%%%%%%%%%%%%%%%%%%%%%%%%%%%%%%%%%%%%%%%%%%%%%%%%%%%%%%%%%%%%%%%%%%%%
\begin{proposition}
\label{pupa}
For generic elements $v_i$, $v'_j \in V$, $1 \le i \le k$, 
$1 \le j \le n$,  of a quasi-conformal grading-restricted vertex algebra, 
the product \eqref{Z2n_pt_eps1q1} is canonincal with respect to the action of the group  
$\left( {\rm Aut} \; \Oo\right)^{\times (k+n-r)}_{z_1, \ldots, z_{k+n-r}}$   
of independent  
$k+n-r$-dimensional  %independent 
changes 
\begin{eqnarray}
\label{zwrho}
&&(z_1, \ldots, z_{k+n-r}) 
\mapsto 
(z'_1, \ldots, z'_{k+n-r})  
= 
(\rho(z_1), \ldots, \rho(z_{k+n-r})),  %  x_{k}; y_1, \ldots, y_{n}),
\end{eqnarray}
of formal parameters. 
\end{proposition}
Proof of this proposition is in Appenedix \ref{rasto}. 
%%

%%%%%%%%%%%%%%%%%%%%%%%%%%%%%%%%%%%%%%%%%%%%%%%%%%%%%%%%%%%%%%%%%%%%%%%%%%%%%%%%%%%%%%%%%%%%%%%%%%
%%
In the geometric interpretation in terms of auxiliary Riemann spheres,  
the definition \eqref{Z2n_pt_eps1q1} depends on the choice of insertion points $p_i$, $1 \le i \le k$,  
with local coordinated $x_i$ 
on $
\widehat{\Sigma}^{(0)}_{1}$,  
and $p'_i$, $1 \le j \le k$, with local coordinates $y_j$ on  
 $
\widehat{\Sigma}_{2}^{(0)}$. Suppose we change %the 
the distribution of points among two initial Riemann spheres. 
We then  
formulate the following  %vectors as follows: 
%%
%%%%%%%%%%%%%%%%%%%%%%%%%%%%%%%%%%%%%%%%%%%%%%%%%%%%%%%%%%%%%%%%%%%%%%%%%%%%%%%%%%%%%%%%%%%%%%%%%%%%
%%
\begin{lemma}
\label{functionformpropcor} 
In the setup above,  
for a fixed set $(\widetilde{v}_{1} , \ldots, \widetilde{v}_{k+n}) \in V$, %$1 \le l \le n$, 
of vertex algebra elements, 
 the $\epsilon$-product
$\Theta(\widetilde{v}_{1}, z_1;  \ldots;  \widetilde{v}_{k+n}, z_{k+n}; \epsilon) 
\in \W_{z_1, \ldots, z_{k+n}}$,  
  \begin{equation}
\label{invar}
\cdot_\epsilon : \W_{z_1, \ldots, z_s} \times  \W_{z_{s+1}, \ldots, z_{k+n}} \rightarrow  
 \W_{z_1, \ldots, z_{k+n}}, 
\end{equation}
 remains the same 
 for elements $\Phi(\widetilde{v}_{1}, z_1;  \ldots;  \widetilde{v}_{s}, z_s) \in 
  \W_{z_1, \ldots, z_s}$ and 
$\Psi(\widetilde{v}_{s+1}, z_{s+1} $; $  \ldots $ ; $ \widetilde{v}_{k+n}, z_{k+n}) \in \W_{z_{k+1}, \ldots, z_{k+n}}$,   
for any $0 \le s \le k+n$.  
\end{lemma}
Proof of this lemma is in Appendix \ref{rasto}. 
%%
%%%%%%%%%%%%%%%%%%%%%%%%%%%%%%%%%%%%%%%%%%%%%%%%%%%%%%%%
%%
\begin{remark}
This Lemma is important for the formulation of cohomological invariants associated to grading-restricted vertex algebras 
on smooth manifolds \cite{preZ}. 
In the case $s=0$, we obtain from \eqref{invar}, 
\begin{equation}
\label{invar}
\cdot_\epsilon: \W \times \W_{z_{1}, \ldots, z_{k+n}} \rightarrow  
 \W_{z_1, \ldots, z_{k+n}}.  
\end{equation}
\end{remark}
%%
%%%%%%%%%%%%%%%%%%%%%%%%%%%%%%%%%%%%%%%%%%%%%%%%%%%%%%%%%%%%%%%%%%%%%%%%%%%%%%%%%%%%%
%%%%%%%%%%%%%%%%%%%%%%%%%%%%%%%%%%%%%%%%%%%%%%%%%%%%%%%%%%%%%%%%%%%%%%%%%%%%%%%%%%%%% 

%%       
%%%%%%%%%%%%%%%%%%%%%%%%%%%%%%%%%%%%%%%%%%%%%%%%%%%%%%%%%%%%%%%%%%%%%%%%%%%%%%%%%
%%%%%%%%%%%%%%%%%%%%%%%%%%%%%%%%%%%%%%%%%%%%%%%%%%%%%%%%%%%%%%%%%%%%%%%%%%%%%%%%%
%%
\section{Spaces for double complexes}
\label{spaces}
In this section we introduce the definition of spaces for a double complex suitable for the construction 
a grading-restricted vertex algebra cohomology for codimension one %regular 
foliations on complex curves.  
We first introduce 
%%
%%%%%%%%%%%%%%%%%%%%%%%%%%%%%%%%%%%%%%%%%%%%%%%%%%%%%%%%%%%%%%%%%%%%%%5555
\begin{definition}
\label{initialspace}
 Let $(v_1, \ldots, v_n)$, $(v'_1, \ldots, v'_k)$ be two sets ot vertex algebra $V$ elements, and 
 $(p_1, \dots, p_n)$ be points with local coordinates $(c_1(p_1), \dots, c_n(p_n))$
 taken on the same transversal section $U_j \in \U$, $j\ge 1$ 
of the foliation $\F$ transversal basis $\U$ on a complex curve. 
Assuming $k \ge 1$, $n \ge 0$,  
we denote by $C^n(V, \W, \F)(U_j)$, $0 \le j \le k$,       
the space of all linear maps \eqref{maps}
\begin{equation}
\label{mapy}
 \Phi: V^{\otimes n } \rightarrow \W_{c_1(p_1), \dots, c_{n}(p_n)},  
\end{equation}
composable with a $k$ of vertex operators \eqref{poper} with formal parameters identified 
with local coordinates $c'_j(p'_j)$ functions around points $p'_j$ 
on each of transversal sections $U_j$, $1 \le j \le k$.  
\end{definition}
%%
%%%%%%%%%%%%%%%%%%%%%%%%%%%%%%%%%%%%%%%%%%%%%%%%%%%%%%%%%%%%%%%%%%%%%%%%%%%%%%%%%%%%%%%%%%%%%%%%%%
%%%%%%%%%%%%%%%%%%%%%%%%%%%%%%%%%%%%%%%%%%%%%%%%%%%%%%%%%%%%%%%%%%%%%%%%%%%%%%%%%%%%%%%%%%%%%%%%%%
%% 
%% 
The set of vertex algebra elements  \eqref{vectors} 
plays the role of parameters in our further construction of 
the vertex algebra cohomology associated with the foliation $F$.  
According to considerations of Subsection \ref{basis}, we assume that each transversal section of a transversal basis 
$\U$ possess a coordinate chart which is induced by a coordinate chart of $\mathcal M$ (cf. \cite{CM}).  
%%

%%%%%%%%%%%%%%%%%%%%%%%%%%%%%%%%%%%%%%%%%%%%%%%%%%%%%%%%%%%%%%%%%%%%%%%%%%%%%%%%%%%%%%%%%
%%%%%%%%%%%%%%%%%%%%%%%%%%%%%%%%%%%%%%%%%%%%%%%%%%%%%%%%%%%%%%%%%%%%%%%%%%%%
%%
Recall the notion of a holonomy embedding (cf. Subsection \ref{basis}, cf. \cite{CM}) which maps 
 a section into another section of a transversal basis, 
 and a coordinate chart on the first section 
into a coordinate chart on the second transversal section.  
%%
%%%%%%%%%%%%%%%%%%%%%%%%%%%%%%%%%%%%%%%%%%%%%%%%%%%%%%%%%%%%%%%%%%%%%%%%%%%%%%%%%%%%%%%%%%%
%%
%% 
Motivated by 
the definition of spaces for the ${\rm \check C}$ech-de Rham complex in \cite{CM} (see Subsection \ref{basis}),
 let us now introduce the following spaces: 
%%
%%%%%%%%%%%%%%%%%%%%%%%%%%%%%%%%%%%%%%%%%%%%%%%%%%%%%%%%%%%%%%%%%%%%%%%%%%%%%%%%%%%%%%%%%%%%%%%%%%%%%%
%
%% 
%%%%%%%%%%%%%%%%%%%%%%%%%%%%%%%%%%%%%%%%%%%%%%%%%%%%%%%%%%%%%%%%%%%%%%%%%%%%%%%%%%%%%%%%%%%%%%%%%%%%
%%
\begin{definition}
\label{defspace}
For $n\ge 0$, and $1 \le m \le k$,
 with Definition \ref{initialspace},  
 we define the space   
\begin{equation}
\label{ourbicomplex}
 {C}^{n}_{m}(V, \W, \U, \F) =  \bigcap_{ 
U_{1} \stackrel{h_1} {\hookrightarrow }  \ldots \stackrel {h_{m-1}}{\hookrightarrow } U_m  
\atop 1 \le j \le m }  
 C^{n}(V, \W, \F) (U_j),    
\end{equation}
where the intersection ranges over all possible $(m-1)$-tuples of holonomy embeddings $h_j$, $j\in \{1, \ldots, m-1\}$, 
between transversal sections of the baisis $\U$  for $\F$. 
%%
%%%%%%%%%%%%%%%%%%%%%%%%%%%%%%%%%%%%%%%%%%%%%%%%%%%%%%%%%%%%%%%%%%%%%%%%%%%%%%%%%%%%%%%%%%%%%%%%%
%% 
%%   
%%   
\end{definition}
%%
%%%%%%%%%%%%%%%%%%%%%%%%%%%%%%%%%%%%%%%%%%%%%%%%%%%%%%%%%%%%%%%%%%%%%%%%%%%%%%%%%%%%%%%%%%%%%%%%%%%%%%%%%
%%%%%%%%%%%%%%%%%%%%%%%%%%%%%%%%%%%%%%%%%%%%%%%%%%%%%%%%%%%%%%%%%%%%%%%%%%%%%%%%%%%%%%%%%%%%%%%%%%%%%%%%%
%
%%
First, we have the following
\begin{lemma}
\label{empty}
\eqref{ourbicomplex} is non-empty. 
\end{lemma}
%%
%
%%%%%%%%%%%%%%%%%%%%%%%%%%%%%%%%%%%%%%%%%%%%%%%%%%%%%%%%%%%%%%%%%%%%%%%%%%%%%%%%%%%%%%%%%%%
%%%%%%%%%%%%%%%%%%%%%%%%%%%%%%%%%%%%%%%%%%%%%%%%%%%%%%%%%%%%%%%%%%%%%%%%%%%%%%%%%%%%%%%%%%%
%%
%%
\begin{lemma}
\label{nezu}
The double complex \eqref{ourbicomplex} does not depend on the choice of transversal basis $\U$. 
\end{lemma}
Thus we then denote  ${C}^{n}_{m}(V, \W, \U, \F)$ as ${C}^{n}_{m}(V, \W, \F)$. 
%%
%%%%%%%%%%%%%%%%%%%%%%%%%%%%%%%%%%%%%%%%%%%%%%%%%%%%%%%%%%%%%%%%%%%%%%%%%%%%%%%%%%%%%%%%%%%%%%%%%%%%%%%%%%%%
Recall the notation of a quasi-conformal grading-restricted vertex algebra given in Appendix \ref{grading}. 
The main statement of this section is contained in the following
\begin{proposition} 
\label{nezc}
For a quasi-conformal grading-restricted vertex algebra $V$ and its module $W$,    
the construction \eqref{ourbicomplex} is canonical, i.e.,
 does not depend on foliation preserving choice of local coordinates on $\mathcal M/\F$.   
\end{proposition}
The proofs of Lemmas \ref{empty}, \ref{nezu}, and Proposition \ref{nezc} is contained in Appendix \ref{proof}. 
\begin{remark}
The condition of quasi-conformality is necessary in the proof of invariance of elements of the space 
$\W_{z_1, \ldots, z_n}$ 
with respect to a vertex algebraic representation (cf. Appendix \ref{grading}) of the group 
$\left({\rm Aut}\; \Oo^{(1)}\right)^{\times n}_{z_1, \ldots, z_n}$. 
In what follows, when it concerns the spaces \eqref{ourbicomplex} we will always assume the quasi-conformality of $V$. 
\end{remark}
%%

%%%%%%%%%%%%%%%%%%%%%%%%%%%%%%%%%%%%%%%%%%%%%%%%%%%%%%%%%%%%%%%%%%%%%%%%%%%%%%%%%%%%%%%%%%%%%%%%%%%%%%%%%%%%%%%%%%%%%%%%%%
%%  
Proofs of generalizations 
of Lemmas \ref{empty}, \ref{nezu}, \ref{subset} and Proposition \ref{nezc}
 for the case of an arbitrary codimension foliation on a smooth $n$-dimensional manifold 
will be given in \cite{preZ}.  
The proof of Proposition \ref{nezc} is contained in Appendix \ref{proof}. 

%%%%%%%%%%%%%%%%%%%%%%%%%%%%%%%%%%%%%%%%%%%%%%%%%%%%%%%%%%%%%%%%%%%%%%%%%%%%%%%%%%%%%%%%%
%%
Let $W$ be a grading-restricted $V$ module. 
Since for $n=0$, $\Phi$ does not include variables, and  
 due to Definition \ref{composabilitydef} of the composability, we can put: %(cf. \cite{Huang}):    
\begin{equation}
\label{proval}
{C}_{k}^{0}(V, \W, \F)= W,  
\end{equation}
for $k\ge 0$. Nevertheless, according to our Definition \ref{ourbicomplex}, mappings that belong to \eqref{proval}
are assumed to be composable with a number of vertex operators with depending on local coordinates 
of $k$ points on $k$ transversal sections.  

%%%
We observe 
\begin{lemma}
\label{subset}
\begin{equation}
\label{susus}
{C}_{m}^{n}(V, \W, \F)\subset {C}_{m-1}^{n}(V, \W, \F).     
\end{equation}
%%
%(when lower index is zero the sequence terminates).
%%
\end{lemma}
The proof of this Lemma is contained in Appendix \ref{proof}. 
%%

%%%%%%%%%%%%%%%%%%%%%%%%%%%%%%%%%%%%%%%%%%%%%%%%%%%%%%%%%%%%%%%%%%%%%%%%%%%%%%%%%%%%%%%%%%%%%%%%%%%%%
%%%%%%%%%%%%%%%%%%%%%%%%%%%%%%%%%%%%%%%%%%%%%%%%%%%%%%%%%%%%%%%%%%%%%%%%%%%%%%%%%%%%%%%%%%%%%%%%%%%%%
%%%%%%%%%%%%%%%%%%%%%%%%%%%%%%%%%%%%%%%%%%%%%%%%%%%%%%%%%%%%%%%%%%%%%%%%%%%%%%%%%%%%%%%%%%%%%%%%%%%%%
\section{The product of $C^n_m(V, \W, \F)$-spaces}
\label{productc}
In this section we consider an application of the material of Section \ref{product} to %the case of 
 double complex spaces $C^n_m(V, \W, \F)$ (Definition \ref{defspace}, Section \ref{spaces}) 
for a foliation $\F$ on a complex curve. 
We introduce the product of two double complex %\eqref{her} 
spaces 
with the image in another double complex space coherent with respect 
to the original differentials \eqref{deltaproduct} and \eqref{halfdelta}, and the symmetry property \eqref{shushu}.
 We prove the canonicity of the product, 
%%
%%%%%%%%%%%%%%%%%%%%%%%%%%%%%%%%%%%%%%%%%%%%%%%%%%%%%%%%%%%%%%%%%%%%%%%%%%%%%%%%%%%%%%%%%%%%%%%%%%%%%%%%
%%%%%%%%%%%%%%%%%%%%%%%%%%%%%%%%%%%%%%%%%%%%%%%%%%%%%%%%%%%%%%%%%%%%%%%%%%%%%%%%%%%%%%%%%%%%%%%%%%%%%%%%
 and derive an analogue of Leibniz formula. 
%%
%%%%%%%%%%%%%%%%%%%%%%%%%%%%%%%%%%%%%%%%%%%%%%%%%%%%%%%%%%%%%%%%%%%%%%%%%%%%%%%%%%%%%%%%%%%%%%%%%%%%%%%%
%%%%%%%%%%%%%%%%%%%%%%%%%%%%%%%%%%%%%%%%%%%%%%%%%%%%%%%%%%%%%%%%%%%%%%%%%%%%%%%%%%%%%%%%%%%%%%%%%%%%%%%%
\subsection{Geometrical adaptation of the $\epsilon$-product to a foliation}
In this subsection we show how the definition of the product of $\W_{z_1, \ldots, z_n}$-spaces 
can be extended to the case of $C^k_n(V, \W, \F)$-spaces for a codimension one foliation of a complex curve. 
Recall Definition \ref{ourbicomplex} of $C^k_n(V, \W, \F)$-spaces in Section \ref{spaces}. 
%%
%%%%%%%%%%%%%%%%%%%%%%%%%%%%%%%%%%%%%%%%%%%%%%%%%%%%%%%%%%%%%%%%%%%%%%%%%%%%%%%%%%%%%%%%%%%%%%%%%%%%%%%%%%
%%
We use again the geometrical scheme of sewing of two Riemann surfaces in order to introduce the product
of two elements $\Phi \in C^k_m(V, \W, \F)$  and $\Psi \in C^n_{m'}(V, \W, \F)$ which belong to 
two double complex spaces \eqref{ourbicomplex} for the same foliation $\F$. 
%%
%Taking into account Definition \ref{her} we obtain. 
%%
%%%%%%%%%%%%%%%%%%%%%%%%%%%%%%%%%%%%%%%%%%%%%%%%%%%%%%%%%%%%%%%%%%%%%%%%%%%%%%%%%%%%%%%%%%%%%%%%%%%%
%%
The construction is again local, thus we assume that both spaces
 $C^k_m(V, \W, \F)$  and $C^n_{m'}(V, \W, \F)$ are considered on the same fixed transversal basis $\U$. 
Moreover, we assume that marked points used in Definition \ref{ourbicomplex} of 
the spases $C^k_m(V, \W, \F)$  and $C^n_{m'}(V, \W, \F)$
are choosen on the same transversal section. 
%%
%%%%%%%%%%%%%%%%%%%%%%%%%%%%%%%%%%%%%%%%%%%%%%%%%%%%%%%%%%%%%%%%%%%%%%%%%%%%%%%%%%%%%%%%%%%%%%%%%%%%%%%%%
 
%%%%%%%%%%%%%%%%%%%%%%%%%%%%%%%%%%%%%%%%%%%%%%%%%%%%%%%%%%%%%%%%%%%%%%%%%%%%%%%%%%%%%%%%%%%%%%%%%%%%%%%%
Let us recall again the setup for two double complex spaces ${C}^{k}_{m}(V, \W, \F)$ and 
${C}^{n}_{m'}(V, \W, \F)$. 
For $(p_1, \dots, p_k)$,  
$(\widetilde{p}_1, \dots, \widetilde{p}_n)$ 
being two sets of points with local coordinates $(c_1(p_1)$, $ \dots $, $ c_{k}(p_k))$ and 
$(\widetilde{c}_1(\widetilde{p}_1), \dots, \widetilde{c}_n(\widetilde{p}_n))$ 
 taken on the $j$-th transversal section $U_j \in \U$, $j\ge 1$,   
of the transversal basis $\U$.  
For $k \ge 0$, $n \ge 0$,   
 $C^k(V, \W, \F)(U_j)$ and $C^n(V, \W, \F)(U_j)$, $0 \le j \le l$,     
be as before the spaces of all linear maps \eqref{maps}
\begin{eqnarray}
\label{mapy11}
 && \Phi: V^{\otimes k } \rightarrow \W_{c_1(p_1), \dots, c_{k}(p_k)},  
\nn
&&
\Psi: V^{\otimes n } \rightarrow \W_{\widetilde{c}_1(\widetilde{p}_1), \dots, \widetilde{c}_n(\widetilde{p}_n) },  
\end{eqnarray}
composable with $l_1$ and $l_2$ vertex operators \eqref{poper} with formal parameters identified 
with local coordinate functions $c'_j(p'_j)$ and $\widetilde{c}'_j(p'_{j'})$ around points $p_j$,  $p'_{j'}$,
on each of transversal sections $U_j$, $1 \le j \le l_1$, and  $U_{j'}$, $1 \le j' \le l_2$,  correspondingly.   
%%
%%%%%%%%%%%%%%%%%%%%%%%%%%%%%%%%%%%%%%%%%%%%%%%%%%%%%%%%%%%%%%%%%%%%%%%%%%%%%%%%%%%%%%%%%%%%%%%%%%%   
%%%%%%%%%%%%%%%%%%%%%%%%%%%%%%%%%%%%%%%%%%%%%%%%%%%%%%%%%%%%%%%%%%%%%%%%%%%%%%%%%%%%%%%%%%%%%%%%%%%
%%
Then, for $k\ge 0$, $1 \le m \le l_1$, and $n\ge 0$, and $1 \le m'\le l_2$,
according to Definition \ref{ourbicomplex}, the spaces $C^k_m(V, \W, \F)$  and $C^n_{m'}(V, \W, \F)$
are:   
%%%
%%%%%%%%%%%%%%%%%%%%%%%%%%%%%%%%%%%%%%%%%%%%%%%%%%%%%%%%%%%%%%%%%%%%%%%%%%%%%%%%%%%%%%%%%
%%%%%%%%%%%%%%%%%%%%%%%%%%%%%%%%%%%%%%%%%%%%%%%%%%%%%%%%%%%%%%%%%%%%%%%%%%%%
%%
%%
\begin{equation}
\label{ourbicomplex111}
 {C}^{k}_{m}(V, \W, \F) =  \bigcap_{ 
U_{1} \stackrel{h_1} {\hookrightarrow }  \ldots \stackrel {h_{m-1}}{\hookrightarrow } U_{m} 
\atop 1 \le i \le m }  
 C^{k}(V, \W, \F) (U_i),    
\end{equation}
\begin{equation}
\label{ourbicomplex112}
 {C}^{n}_{m'}(V, \W, \U, \F) =  \bigcap_{ 
U_{1} \stackrel{h'_1} {\hookrightarrow }  \ldots \stackrel {h'_{m'-1}}{\hookrightarrow } U_{m'}  
\atop 1 \le i' \le m' }  
 C^{n}(V, \W, \F) (U_{i'}),    
\end{equation}
where the intersection ranges over all possible $m$- and $m'$-tuples of holonomy embeddings 
$h_i$, $i\in \{1, \ldots, m-1\}$, and $h'_{i'}$, $i'\in \{1, \ldots, m'-1\}$,  
between transversal sections $(U_1, \ldots, U_{m})$ and $(U_1, \ldots, U_{m'})$ of the baisis $\U$  for $\F$. 
 In the setup above, we then formulate 
%%
%
%% 
%%%%%%%%%%%%%%%%%%%%%%%%%%%%%%%%%%%%%%%%%%%%%%%%%%%%%%%%%%%%%%%%%%%%%%%%%%%%%%%%%%%%%%%%%%%%%%%%%%%%%%
\begin{definition}
\label{cvitochki}
For 
$\Phi(v_1, x_1; \ldots; v_{k}, x_k)  \in  C^{k}_{m}(V, \W, \F)$,  and 
$\Psi(v'_{1}, y_{1}; \ldots; v'_{n}, y_{n})  \in   C_{m'}^{n}(V, \W, \F)$ %$n'=n-k-1$, 
the product 
\begin{eqnarray}
\label{newcpro}
&&\Phi(v_1, x_1; \ldots; v_{k}, x_k) \cdot_\epsilon \Psi(v'_{1}, y_{1}; \ldots; v'_{n}, y_{n})  
\nn
&&
\qquad \qquad \mapsto {\widehat R}\; \Theta\left( v_1, x_1; \ldots; v_{k}, x_k; v'_{1}, y_{1}; \ldots; v'_{n}, y_{n}; \epsilon\right), 
%%
%\nn
%&&
%%
\end{eqnarray}
 is a $\W_{z_1, \ldots, z_{k+n-r}}%; y_{1}, \ldots, y_{n} 
$
-valued rational 
form 
\begin{eqnarray}
\label{Z2n_pt_epsss}
&& \langle w',  {\widehat R}\;\Theta (v_1, x_1; \ldots; v_{k}, x_k; v'_{1}, y_{1}; \ldots; v'_{n}, y_{n}; \epsilon) \rangle 
\nn
&& \qquad =\langle w',  \Phi (
v_1, x_1; \ldots; v_{k}, x_k) \cdot_{\epsilon} \Psi(v'_{1}, y_{1}; \ldots; v'_{n}, y_{n}
) \rangle 
\nn 
& & \qquad  =  
\sum_{l \in \Z }  \epsilon^l 
 \sum_{u\in V_l } 
 \langle w', Y^{W}_{WV}\left(  
\Phi (v_{1}, x_{1};  \ldots; v_{k}, x_{k}), \zeta_1\right)\; u \rangle  
\nn
& &
\qquad   \qquad  \qquad 
 \langle w', Y^{W}_{WV}\left( 
\Psi(v'_{1}, y_{1}; \ldots; v'_{n}, y_{n})
 , \zeta_{2}\right) \; \overline{u} \rangle,     
\end{eqnarray}
defined by \eqref{Z2n_pt_eps1q1}.   
\end{definition}
%%
%%%
%%
%%%%%%%%%%%%%%%%%%%%%%%%%%%%%%%%%%%%%%%%%%%%%%%%%%%%%%%%%%%%%%%%%%%%%%%%%%%%%%%%%%%%%%%%%%%%%%%%%
%%%%%%%%%%%%%%%%%%%%%%%%%%%%%%%%%%%%%%%%%%%%%%%%%%%%%%%%%%%%%%%%%%%%%%%%%%%%%%%%%%%%%%%%%%%%%%%%%
%
%%
Let $t$ be the number of common vertex operators the mappings  
$\Phi(v_{1}, x_{1}$;  $\ldots$; $v_{k}, x_{k}) \in C^{k}_{m}(V, \W, \F)$ and 
$\Psi(v'_{1}, y_{1}; \ldots; v'_{n}, y_{n}) \in C^{n}_{m'}(V, \W, \F)$, 
are composable with.  
%%
%
%
%%
%%%%%%%%%%%%%%%%%%%%%%%%%%%%%%%%%%%%%%%%%%%%%%%%%%%%%%%%%%%%%%%%%%%%%%%%%%%%%%%
%%
%%%%%%%%%%%%%%%%%%%%%%%%%%%%%%%%%%%%%%%%%%%%%%%%%%%%%%%%%%%%%%%%%%%%%%%%%%%%%%5
\begin{proposition}
\label{tolsto}
For $\Phi(v_1, x_1; \ldots; v_{k}, x_k) \in C_{m}^{k}(V, \W, \F)$ and 
$\Psi(v'_{1}, y_{1}; \ldots; v'_{n}, y_{n})\in C_{m'}^{n}(V, \W, \F)$, 
the product $\widehat{R} \; \Theta\left(v_1, x_1; \ldots; v_{k}, x_k; v'_{1}, y_{1}; \ldots; v'_{n}, y_{n}
; \epsilon\right)$ \eqref{Z2n_pt_epsss} 
belongs to the space $C^{k+n-r
}_{m+m'-t
}(V, \W, \F)$, i.e.,  
\begin{equation}
\label{toporno}
\cdot_\epsilon : C^{k}_{m}(V,\W, \F) \times C_{m'}^{n}(V, \W, \F) \to C_{m+m'-t
}^{k+n-r}(V, \W, \F).  
\end{equation}
\end{proposition}
%%
%%
%%%%%%%%%%%%%%%%%%%%%%%%%%%%%%%%%%%%%%%%%%%%%%%%%%%%%%%%%%%%%%%%%%%%%%%%%%%%%%%%%%%%%%%%
\begin{proof}
In Proposition \ref{derga} we proved that 
$\widehat{R}\; \Theta\left(v_1, x_1 \right. $ ; $ \left. \ldots; v_{k}, x_k; v'_{1}, y_{1}; \ldots; v'_{n}, y_{n}
; \epsilon\right) \in \W_{x_1;, \ldots, x_k;  y_{1}, \ldots, y_{n} 
}$.   
%%
%%
%%%%%%%%%%%%%%%%%%%%%%%%%%%%%%%%%%%%%%%%%%%%%%%%%%%%%%%%%%%%%%%%%%%%%%%%%%%%%%%%%%%%%%%
Namely,  
%%%%%
the rational 
form corresponding to the  
$\epsilon$-product $\widehat{R} \;  \Theta (v_1, x_1; $     
$\ldots; v_{k}, x_k  ;  v'_{1}, y_{1} ;  \ldots ;
  v'_{n}, y_{n}
; \epsilon )$ converges in $\epsilon$,  and %the fact it 
satisfies \eqref{shushu},  $L_V(0)$-conjugation \eqref{loconj} and 
$L_V(-1)$-derivative \eqref{lder1} properties. 
%%
%%%%%%%%%%%%%%%%%%%%%%%%%%%%%%%%%%%%%%%%%%%%%%%%%%%%%%%%%%%%%%%%%%%%%%%%%%%%%%%%%%%%%%%%%%%%%%%%%%%%%%%
%%
 The action of $\sigma \in S_{k+n-r}$ on the product 
$\Theta (v_{1}, x_{1};  \ldots; v_{k}, x_{k}; v'_{k+1}, y_{1};  \ldots$;  $v'_{n}, y_{n};
\epsilon)$   
 \eqref{Z2n_pt_epsss}  
is given by \eqref{sigmaction}.  
%%
%%%%%%%%%%%%%%%%%%%%%%%%%%%%%%%%%%%%%%%%%%%%%%%%%%%%%%%%%%%%%%%%%%%%%%%%%%%%%%%%%%%%%%%%%%%%%%%%%%%%%%%%
%%
Then we see that for the sets of points $(p_1, \dots, p_k; p'_1, \dots, p'_n)$, 
taken on the same 
transversal section $U_j \in \U$, $j\ge 1$, 
by Proposition \ref{derga} we obtain a map 
\begin{eqnarray}
\label{mapy33}
 &&\widehat{R}\; \Theta\left(v_1, x_1; \ldots; v_{k}, x_k; v'_{1}, y_{1}; \ldots; v'_{n}, y_{n}
; \epsilon\right) 
\nn
&&
\qquad \qquad \qquad \qquad : V^{\otimes (k+n) } \rightarrow   \W_{c''_1(p''_1), \dots, c''_1(p''_{k+n-r})},  
\end{eqnarray}
 with formal parameters $(z_1, \ldots, z_{k+n-r})$ identified with local coordinates 
$(c''_1(p''_1), \dots $, $ c''_1(p''_{k+n-r}))$ of points 
\[
(p''_1, \dots, p''_{k+n-r})= (p_1, \ldots, p_k; p_1, \ldots, \widehat{p'_{i_l}}, \ldots,  p'_n),  
\]
for coinciding points $p_{i_l}=p'_{j_l}$, $1\le l \le r$.  
%%
%%%%%%%%%%%%%%%%%%%%%%%%%%%%%%%%%%%%%%%%%%%%%%%%%%%%%%%%%%%%%%%%%%%%%%%%%%%%%%%%%%%%%%%%
Next, we prove 
\begin{proposition}
%%%%%
\label{ccc}
The product 
$\Theta\left(
v_{1}, x_{1}; \ldots; v_{k},  x_{k}; v'_{1}, y_{1}; \ldots; v'_{n},  y_{n}
; \epsilon\right)$
 \eqref{Z2n_pt_epsss}  
is composable with $m+m'-t$ vertex operators.  
\end{proposition}
The proof of this proposition is contained in Appendix \ref{rasto}. 

%%
%%%%%%%%%%%%%%%%%%%%%%%%%%%%%%%%%%%%%%%%%%%%%%%%%%%%%%%%%%%%%%%%%%%%%%%%%%%%%%%%%%%%%%%%%%%
Since we have proved that 
the product $\widehat{R} \; \Theta\left(v_1, x_1; \ldots; v_{k}, x_k; v'_{1}, y_{1}; \ldots; v'_{n}, y_{n}
; \epsilon\right)$ is 
composable with a $m+m'-t$ of vertex operators \eqref{poper} with formal parameters identified 
with local coordinates $c_j(p''_j)$ functions around points $(p_1, \ldots, p_k; p'_1, \ldots, p'_n)$ 
on each of transversal sections $U_j$, $1 \le j \le l$, we conclude that 
according to Definition \ref{initialspace}, 
the product $\widehat{R} \;\Theta\left(v_1, x_1; \ldots; v_{k}, x_k; v'_{1}, y_{1}; \ldots; v'_{n}, y_{n}
; \epsilon\right)$ belongs to the space $C^{k+n-r}(V, \W, \F)(U_j)$, $0 \le j \le l$, 
for $l \ge 0$, 
%%
%%%%%%%%%%%%%%%%%%%%%%%%%%%%%%%%%%%%%%%%%%%%%%%%%%%%%%%%%%%%%%%%%%%%%%%%%%%%%%%%%%%%%%%%%%%%
%%
 on one of transversal sections $U_j \in \U$, $j\ge 1$.  
\begin{eqnarray}
\label{mapy11}
 && \Theta: V^{\otimes (k+n-r) } \rightarrow \W_{c_1(p_1), \dots, c_{k}(p_k); c'_1(p'_1), \dots, c'_{n}(p'_n)},  
\end{eqnarray}
and the intersection ranges over all possible $m+m'-t$-tuples of holonomy embeddings 
$h_i$, $i\in \{1, \ldots, m+m'-t -1\}$, %and $h'_{i'}$, $i'\in \{1, \ldots, m'-1\}$,   
between transversal sections $U_1, \ldots, U_{l}$ of the baisis $\U$  for $\F$. 
%%
%%%%%%%%%%%%%%%%%%%%%%%%%%%%%%%%%%%%%%%%%%%%%%%%%%%%%%%%%%%%%%%%%%%%%%%%%%%%%%%%%%%%%%%%%%%%
%%
The product 
$\Theta\left(v_1, x_1; \ldots; v_{k}, x_k; v'_{1}, y_{1}; \ldots; v'_{n}, y_{n}
; \epsilon\right)$
belongs to the space 
\begin{equation}
\label{ourbicomplex1110000}
 {C}^{k+n-r}_{m+m'-t}(V, \W, \U, \F) =  \bigcap_{ 
U_{1} \stackrel{h_1} {\hookrightarrow }  \ldots \stackrel {h_{m+m'-1}}{\hookrightarrow } U_{m+m'}   
\atop 1 \le j \le m+m' -t}  
 C^{k+n-r}(V, \W, \F) (U_j),    
\end{equation}
where the intersection ranges over all possible $m+m'-t$-tuples 
of holonomy embeddings $h_i$, $i\in \{1, \ldots, m+m'-t-1\}$,  
between transversal sections $U_1, \ldots, U_{m+m'-t-1}$ of the baisis $\U$  for $\F$. 
This finishes the proof of the proposition. 
\end{proof}
%%

%%
%%%%%%%%%%%%%%%%%%%%%%%%%%%%%%%%%%%%%%%%%%%%%%%%%%%%%%%%%%%%%%%%%%%%%%%%%%%%%%%%%%%%%%%%%%%%%%%%%%%%%%%%%%%
%%%%%%%%%%%%%%%%%%%%%%%%%%%%%%%%%%%%%%%%%%%%%%%%%%%%%%%%%%%%%%%%%%%%%%%%%%%%%%%%%%%%%%%%%%%%%%%%%%%%%%%%%%%
\section{Coboundary operators and cohomology of codimension one foliations}
\label{coboundary}
%%
%%%%%%%%%%%%%%%%%%%%%%%%%%%%%%%%%%%%%%%%%%%%%%%%%%%%%%%%%%%%%%%%%%%%%%%%%%%%%%%%%%%%%%%%%%%%%%%%%%%%%%%%%%%%
%%
%%
In this Section we recall the definition of differential operators acting on double complex spaces. 
%%
%%
%%%%%%%%%%%%%%%%%%%%%%%%%%%%%%%%%%%%%%%%%%%%%%%%%%%%%%%%%%%%%%%%%%%%%%%%%%%%%%%%%%%%%%%%%%%%
%%%%%%%%%%%%%%%%%%%%%%%%%%%%%%%%%%%%%%%%%%%%%%%%%%%%%%%%%%%%%%%%%%%%%%%%%%%%%%%%%%%%%%%%%%%%
Recall the definitions of $E$-elements given in Appendix \ref{properties}. 
%%
%%%
Consider the vector of $E$-elements: 
\begin{eqnarray}
\label{mathe}
\mathcal E &=&  \left( E^{(1)}_W,\; \sum\limits_{i=1}^n (-1)^i \; E^{(2)}_{V; \one_V},  \; E^{W; (1)}_{W V} \right).   
%%
%%\nn
%%
 %%
%%
\end{eqnarray} 
Then we formulate 
%%
%%%%%%%%%%%%%%%%%%%%%%%%%%%%%%%%%%%%%%%%%%%%%%%%%%%%%%%%%%%%%%%%%%%%%%%%%%%%%%%%%%%%%%%%%%%%
%%%%%%%%%%%%%%%%%%%%%%%%%%%%%%%%%%%%%%%%%%%%%%%%%%%%%%%%%%%%%%%%%%%%%%%%%%%%%%%%%%%%%%%%%%%%
%%
\begin{definition}
%%%%%%%%%%%%%%%%%%%%%%%%%%%%%%%%%%%%%%%%%%%%%%%%%%%%%%%%%%%%%%%%%%%%%%%%%%%%%%%%%%%%%%%%%%%%%
%%
%%
The coboundary operator 
${\delta}^{n}_{m}$ acting on elements $\Phi \in C^{n}_{m}(V, \W, \F)$ of the spaces \eqref{ourbicomplex}, 
is defined by 
\begin{equation}
\label{deltaproduct}
\delta^n_m \Phi ={\mathcal E} \cdot_\epsilon \Phi,    
\end{equation}
\end{definition}
%%
%
%%
%%%%%%%%%%%%%%%%%%%%%%%%%%%%%%%%%%%%%%%%%%%%%%%%%%%%%%%%%%%%%%
%%
%%%%%%%%%%%%%%%%%%%%%%%%%%%%%%%%%%%%%%%%%%%%%%%%%%%%%%%%%
\begin{remark}
The action of $\delta^n_m$ 
has the orthogonality conditition form (cf. Section \ref{gv})   
 \cite{G} with respect to the $\cdot_\epsilon$ product \eqref{Z2n_pt_epsss}.  
Note that it is assumed that the coboundary operator does not affect 
$dc(p)^{\wt(v_i)}$-tensor multipliers in $\Phi$. 
\end{remark}
%%%%%%%%%%%%%%%%%%%%%%%%%%%%%%%%%%%%%%%%%%%%%%%%%%%%%%%%%%
%%
Then we obtain  
%%
%%%%%%%%%%%%%%%%%%%%%%%%%%%%%%%%%%%%%%%%%%%%%%%%%%%%%%%%%
\begin{lemma}
%%
%%%%%%%%%%%%%%%%%%%%%%%%%%%%%%%%%%%%%%%%%%%%%%%%%%%%%%%%%%%%%%%%%%%%%%%%%%%%%%%%%%%%%%%%%%
%%
%%
%%
The definition \eqref{deltaproduct} is equivalent to 
 a multi-point vertex algebra connection 
(cf. Definition \ref{locus} in Section \ref{connections}): 
\begin{equation}
\label{deltaproduct}
\delta^n_m \Phi = G(p_1, \ldots, p_{n+1}),     
\end{equation}
where 
%%%%%%%%%%%%%%%%%%%%%%%%%%%%%%%%%%%%%%%%%%%%%%%%%%%%%%%%%%%%%%%%%%%%%%%%%%%%%%%%%%%%%%%%%%%%%%%%%%%%%%%%%%%
%%
\begin{eqnarray}
\label{hatdelta1}
%%
%%&&
%% 
%%\nonumber 
%%
%
%%
 G(p_1, \ldots, p_{n+1}) 
%%%%%%%%%%%%%%%%%%%%%%%%%%%%%%%%%%%%%%%%%%%%%%%%%%%%%%%%%%%%%%%%%%%%%%%%%%%%%%%%%%%%%%%%%%%%%%%%%%%%%%%%%%%%%%%%%%%
%%
&=& 
%% %R \left(
\langle w', \sum_{i=1}^{n}(-1)^{i} \; \Phi\left( \omega_V\left(v_i, c_i(p_i) - c_{i+1}(p_{i+1})) 
 v_{i+1} \right)  
 \right) \rangle,  %\right) 
\\
\nonumber
%%
%R\left(
&+& \langle w', \omega_W \left(v_1, c_1(p_1)  \right) \; \Phi (v_2, c_2(p_2); \ldots; v_{n+1}, c_n(p_{n+1}) )  
\rangle  %\right)  
\nn
\nonumber
 &+& (-1)^{n+1} %R\left(
\langle w', 
 \omega_W(v_{n+1}, c_{n+1} (p_{n+1})) 
\; \Phi(v_1, c_1(p_2); \ldots; v_n, c_n(p_n)) \rangle,  % \right), 
\end{eqnarray}
for arbitrary $w' \in W'$ (dual to $W$).
\end{lemma}
\begin{proof}
%%
%%%%%%%%%%%%%%%%%%%%%%%%%%%%%%%%%%%%%%%%%%%%%%%%%%%%%%%%%%%%%%%%%%%%%%%%%%%%%%%%%%%%%%%%%%%%%%%%%%%%%%%%%
%
%%
%
%%
The statement follows from the intertwining operator (cf. Appendix \ref {grading}) representation of the definition 
 \eqref{deltaproduct} 
in the form 
\[
\delta^n_m \Phi= \sum\limits_{i=1}^3  %R\left(
\langle w', 
  e^{\xi_i L_W(-1)}\; \omega^{W}_{WV} \left( \Phi_i \right) u_i \rangle, % \right),   
\]
for some $\xi_i \in \C$, and $u_i \in V$, and $\Phi_i$ obvious from \eqref{deltaproduct}.  
Namely, 
\begin{eqnarray*}
&& \delta^n_m \Phi=  
 \langle w',  e^{ c_1(p_1) L(-1)_{W} } \;  \omega^{W}_{WV}   
%%%%
\left( \Phi \left( v_2, c_2(p_2); \ldots; v_n, c_{n+1} (p_{n+1}), - c_1(p_1)  \right) v_1    
\rangle %\Big)  
\right. 
\nn
&&
+\sum\limits_{i=1}^n (-1)^i e^{\zeta L_W(-1)} 
\langle w', \omega^{W}_{WV} \left( 
\Phi \left( \omega_V\left(v_i, c_i(p_i) - c_{i+1}(p_{i+1})\right), -\zeta \right)  \one_V \rangle \right) 
\nn
&&
+   %R \left(
 \langle w',  e^{ c_{n+1}(p_{n+1}) L(-1)_{W} } \;  \omega^{W}_{WV}    
\left( \Phi \left( v_1, c_1(p_1); \ldots; v_{n}, c_n (p_n), - c_{n+1}(p_{n+1})  \right) v_{n+1}    \right)   
\rangle,  
\end{eqnarray*}
for an arbitrary $\zeta\in \C$. 
\end{proof}
%%
%%%%%%%%%%%%%%%%%%%%%%%%%%%%%%%%%%%%%%%%%%%%%%%%%%%%%%%%%%%%%%%%%%%%%%%

%%%%%%%%%%%%%%%%%%%%%%%%%%%%%%%%%%%%%%%%%%%%%%%%%%%%%%%%%%%%%%%%%%%%%%%%%%%%%%%%%%%%%%%%%%%%%%%%%%%%%%%%
\begin{remark}
%%
%
%%%%%%%%%%%%%%%%%%%%%%%%%%%%%%%%%%%%%%%%%%%%%%%%%%%%%%%%%%%%%%%%%%%%%%%%%%%%%%%%%%%5
%%%
Inspecting construction of the double complex spaces \eqref{ourbicomplex}
 we see that the action \eqref{hatdelta1} of the 
$\delta^n_m$ on an element of $C^n_m(V, \W, \F)$ provides a coupling (in terms of $\W_{z_1, \ldots, z_n}$-valued 
rational functions) of vertex operators
 taken at the local coordinates 
$c_i(z_{p_i})$, $0 \le i \le k$,  
at the vicinities of the same points $p_i$ taken on transversal sections for $\F$,  
with elements of $C^n_{m-1}(V, \W, \F)$ taken at   
points at the local coordinates 
$c_i(z_{p_i})$, $0 \le i \le n$ on $\mathcal M$ for points $p_i$ considered on the leaves of $\mathcal M/\F$.  
\end{remark}
%%%%%%%%%%%%%%%%%%%%%%%%%%%%%%%%%%%%%%%%%%%%%%%%%%%%%%%%%%%%%%%%%%%%%%%%%%%%%%%%%%%%%%%%%% 
%%%%%%%%%%%%%%%%%%%%%%%%%%%%%%%%%%%%%%%%%%%%%%%%%%%%%%%%%%%%%%%%%%%%%%%%%%%%%%%%%%%%%%%%%%%%%%%%%%%%%
%%%%%%%%%%%%%%%%%%%%%%%%%%%%%%%%%%%%%%%%%%%%%%%%%%%%%%%%%%%%%%%%%%%%%%%%%%%%%%%%%%%%%%%%%%%%%%%%%%%%%
\subsection{Complexes on transversal connections}  
\label{comtrsec}
In addition to the double complex $(C^n_m(V$, $\W$, $\F)$, $\delta^n_m)$ provided by \eqref{ourbicomplex} and 
\eqref{deltaproduct},
 there exists an exceptional short double complex which we call transversal connection complex.  
 We have  
\begin{lemma}
\label{lemmo}
 For $n=2$, and $k=0$, there exists a subspace
 $C^{0}_{ex}(V, \W, \F)$ 
\[
{C}_{m}^{2}(V, \W, \F) \subset C^{0}_{ex}(V, \W, \F) \subset C_{0}^{2}(V, \W, \F), 
\]
 for all $m \ge 1$, with the action of coboundary operator 
${\delta}^{2}_{m}$ defined.   
\end{lemma}
\begin{proof}
Let us consider the space $C_{0}^{2}(V, \W, \F)$.
vertex operators composable. 
%%
%% 
%%%%%%%%%%%%%%%%%%%%%%%%%%%%%%%%%%%%%%%%%%%%%%%%%%%%%%%%%%%%%%%%%%%%%%%%%%%%%%%%%%%%%%%%%%%%%%%%%%%%%
%
Indeed, the space $C_{0}^{2}(V, \W, \F)$ contains elements of $\W_{c_1(p_1), c_2(p_2)}$ so that the action  
of $\delta_{0}^{2}$ is zero. 
%%
%%%%%%%%%%%%%%%%%%%%%%%%%%%%%%%%%%%%%%%%%%%%%%%%%%%%%%%%%%%%%%%%%%%%%%%%%%%%%%%%%%%%%%%%%%%%%%%%%%
Nevertheless, as for $\mathcal J^n_m(\Phi)$ in \eqref{Jnm}, Definition \ref{composabilitydef}, let us consider sum of  
projections 
\[
P_r: \W_{z_i, z_j} \to W_r, 
\]
for $r\in \C$, and $(i, j)=(1,2), (2, 3)$,   
so that the condition \eqref{Jnm}  is satisfied for some connections similar to the action \eqref{Jnm} of $\delta^2_0$. 
Separating the first two and the second two summands in \eqref{hatdelta1}, we find that for a subspace 
of ${C}_{0}^{2}(V, \W, \F)$, which we denote as 
${C}_{ex}^{2}(V, \W, \F)$, consisting of three-point connections $\Phi$ such that
 for $v_{1}$, $v_{2}$, $v_{3} \in V$, 
$w'\in W'$, and arbitrary $\zeta \in \C$, the following forms of connections 
%%
%%%%%%%%%%%%%%%%%%%%%%%%%%%%%%%%%%%%%%%%%%%%%%%%%%%%%%%%%%%%%%%%%%%%%%%%%%%%%%%%%%%%%%%%%%%%%%%%%%%%%
%%%%%%%%%%%%%%%%%%%%%%%%%%%%%%%%%%%%%%%%%%%%%%%%%%%%%%%%%%%%%%%%%%%%%%%%%%%%%%%%%%%%%%%%%%%%%%%%%%%%%
\begin{eqnarray}
\label{pervayaforma}
&& G_1(c_1(p_1), c_2(p_2), c_3(p_3))
\nn
&&
= \sum_{r\in \C} \Big( \langle w', E^{(1)}_{W} \left(  v_{1}, c_1(p_{1});     
P_{r} \left(   \Phi \left(v_{2}, c_2(p_{2})-\zeta;  v_{3}, c_3(p_{3}\right) - \zeta  \right)  \right) \rangle  
\nn
&&\quad +\langle w', 
\Phi \left( v_{1}, c_1(p_{1});  P_{r} \left(E^{(2)}_{V}
\left(v_{2}, c_2(p_{2})-\zeta; v_{3}, c_3(z_{3} \right)-\zeta; \one_V \right),    
 \zeta \right) 
\rangle \Big) 
\nn
&&
%%%%%%%%%%%%%%%%%%%%%%%%%%%%%%%%%%%%%%%%%%%%%%%%%%%%%%%%%%%%%%%%%%%%%%%%%%%%%%%%%%%%%%
\nn
&& = \sum_{r\in \C} \big( \langle w', \omega_{W} \left( v_{1}, c_1(p_{1}) \right)\;       
 P_{r}\left( \Phi\left(v_{2}, c_2(p_{2})-\zeta; v_{3}, c_3(p_{3}) - \zeta \right) \right) \rangle 
\nn
&&\quad +\langle w', 
\Phi \left( v_{1}, c_1(p_{1});  P_{r} \left(  \omega_V  
\left(v_{2}, c_2(p_{2})-\zeta\right)  \omega_V \left( v_{3}, c_3(z_{3})-\zeta \right) \one_V \right),     
 \zeta \right)
\rangle\big),  
\nn
\end{eqnarray}
%%
%%%%%%%%%%%%%%%%%%%%%%%%%%
and 
%%%%%%%%%%%%%%%%%%%%%%%%%%
%%
\begin{eqnarray}
\label{vtorayaforma}
&&
G_2(c_1(p_1), c_2(p_2), c_3(p_3)) 
\nn 
 && \qquad =\sum_{r\in \C}\Big(\langle w', 
 \Phi \left(  P_{r} \left(   E^{(2)}_{V} \left(v_{1}, c_1(p_{1})-\zeta, v_2, c_2(p_{2})-\zeta; \one_V \right) \right), 
\zeta;  
 v_{3}, c_3(p_{3}) \right) \rangle  
\nn
&&\quad + \langle w', 
E^{W; (1)}_{WV} \left(    P_{r} \left(    \Phi \left(   v_{1},  c_1(p_{1})-\zeta;   v_{2}, c_2(p_{2}\right) -\zeta \right), 
\zeta ;   
 v_{3}, c_3(p_{3}) \right) \rangle \Big) 
\nn
&&
\nn
%%%%%%%%%%%%%%%%%%%%%%%%%%%%%%%%%%%%%%%%%%%%%5
&&
 = \sum_{r\in \C}\big(\langle w', 
 \Phi \left(P_{r}( \omega_V \left(v_{1}, c_1(p_{1})-\zeta\right) \omega_V\left(v_2, c_2(p_{2})-\zeta) \one_V, 
\zeta \right));  
 v_{3}, c_3(p_{3}) \right) \rangle  
\nn
&&\quad + \langle w', 
 \omega_V \left(v_{3}, c_3(p_{3}) \right) \; P_{r} \left(    \Phi \left(v_{1}, c_1(p_{1})-\zeta; v_{2}, c_2(p_{2})-\zeta
 \right)  
\right)
  \rangle \big),  
\end{eqnarray}
%%
%%%%%%%%%%%%%%%%%%%%%%%%%%%%%%%%%%%%%%%%%%%%%%%%%%%
are absolutely convergent in the regions 
\[
|c_1(p_{1})-\zeta|>|c_2(p_{2})-\zeta|,  
\]
\[
 |c_2 (p_{2})-\zeta|>0, 
\]
%%
 %and 
%%
\[
|\zeta-c_3(p_{3})|>|c_1(p_{1})-\zeta|, 
\]
\[
|c_2(p_{2})-\zeta|>0,
\]
%%%%%%%%%%%%%%%%%%%%%%%%%%%%%%%%%%%%%%%%%%%%%%%%%%%%%%%%%%%%%%%%%%%%%%%%%
%% 
%%
where $c_i$, $1 \le i \le 3$ are coordinate functions, 
 respectively, and can be analytically extended to  
rational form-valued functions in $c_1(p_{1})$ and $c_2(p_{2})$ with the only possible poles at
$c_1(p_{1})$, $c_2(p_{2})=0$, and $c_1(p_{1})=c_2(p_{2})$.
Note that \eqref{pervayaforma} and \eqref{vtorayaforma} constitute the first two and the last two terms of 
\eqref{hatdelta1} correspondingly.  
According to Proposition \ref{comp-assoc} (cf. Appendix \ref{composable}),  
 ${C}_{m}^{2}(V, \W, \F)$ is a subspace of ${C}_{ex}^{2}(V, \W, \F)$,  for $m\ge 0$,
and $\Phi \in {C}_{m}^{2}(V, \W, \F)$ are composable with $m$ vertex operators. 
Note that \eqref{pervayaforma} and \eqref{vtorayaforma} represent sums of forms $G_{tr}(p, p')$ 
of transversal connections 
\eqref{transa} (cf. Section \ref{cohomological}). 
\end{proof}
\begin{remark}
It is important to mention that, according to the general principle, abserved in \cite{BG}, 
for non-vanishing connection $G(c(p), c(p'), c(p''))$,  
there exists an invariant structure, e.g., a cohomological class. 
In our case, it appears as a non-empty subspaces 
${C}_{m}^{2}(V, \W, \F) \subset {C}_{ex}^{2}(V, \W, \F)$ in ${C}_{0}^{2}(V, \W, \F)$. 
\end{remark}

%%
%%%%%%%%%%%%%%%%%%%%%%%%%%%%%%%%%%%%%%%%%%%%%%%%%%%%%%%%%%%%%%%%%%%%%%%%%%
%%
%%  
%%%%%
Then we have 
\begin{definition}
\label{cobop}
 The coboundary operator 
\begin{equation}
\label{halfdelta}
{\delta}^{2}_{ex}: {C}_{ex}^{2}(V, \W, \F)
\to {C}_{0}^{3}(V, \W, \F),
\end{equation}
is defined 
 by 
three point connection of the form 
\begin{equation}
\label{ghalfdelta}
\delta^2_{ex} \Phi= {\mathcal E}_{ex}\cdot_\epsilon \Phi =G_{ex}(p_1, p_2, p_3),  
\end{equation}
where 
\begin{equation}
\label{mathe2}
\mathcal E_{ex} = \left( E^{(1)}_W,\; \sum\limits_{i=1}^2 (-1)^n E^{(2)}_{V; \one_V}, \; E^{W; (1)}_{W V} \right),  
\end{equation}
\begin{eqnarray}
\label{ghalfdelta1}
G_{ex}(p_1, p_2, p_3) &=&  %R \left(
\langle w', \omega_{W} \left(v_{1}, c_1(p_1)) \; 
\Phi \left(v_{2}, c_2(p_2); v_{3}, c_3(p_3) \right) \rangle \right) 
\nn
&&
\quad \quad 
- %R\left(
 \langle w',  
\Phi \left(  \omega_{V}( v_{1}, c_1(p_1))\; \; \omega_V (v_{2}, c_2(p_2) ) \one_V; v_{3}, c_3(p_3) \right)  %\right)
\rangle
\nn
 &&\quad
+ %R \left(
 \langle w', \Phi(v_{1}, c_1(p_1);  \; \omega_{V} (v_{2}, c_2(p_2))\;  \omega_V( v_{3}, c_3(p_3)) \one_V) 
 \rangle %\right)
\nn
 &&\quad \quad 
+ %R \left(
 \langle w',  
 \omega_W (v_{3}, c_3(p_3)) \; \Phi \left(v_{1}, c_1(p_1); v_{2}, c_2(p_2) \right)  
\rangle,% \right), 
\end{eqnarray}
for $w'\in W'$,
$\Phi\in {C}_{ex}^{2}(V, \W, \F)$,
$v_{1}, v_{2}, v_{3}\in V$ and $(z_{1}, z_{2}, z_{3})\in F_{3}\C$.   
\end{definition}
%%
%%%%%%%%%%%%%%%%%%%%%%%%%%%%%%%%%%%%%%%%%%%%%%%55
%
%%%%%%%%%%%%%%%%%%%%%%%%%%%%%%%%%%%%%%%%%%%%%%%%%%%%%%%%%%%%%%%%%%%%%%%%%
Then we have 
%%
%% tut nado napisat, chto 
%%
\begin{proposition}
\label{cochainprop}
The operators \eqref{deltaproduct} and \eqref{halfdelta} provide the chain-cochain complexes    
\begin{equation}
\label{conde}
{\delta}^{n}_{m}: C_{m}^{n}(V, \W, \F)  
\to C_{m-1}^{n+1}(V, \W, \F),   
\end{equation}  
\begin{equation}
\label{deltacondition}  
{\delta}^{n+1}_{m-1} \circ {\delta}^{n}_{m}=0,  
\end{equation} 
\[
{\delta}^{2}_{ex}\circ {\delta}^{1}_{2}=0, 
\]
\begin{equation}\label{hat-complex}
0\longrightarrow {C}_{m}^{0}(V, \W, \F) 
\stackrel{{\delta}^{0}_{m}}{\longrightarrow}
 {C}_{m-1}^{1}(V, \W, \F)  
\stackrel{{\delta}^{1}_{m-1}}{\longrightarrow}\cdots 
\stackrel{{\delta}^{m-1}_{1}}{\longrightarrow}
 {C}_{0}^{m}(V, \W, \F)\longrightarrow 0, 
\end{equation}
\begin{equation}\label{hat-complex-half}
0\longrightarrow {C}^{0}_{3}(V, \W, \F)
\stackrel{{\delta}_{3}^{0}}{\longrightarrow} 
 {C}^{1}_{2}(V, \W, \F)
\stackrel{{\delta}_{2}^{1}}{\longrightarrow}{C}_{ex}^{2}(V, \W, \F)
\stackrel{{\delta}_{ex}^{2}}{\longrightarrow} 
 {C}_{0}^{3}(V, \W, \F)\longrightarrow 0,
\end{equation}
on the spaces \eqref{ourbicomplex}. 
\end{proposition}
%%
%%%%%%%%%%%%%%%%%%%%%%%%%%%%%%%%%%%%%%%%%%%%%%%%%%%%%%%%%%55
%
%%
Since  
\[
{\delta}_{2}^{1} \;  {C}_{2}^{1}(V, \W, \F) \subset 
 {C}_{1}^{2}(V, \W, \F)\subset 
 {C}_{ex}^{2}(V, \W, \F),
\]
the second formula follows from the first one, and 
\[
{\delta}^{2}_{ex}\circ  {\delta}^{1}_{2}
= {\delta}^{2}_{1}\circ  {\delta}^{1}_{2} 
=0.
\]
%%
%%%%%%%%%%%%%%%%%%%%%%%%%%%%%%%%%%%%%%%%%%%%%%%%%%%%%%%%%%%%%%%%%%%%%%%%%%%%%%%5
\begin{proof}
The proof of this proposition is analogous to that of Proposition (4.1) in \cite{Huang} 
for chain-cochain complex of a grading-restricted vertex algebra.  
The only difference is that we work with the space $\W_{c_1(p_1), \ldots, c_n(p_n)}$ 
instead of
 $W_{z_1, \ldots, z_n}$. 
\end{proof}
%%
%%%%%%%%%%%%%%%%%%%%%%%%%%%%%%%%%%%%%%%%%%%%%%%%%%%%%%%%%%%%%%%%%%%%%%%%%%%%%%%%%%%%%%%%%%%%%%%%%%%%%%%%%
%% 
%% 
%%%%%%%%%%%%%%%%%%%%%%%%%%%%%%%%%%%%%%%%%%%%%%%%%%%%%%%%%%%%%%%%%%%%%%%%%%%%%%%%%%%%%%%%%%%%%%%%%%%%%%%%%%%%%%%%%%
%%%%%%%%%%%%%%%%%%%%%%%%%%%%%%%%%%%%%%%%%%%%%%%%%%%%%%%%%%%%%%%%%%%%%%%%%%%%%%%%%%%%%%%%%%%%%%%%%%%%%%%%%%%%%%%%%%
\subsection{Vertex algebra cohomology and relation to
Crainic and Moerdijk construction}
Now let us define 
 the cohomology  
of the leaf space $\mathcal M/\F$ for codimension one foliation $\F$ 
associated with a grading-restricted vertex algebra $V$. 
%%%
%%
%% 
\begin{definition}
\label{defcohomology}
 We define the 
 $n$-th cohomology $H^{n}_{k}(V, \W, \F)$ of $\mathcal M/\F$ with coefficients in   
$\W_{z_1, \ldots, z_n}$ (containing maps composable $k$ vertex operators on $k$ transversal sections)    
to be the factor space of closed 
multi-point connections by the space of connection forms: 
\begin{equation}
\label{cohom1}
 H_{k}^{n}(V, \W, \F)= {\mathcal Con}_{k; \; cl}^n/G^{n-1}_{k+1}. 
\end{equation}
\end{definition}
%%%%%%%%%%%%%%%%%%%%%%%%%%%%%%%%%%%%%%%%%%%%%%%%%%%%%%%%%%%%%%%%%%%
%%
%%  
%%
Note that due to \eqref{hatdelta1}, \eqref{ghalfdelta1}, 
and Definitions \ref{locus} and \ref{gform} (cf. Section \ref{coboundary}), 
  it is easy to see that \eqref{cohom1} is equivalent to the standard cohomology definition  
\begin{equation}
\label{cohom}
 H_{k}^{n}(V, \W, \F)= \ker  \delta^{n}_{k}/\mbox{\rm im}\; \delta^{n-1}_{k+1}. 
\end{equation}  
%%
%%%%%%%%%%%%%%%%%%%%%%%%%%%%%%%%%%%%%%%%%%%%%%%%%%%%%%%%%%%%%%%%%%%%%%%%%%%%%%%%%%%%%%%%%%%%%%%%%%%%%%%5
%%%%%%%%%%%%%%%%%%%%%%%%%%%%%%%%%%%%%%%%%%%%%%%%%%%%%%%%%%%%%%%%%%%%%%%%%%%%%%%%%%%%%%%%%%%%%%%%%%%%%%%%
%%%%%%%%%%%%%%%%%%%%%%%%%%%%%%%%%%%%%%%%%%%%%%%%%%%%%%%%%%%%%%%%%%%%%%%%%%%%%%%%%%%%%%%%%%%%%%%%%%%%
%

%%
Recall the construction of the ${\rm \check C}$ech-de~Rham cohomology of a foliation \cite{CM}.  
 Consider a foliation $\F$ of codimension one %$n$
 defined on a smooth complex curve %manifold 
$\mathcal M$.
Consider the double
complex 
\begin{equation}
\label{Cpq}
C^{k,l}=\prod_{U_0\stackrel{h_1}{\hookrightarrow }\cdots
\stackrel{h_k}{\hookrightarrow } U_k}\Omega^l(U_0),
\end{equation}
 where $\Omega^l(U_0)$ is the space of differential $l$-forms on $U_0$, 
and the 
product ranges over all $k$-tuples of holonomy embeddings between
transversal sections from a fixed transversal basis $\U$.
Component of $\varpi\in C^{k, l}$ are denoted by
$\varpi(h_1, \ldots, h_l)\in \Omega^l(\U_0)$.
 The
vertical differential is defined as 
\[
(-1)^k d:C^{k,l}\to
C^{k,l+1},
\]
 where $d$ is the usual de~Rham differential. 
The horizontal differential 
\[
\delta:C^{k,l}\to C^{k+1,l}, 
\]
 is given by
\[
\delta= \sum\limits_{i=1}^k (-1)^{i} \delta_{i},   
\]
\begin{equation}
\label{deltacpq}
\delta _{i} \varpi(h_1, \ldots , h_{k+1}) = G(h_1, \ldots , h_{k+1}),  
\end{equation}
where $G(h_1, \ldots , h_{k+1})$ is the multi-point connection of the form \eqref{locus}, i.e., 
\begin{equation}
\label{delta} 
\delta_{i} \varpi(h_1, \ldots , h_{p+1})= \left\{ \begin{array}{lll}
                                      h_{1}^{*}\varpi(h_2, \ldots , h_{p+1}),  \ \ \mbox{if $i=0,$}\\ 
                                      \varpi(h_1, \ldots, h_{i+1}h_{i}, \ldots, h_{p+1}),  \ \ \mbox{if $0<i< p+1,$}\\
                                      \varpi(h_1, \ldots, h_p), \ \ \mbox{if $i= p+1.$}
                        \end{array}
                \right.
\end{equation}
This double complex is actually a bigraded differential algebra, with the usual product 
\begin{equation}
\label{bigradif}
(\varpi \cdot  \eta)(h_1, \ldots , h_{k+k\,'})= (-1)^{kk\,'} \varpi(h_1, \ldots , 
h_{k}) \; h_{1}^{*} \ldots h_{k}^{*}\;.\eta(h_{k+1}, \ldots h_{k+k\,'}), 
\end{equation}
for $\varpi\in C^{k, l}$ and $\eta\in C^{k',l'}$, 
thus $(\varpi\cdot\eta)(h_1, \ldots , h_{k+k\,'}) \in C^{k+ k',l+ l'}$. 
%%
%%%%%%%%%%%%%%%%%%%%%%%%%%%%%%%%%%%%%%%%%%%%%%%%%%%%%%%%%%%%%%%%%%%%%%%%%%
%%
\begin{definition}
The cohomology $\check{H}^*_\U(M/\F)$ of this complex is called the ${\rm \check C}$ech-de~Rham
cohomology of the leaf space $\mathcal M/\F$ with respect to the transversal basis $\U$. 
It is defined by 
\[
\check{H}^*_\U(M/\F)= {\mathcal Con}^{k+1}_{cl}(h_1, \ldots , h_{k+1})/ G^k(h_1, \ldots , h_{k}),  
\]
where ${\mathcal Con}^{k+1}_{cl}(h_1, \ldots , h_{k+1})$ is the space of closed multi-point connections, 
and $G^k$ $(h_1$, $\ldots$, $h_{k})$ is the space of $k$-point connection forms. 
\end{definition}
%%

%%%%%%%%%%%%%%%%%%%%%%%%%%%%%%%%%%%%%%%%%%%%%%%%%%%%%%%%%%%%%%%%%%%%%%%%%%%%%%%%%%%%%%%%%%%%%
%%%%%%%%%%%%%%%%%%%%%%%%%%%%%%%%%%%%%%%%%%%%%%%%%%%%%%%%%%%%%%%%%%%%%%%%%%%%%%%%%%%%%%%%%%%%%
%%
In this subsection we show 
%In particular, we have 
%%
the following 
%%
%%%%%%%%%%%%%%%%%%%%%%%%%%%%%%%%%%%%%%%%%%%%%%%%%%%%%%%%%%%%%%%%%%%%%%%%%%%%%%%%%%%%%%%%%%%%%%%%
\begin{lemma}
In the case of codimension one foliation on a smooth complex curve, 
the construction of the double complex $\left(C^{k,l}, \delta\right)$,  \eqref{Cpq}, \eqref{deltacpq} 
 results from %follows from 
the construction of the double complexes  
$\left(C_m^n(V, \W, \F), \delta^n_m\right)$ of \eqref{hat-complex} and \eqref{hat-complex-half}.  
\end{lemma}
%%
%%%%%%%%%%%%%%%%%%%%%%%%%%%%%%%%%%%%%%%%%%%%%%%%%%%%%%%%%%%%%%%%%%%%%%%%%%%%%%%%%
\begin{proof}
One constructs the space of differential forms of degree $k$ %by  
\begin{eqnarray}
\label{bomba1}
 \langle w', \Phi \left( dc_1(p_1)^{\wt(v_1)} \otimes v_1, c_1(p_1); \ldots ; dc_n(p_n)^{\wt(v_n)} v_{n}, 
c_n(p_n)\right) \rangle, 
\end{eqnarray} 
by elements $\Phi$ of $C^n_m(V, \W, \F)$  
such that $n=k$ the total degree 
\[
\sum\limits_{i=1}^n \wt(v_i) =  l,    
\]
$v_i \in V$.     
%%
%%%%%%%%%%%%%%%%%%%%%%%%%%%%%%%%%%%%%%%%%%%%%%%%%%%%%%%%%%%%%%%%%%%%%%%%%%
The condition of composability of $\Phi$ with $m$ vertex operators allows us make the association of 
the differential form $\varpi(h_1,\ldots, h_{n})$ with \eqref{bomba1}
$(h^*_1, \ldots, h^*_k)$ with $(v_i, \ldots, v_k)$, and to   
represent 
a sequence of holomorphic embeddings $h_1, \ldots, h_p$ for $U_0, \ldots, U_p$ in \eqref{Cpq} by 
vertex operators $\omega_W$, i.e, 
\[
\left( h(h^*_1) \ldots h(h^*_{n}) \right)(z_1, \ldots, z_n)) 
=  \omega_W\left(v_1, t_1(p_1)  \right) 
\ldots \omega_W\left(v_l, t(p_n) \right).   
\]
Then, by using Definitions of coboundary operator \eqref{deltaproduct},   
  we see that the definition of the coboundary operator of \cite{CM} is parallel to 
the definition \eqref{deltaproduct}.
\end{proof}
%%
%%%%%%%%%%%%%%%%%%%%%%%%%%%%%%%%%%%%%%%%%%%%%%%%%%%%%%%%%%%%%%%%%%%%%%%%%%%%%%%%%%%%%%%%%%%%%%%%%%%%%%
%%%%%%%%%%%%%%%%%%%%%%%%%%%%%%%%%%%%%%%%%%%%%%%%%%%%%%%%%%%%%%%%%%%%%%%%%%%%%%%%%%%%%%%%%%%%%%%%%%%%%5
%%

%%%%%%%%%%%%%%%%%%%%%%%%%%%%%%%%%%%%%%%%%%%%%%%%%%%%%%%%%%%%%%%%%%%%%%%%%%%%%%%%%%%%%%%
%%%%%%%%%%%%%%%%%%%%%%%%%%%%%%%%%%%%%%%%%%%%%%%%%%%%%%%%%%%%%%%%%%%%%%%%%%%%%%%%%%%%%%%
%%%%%%%%%%%%%%%%%%%%%%%%%%%%%%%%%%%%%%%%%%%%%%%%%%%%%%%%%%%%%%%%%%%%%%%%%%%%%%%%%%%%%%%
%%
\section{Properties of the $\epsilon$-product of $C^{k}_{m}(V, \W, \F)$-spaces}
\label{pyhva}
%%
%%%%%%%%%%%%%%%%%%%%%%%%%%%%%%%%%%%%%%%%%%%%%%%%%%%%%%%%%%%%%%%%%%%%%%%%%%%%%%%%%%%%%%%%%%
Since the product of $\Phi(v_{1}, x_{1}$;  $\ldots$; $v_{k}, x_{k}) \in C^{k}_{m}(V, \W, \F)$ and 
$\Psi(v'_{1}, y_{1}; \ldots; v'_{n}, y_{n}) \in C^{n}_{m'}(V, \W, \F)$ results in 
an element of $C^{k+n-r}_{m+m'-t}(V, \W, \F)$, then,  
similar to 
\cite{Huang}, the following 
%%%
corollary follows directly from Proposition \eqref{tolsto} and Definition \ref{sprod}: %, \ref{her}, and \ref{her}:
%%
%%
%%%%%%%%%%%%%%%%%%%%%%%%%%%%%%%%%%%%%%%%%%%%%%%%%%%%%%%%%%%%%%%%%%%%%%%%%%%%%%%%%%%%%%%%%%%%%%%%55
%
%%%%%%%%%%%%%%%%%%%%%%%%%%%%%%%%%%%%%%%%%%%%%%%%%%%%%%%%%%%%%%5
\begin{corollary}
For the spaces $\W_{x_1, \ldots, x_k}$ and $\W_{y_1, \ldots, y_n}$
 with the product \eqref{Z2n_pt_epsss} $\Theta \in \W_{z_1, \ldots, z_{k+n-r}  
}$,   
the subspace of $\hom(V^{\otimes n}, 
%%
%\overline
%%
\W_{z_1, \ldots, z_{k+n-r} } )$   
consisting of linear maps 
having the $L_W(-1)$-derivative property, having the $L_V(0)$-conjugation property
or being composable with $m$ vertex operators is invariant under the 
action of $S_{k+n-r}$. \hfill $\square$
\end{corollary}
%%
%%%%%%%%%%%%%%%%%%%%%%%%%%%%%%%%%%%%%%%%%%%%%%%%%%%%%%%%%%%%%%%%%%%%%%%%%%%
%%%%%%%%%%%%%%%%%%%%%%%%%%%%%%%%%%%%%%%%%%%%%%%%%%%%%%%%%%%%%%%%%%%%%%%%%%%
%%
We also have 
\begin{corollary}
\label{functionformprop}
For a fixed set $(v_1, \ldots v_k; v_{k+1}, \ldots, v_{k+n-r}) \in V$ of vertex algebra elements, and
 fixed $k+n-r$, and $m+m'-t$,  
%%% 
 the $\epsilon$-product 
 $\widehat{R}\; \Theta(v_1, z_1; \ldots; v_k, z_k; v_{k+1}, z_{k+1}; \ldots $ ; $ v_{k+n-r}, y_{k+n-r}; \epsilon)$,  
\[
\cdot_{\epsilon}: C^{k}_m(V, \W, \F) \times C^{n}_{m'}(V, \W, \F) \rightarrow C^{k+n-r}_{m+m'-t}(V, \W, \F), 
\]
of the spaces $C^{k}_{m}(V, \W, \F)$ and $C^{n}_{m'}(V, \W, \F)$, 
for all choices 
of $k$, $n$, $m$, $m'\ge 0$, 
is the same element of $C^{k+n-r}_{m+m'-t}(V, \W, \F)$
for all possible $k \ge 0$. 
\hfill $\square$
\end{corollary}
%%
%%%%%%%%%%%%%%%%%%%%%%%%%%%%%%%%%%%%%%%%%%%%%%%%%%%%%%%%%%%%%%%%%%%%%%%%%%%%%%%%%%%%%%%%%%%%%%%%%%%%55
%%
\begin{proof}
In Proposition \ref{derga} we have proved that the result of the maps belongs to 
$W_{z_1, \ldots, z_{k+n-r}}$, 
 for all $k$, $n \ge 0$, and fixed $k+n-r$.    
%%
%%%%%%%%%%%%%%%%%%%%%%%%%%%%%%%%%%%%%%%%%%%%%%%%%%%%%%%%%%%%%%%%%%%%%%%%%%%%%%%%%%%%%%%%%%%%%%%%%%%%%%%%%555
%%
As in proof of Proposition \ref{tolsto},  
 by checking conditions  for the  forms \eqref{Inms} and \eqref{Jnms},  
we see by Proposition \ref{comp-assoc}, the product 
$\widehat{R} \; \Theta(v_1, x_1; \ldots; v_k, x_k; v'_1, y_1; \ldots; v'_n, y_n)$
 is composable with fixed  $m+m'-t$ number of vertex operators.  %vertex operators for evey $0 \le l \le k$.  
\end{proof}
%%
%%%%%%%%%%%%%%%%%%%%%%%%%%%%%%%%%%%%%%%%%%%%%%%%%%%%%%%%%%%%%%%%%%%%%%%%%%%%%%%%%%%%%%%%
%%
By Proposition \ref{pupa}, elements of the space $\W_{z_1, \ldots, z_{k+n-r}}$  
resulting from the $\epsilon$-product \eqref{Z2n_pt_eps1q1} are  
invariant with respect to independent changes of formal parameters of the group  
$\left({\rm Aut} \;  \Oo\right)^{\times (k+n-r)}_{z_1, \ldots, z_{k+n-r}}$. 
%%
%%%%%%%%%%%%%%%%%%%%%%%%%%%%%%%%%%%%%%%%%%%%%%%%%%%%%%%%%%%%%%%%%%%%%%%%%%%%%%%%%%%%%55
%% 
Now we prove the following  
\begin{corollary}
For $\Phi(v_{1}, x_{1};  \ldots; v_{k}, x_{k}) \in C_{m}^{k}(V, \W, \F)$ and 
$\Psi(v'_{1}, y_{1}; \ldots; v'_{n},y_{n}) \in C_{m'}^{n}(V $, $ \W, \F)$, 
the product 
\begin{eqnarray}
\label{posta}
&& \widehat{ R} \;\Theta \left(
v_{1}, x_{1};  \ldots; v_{k}, x_{k}; v'_{1}, y_{1}; \ldots; v'_{n},y_{n}; 
\epsilon \right) 
\nn
&& \qquad \qquad 
=
\Phi(v_{1}, x_{1};  \ldots; v_{k}, x_{k}) \cdot_\epsilon \Psi (v'_{1}, y_{1}; \ldots; v'_{n},y_{n}), 
%%
%\nn
%&&
%%
\end{eqnarray} 
is canonical with respect to the action 
\begin{eqnarray}
&&
(z_1, \ldots, z_{k+n-r})
\mapsto (z'_1, \ldots, z'_{k+n-r})
=
( \rho(z_1), \ldots, \rho(z_{k+n-r})  ), 
\end{eqnarray}
 of elements the group 
 $\left({\rm Aut} \; \Oo\right)^{\times (k+n-r)}_{z_1, \ldots, z_{k+n-r}}$.  
\end{corollary}
\begin{proof}
In Subsection \ref{properties} we have proved that the product \eqref{Z2n_pt_eps1q1}  
 belongs to $W_{z_1, \ldots, z_{k+n-r}}$, 
and is invariant with  
respect to the group $\left({\rm Aut}\; \Oo \right)^{\times (k+n-r)}_{z_1, \ldots, z_{k+n-r}}$. 
Similar as in the proof of Proposition \ref{nezc},  
 vertex operators $\omega_{V}(v_i, x_i)$, $1\le i \le m$, composable with $\Phi(v_1, x_1; \ldots; v_{k}, x_k)$,  
 and vertex operators $\omega_{V}(v_j, y_j)$, $1 \le j \le m'$, composable with  
$\Psi(v'_{1}, y_{1}; \ldots; v'_{n}, y_{n})$, are also invariant with respect to independent changes of coordinates  
$(\rho(z_1),  \ldots, \rho(z_{k+n-r}))     
\in 
 \left({\rm Aut} \; \Oo \right)^{\times (k+n-r)}_{z_1, \ldots, z_{k+n-r}}$.   
\end{proof}
%%
%%%%%%%%%%%%%%%%%%%%%%%%%%%%%%%%%%%%%%%%%%%%%%%%%%%%%%%%%%%%%%%%%%%%%%%%%%%%%%%%%%%%%%%%
%%%%%%%%%%%%%%%%%%%%%%%%%%%%%%%%%%%%%%%%%%%%%%%%%%%%%%%%%%%%%%%%%%%%%%%%%%%%%%%%%%%%%%%%
%%
\subsection{Coboundary operator acting on the product of elements of  $C^n_m(V, \W, \F)$-spaces
}
%%
%%%%%%%%%%%%%%%%%%
%%
%%%%%%%%%%%%%%%%%%%%%%%%%%%%%%%%%%%%%%%%%%%%%%%%%%%%%%%%%%%%%%%%%%%%%%%%%%%%%%%%%%%%%%%%
 In Proposition \ref{tolsto} we proved that the  product \eqref{Z2n_pt_epsss} of elements of 
  spaces  $C_{m}^{k}(V, \W, \F)$ and $C_{m'}^{n}(V, \W, \F)$ belongs to $C^{k+n-r}_{m+m'-t}(V, \W, \F)$.  
Thus, the product admits the action ot the differential operators $\delta^{k+n-r}_{m+m'-t}$ and $\delta^{2-r}_{ex-t}$ 
defined in 
\eqref{deltaproduct} and \eqref{halfdelta}.  
%%
%%%%%%%%%%%%%%%%%%%%%%%%%%%%%%%%%%%%%%%%%%%%%%%%%%%%%%%%%%%%%%%%%%%%%%%%%%%%%%%5
%%
%%%%%%%%%%%%%%%%%%%%%%%%%%%%%%%%%%%%%%%%%%%%%%%%%%%%%%%%%%%%%%%%%%%%%%%%%%%%%%%%%%%%%%%%%%%%%%%5
%%%%%%%%%%%%%%%%%%%%%%%%%%%%%%%%%%%%%%%%%%%%%%%%%%%%%%%%%%%%%%%%%%%%%%%%%%%%%%%%%%%%%%%%%%%%%%%%
%%
%%%%%%%%%%%%%%%%%%%%%%%%%%%%%%%%%%%%%%%%%%%%%%%%%%%%%%%%%%%%%%%%%%%%%%%%%%%%%%%%%%%%%%%%
%%%%%%%%%%%%%%%%%%%%%%%%%%%%%%%%%%%%%%%%%%%%%%%%%%%%%%%%%%%%%%%%%%%%%%%%%%%%%%%%%%%%%%%%
%%
 The co-boundary operators \eqref{deltaproduct} and \eqref{halfdelta} 
 possesse a variation of Leibniz law with respect to the product 
\eqref{Z2n_pt_epsss}. Indeed, we state here 
%%
%%
%%%%%%%%%%%%%%%%%%%%%%%%%%%%%%%%%%%%%%%%%%%%%%%%%%%%%%%%%%%%%%%%%%%%%%%%%%%%%%%%%
%
%%
%%%%%%%%%%%%%%%%%%%%%%%%%%%%%%%%%%%%%%%%%%%%%%%%%%%%%%%%%%%%%%%%%%%%%%%%%%%%%
%%
\begin{proposition}
\label{tosya}
For $\Phi(v_{1}, x_{1};  \ldots; v_{k}, x_{k}) \in C_{m}^{k}(V, \W, \F)$ 
and 
$\Psi(v'_{1}, y_{1}; \ldots; v'_{n}, y_{n}) \in C_{m'}^{n}(V, \W, \F)$, 
%%
% $w' \in W'$ one has 
%%
the action of the differential $\delta_{m + m'-t}^{k + n-r}$ \eqref{deltaproduct} 
(and $\delta^{2-r}_{ex-t}$ \eqref{halfdelta}) 
 on the $\epsilon$-product \eqref{Z2n_pt_epsss} is given by 
\begin{eqnarray}
\label{leibniz}
&& \delta_{m + m'-t}^{k + n-r} \left(  \Phi (v_{1}, x_{1};  \ldots; v_{k}, x_{k}) 
 \cdot_{\epsilon} \Psi (v'_{1}, y_{1}; \ldots; v'_{n}, y_{n}) \right) %\rangle \right) 
\nn
&&
 \qquad = 
\left( \delta^{k}_{m} \Phi (\widetilde{v}_{1}, z_{1};  \ldots; \widetilde{v}_{k}, z_{k}) \right)  
\cdot_{\epsilon} \Psi (\widetilde{v}_{k+1}, z_{k+1}; \ldots; \widetilde{v}_{k+n-r}, z_{k+n-r})  
\nn
&&
\; + (-1)^k 
\Phi (\widetilde{v}_{1}, z_{1};  \ldots; \widetilde{v}_{k}, z_{k}) \cdot_{\epsilon}   \left( \delta^{n-r}_{m'-t} 
\Psi(\widetilde{v}_{1}, z_{k+1}; \ldots; 
 \widetilde{v}_{k+n-r}, z_{k+n-r})  \right),  
%%
%\nn
%&&
%%
\end{eqnarray}
%%%%%%%%%%%%%%%%%%%%%%%%%%%%%%%%%%%%%%%%%%%%%%%%%%5
%%
where we use the notation as in \eqref{zsto} and \eqref{notari}. 
\end{proposition}
The proof of this proposition is in Appendix \ref{tosya}. 
%%
%%%%%%%%%%%%%%%%%%%%%%%%%%%%%%%%%%%%%%%%%%%%%%%%%%%%%%%%%%%%%%%%%%%%%%%
%%%%%%%%%%%%%%%%%%%%%%%%%%%%%%%%%%%%%%%%%%%%%%%%%%%%%%%%%%%%%%%%%%%%%%
\begin{remark}
Checking 
\eqref{deltaproduct} we see that an extra arbitrary vertex algebra element $v_{n+1} \in V$, as well as corresponding 
 extra arbitrary formal parameter $z_{n+1}$ appear as results of the action of $\delta^{n}_m$ on 
$\Phi \in C^n_m(V, \W, \F)$ mapping it to $C^{n+1}_{m-1}(V, \W, \F)$. 
In application to the $\epsilon$-product \eqref{Z2n_pt_epsss} these extra arbitrary elements are involved in the 
definition of the action of $\delta_{m + m'-t}^{k + n-r}$ on 
 $\Phi (v_{1}, x_{1};  \ldots; v_{k}, x_{k}) 
 \cdot_{\epsilon} \Psi (v'_{1}, y_{1}; \ldots; v'_{n}, y_{n})$.  
\end{remark}
Note that both sides of \eqref{leibniz} belong to the space 
$C_{m + m'-t + 1}^{n + n' -r -1}(V, \W, \F)$.  
The co-boundary operators $\delta^n_m$ and $\delta^{n'}_{m'}$
 in \eqref{leibniz} do not include the number of common vertex algebra elements 
(and formal parameters), neither the number of common vertex operators corresponding mappings composable with.
 The dependence on common vertex algebra elements, parameters, and composable vertex operators is taken into 
account in mappings multiplying the action of co-boundary operators on $\Phi$. 
%%

%%
%%%%%%%%%%%%%%%%%%%%%%%%%%%%%%%%%%%%%%%%%%%%%%%%%%%%%%%%%%%%%%%%%%%%%%%%%%%%%%%%%%%%%
We have the following 
\begin{corollary}
The product \eqref{Z2n_pt_epsss} and the differential operators \eqref{deltaproduct}, \eqref{halfdelta}  
endow the space $C^k_m(V, \W, \F)$ $\times$ $ C^n_{m'}(V, \W, \F)$, $k$, $n \ge0$, $m$, $m' \ge0$, %$n\ge 0$, $m\ge 0$
 with the structure of a %double
double graded differential algebra $\mathcal G\left(V, \W, \cdot_\epsilon, \delta^{k+n-r}_{m+m'-t}\right)$. 
\hfill $\square$
\end{corollary}
%%
%%%%%%%%%%%%%%%%%%%%%%%%%%%%%%%%%%%%%%%%%%%%%%%%%%%%%%%%%%%%%%%%%%%%%%%%%%%%%%%
Finally, we prove the following 
%%
%%%%%%%%%%%%%%%%%%%%%%%%%%%%%%%%%%%%%%%%%%%%%%%%%%%%%%%%%%%%%%%%%%%%%%%%%%%%%%
%%
\begin{proposition}
The multiplication \eqref{Z2n_pt_epsss} extends the chain-cochain 
complexes 
\eqref{hat-complex} and \eqref{hat-complex-half} property \eqref{deltacondition} %structure 
to all products $C^k_m(V, \W, \F) \times C^{n}_{m'}(V, \W, \F)$, 
$k$, $n \ge0$, $m$, $m' \ge0$.  
\end{proposition}
\begin{proof}
For $\Phi \in C^k_m(V, \W, \F)$ and  $\Psi \in C^{n}_{m'}(V, \W, \F)$ 
 we proved in Proposition \ref{tolsto} that the product $\Phi \cdot_\epsilon \Psi$ belongs to the space 
$C^{k+n-r}_{m+m'-t}(V, \W, \F)$. 
Using \eqref{leibniz} and chain-cochain property for $\Phi$ and $\Psi$ we also check that 
\begin{eqnarray}
&& \delta^{k+n+1-r}_{m+m'-1-t} \circ \delta^{k+n-r}_{m+m'-t} \left( \Phi \cdot_\epsilon \Psi\right)=0. 
\nn
&&
\delta^{2-r}_{ex-t} \circ \delta^{1-r}_{2-t} \left( \Phi \cdot_\epsilon \Psi\right)=0. 
\end{eqnarray}
Thus, the the chain-cochain property extends to the product 
$C^k_m(V, \W, \F) \times C^{n}_{m'}(V, \W, \F)$. 
\end{proof}
%%
%%%%%%%%%%%%%%%%%%%%%%%%%%%%%%%%%%%%%%%%%%%%%%%%%%%%%%%%%%%%%%%%%%%%%%%%%%%%%%%%%%%%%%%%%%%%%%%%%%%%%
%%%%%%%%%%%%%%%%%%%%%%%%%%%%%%%%%%%%%%%%%%%%%%%%%%%%%%%%%%%%%%%%%%%%%%%%%%%%%%%%%%%%%%%%%%%%%%%%%%%%%
\subsection{The exceptional complex} 
\label{example}
%%
%%%%%%%%%%%%%%%%%%%%%%%%%%%%%%%%%%%%%%%%%%%%%%%%%%%%%%%%%%%%%%%%%%%%%%%%%%%%%%%%%%%%%%%%%%%%%%
%%
%%%%%%%%%%%%%%%%%%%%%%%%%%%%%%%%%%%%%%%%%%%%%%%%%%%%%%%%%%%%%%%%%%%%%%%%%%%%%%%%%%%%%%%%%%%%%%
%%
For elements of the spaces $C^2_{ex}(V, \W, \F)$ % of the exceptional double complex  
we have the following %Proposition
%%
%%%%%%%%%%%%%%%%%%%%%%%%%%%%%%%%%%%%%%%%%%%%%%%%%%%%%%%%%%%%%%%%%%%%%%%%%%%%%%%%%%%%%%%%%%%%%%%
\begin{corollary}
The product of elements of the spaces $C^{2}_{ex} (V, \W, \F)$ and  $C^n_{m} (V, \W, \F)$ is given by 
\eqref{Z2n_pt_epsss}, %and 
%%
%%%%%%%%%%%%%%%%%%%%%%%%%%%%%%%%%%%%%%%%%%%%%%%%%%%%%%%%%%%%%%%%%%%%%%%%%%%%%%%%%%%% 
%%
\begin{equation}
\label{pupa3}
\cdot_\epsilon: C^{2}_{ex} (V, \W, \F) \times C^n_{m} (V, \W, \F) \to C^{n+2-r}_{m} (V, \W, \F),  
\end{equation}
and, in particular, 
\[
\cdot_\epsilon: C^{2}_{ex} (V, \W, \F) \times C^{2}_{ex} (V, \W, \F) \to C^{4-r}_{0} (V, \W, \F).  
\]
\end{corollary}
\begin{proof}
The fact that the number of formal parameters is $n+2-r$ in the product \eqref{Z2n_pt_epsss} 
follows from  
Proposition \eqref{derga}.  
Consider the product \eqref{Z2n_pt_epsss} for  
$C^{2}_{ex} (V, \W, \F)$ and %\times 
$C^n_{m} (V, \W, \F)$.  
It is clear that, similar to considerations of the proof of Proposition \ref{tolsto}, 
the total number $m$ of vertex operators the product $\Theta$ is composable to remains the same. 
\end{proof}
%%

%%%%%%%%%%%%%%%%%%%%%%%%%%%%%%%%%%%%%%%%%%%%%%%%%%%%%%%%%%%%%%%%%%%%%%%%%%%%%%%%%%%%%%%%%%%%%%%%%%%%%%%%%%%
%%%%%%%%%%%%%%%%%%%%%%%%%%%%%%%%%%%%%%%%%%%%%%%%%%%%%%%%%%%%%%%%%%%%%%%%%%%%%%%%%%%%%%%%%%%%%%%%%%%%%%%%%%%
%%%%%%%%%%%%%%%%%%%%%%%%%%%%%%%%%%%%%%%%%%%%%%%%%%%%%%%%%%%%%%%%%%%%%%%%%%%%%%%%%%%%%%%%%%%%%%%%%%%%%%%%%%%
\section
{Product-type cohomological classes}
\label{gv}
%%%%%%%%%%%%%%%%%%%%%%%%%%%%%%%%%%%%%%%%%%%%%%%%%%%%%%%%%%%%%%%%%%%%%%%%%%%%%%%%%%%%%%%%%%%%%%%%%%%%%%%%%%%
%%%%%%%%%%%%%%%%%%%%%%%%%%%%%%%%%%%%%%%%%%%%%%%%%%%%%%%%%%%%%%%%%%%%%%%%%%%%%%%%%%%%%%%%%%%%%%%%%%%%%%%%%%%
\subsection{The commutator multiplication} 
In this subsection we define further product of pair of elements of spaces $C^k_m(V, \W, \F)$ and $C^n_{m'}(V, \W, \F)$, 
suitable for formulation of cohomological invariants. 
%%
%%
%%%%%%%%%%%%%%%%%%%%%%%%%%%%%%%%%%%%%%%%%%%%%%%%%%%%%%%%%%%%%%%%%%%%%%%%%%%%%%%%%%%%%%%%%%%%%%%%%%%%%%%%
%%
Let us consider the mappings 
\[
\Phi(v_1, z_1  ;   \ldots ;  v_{n}, z_k) \in   C_{m}^{k}(V, \W, \F), %and  
\] 
\[
\Psi (v_{k+1}, z_{k+1};  
\ldots; 
 v_{k+n}, z_{k+n}) \in   C_{m'}^{n}(V, \W, \F),   
\]     
 with $r$ common vertex algebra elements (and, correspondingly, $r$ formal variables), and 
 $t$ common vertex operators mappings $\Phi$ and $\Psi$ are composable with. 
Note that when applying the co-boundary operators \eqref{deltaproduct} and \eqref{halfdelta} to a map 
$\Phi(v_1, z_1; 
 \ldots; 
v_n, z_n)  
\in C^n_m(V, \W, \F)$, 
\[
\delta^n_m: \Phi(v_1, z_1; 
\ldots; 
 v_n,z_n) 
\to 
\Phi(v'_1, z'_1;  
 \ldots; 
v'_{n+1}, z'_{n+1}) 
\in C^{n+1}_{m-1}(V, \W, \F),
\]
 one does not necessary assume that we keep 
 the same set of vertex algebra elements/formal parameters and 
vertex operators composable with for $\delta^n_m \Phi$, 
though it might happen that some of them could be common with $\Phi$.    
Then we have 
%%%%%%%%%%%%%%%%%%%%%%%%%%%%%%%%%%%%%%%%%%%%%%%%%%%%%%%%%%%%%%%%%%%%%%%%%%%%%%%%%%%%%%%%%%%%%%%%
%%
\begin{definition}
\label{coma}
Let us define 
 extra product 
of $\Phi$ and $\Psi$, %naturally coming from the definition of the spaces  
 \begin{eqnarray}
\label{defproduct}
&&
 \Phi \cdot \Psi: V^{\otimes(k +n-r)} \to  \W_{z_1, \ldots, z_{k+ n-r}}, \; 
\\
\label{product1}
&&
 \Phi \cdot \Psi = \left[\Phi,_{\cdot \epsilon} \Psi\right]= \Phi \cdot_\epsilon \Psi- \Psi \cdot_\epsilon \Phi,    
\end{eqnarray}
where brackets denote ordinary commutator in $\W_{z_1, \ldots, z_{k+ n-r}}$.  
\end{definition}
%%%%%%%%%%%%%%%%%%%%%%%%%%%%%%%%%%%%%%%%%%%%%%%%%%%%%%%%%%%%%%%%%%%%%%%%%%%%%%%%%%%%%%5
%
%%%%%%%%%%%%%%%%%%%%%%%%%%%%%%%%%%%%%%%%%%%%%%%%%%%%%%%%%%%%%%%%%%%%%%%%%%%%%%%%%%%%%%%%%%%%%%%
Due to the properties of the maps $\Phi\in C_{m}^{k}(V, \W, \F)$ and 
$\Psi\in   C_{m'}^{n}(V, \W, \F)$, we obtain 
%%
%%%%%%%%%%%%%%%%%%%%%%%%%%%%%%%%%%%%%%%%%%%%%%%%%%%%%%%%%%%%%%%%%%%%%%%%%%
\begin{lemma}
\label{propo}
The product $\Phi \cdot \Psi$
belongs to the space $C_{m + m'- t }^{k +n-r}(V, \W, \F)$.  
For $k=n$ and 
\[
\Psi (v_{n+1}, z_{n+1}; 
 \ldots; 
v_{ 2n}, z_{2n}) 
= 
\Phi(v_{1}, z_1; 
 \ldots;  
v_{ n}, z_n), 
\]
we obtain from \eqref{product} and \eqref{Z2n_pt_epsss} that 
\begin{eqnarray}
\label{fifi}
\Phi(v_{1}, z_1; 
 \ldots;  
v_{ n}, z_n) 
\cdot
 \Phi(v_{1}, z_1; 
\ldots; 
 v_{ n}, z_n) 
=0. 
\end{eqnarray}  
\hfill $\square$
\end{lemma}
The product \eqref{defproduct} will be used in the next subsection in  order to introduce cohomological invariants. 
%%
%%%%%%%%%%%%%%%%%%%%%%%%%%%%%%%%%%%%%%%%%%%%%%%%%%%%%%%%%%%%%%%%%%%%%%%%%%%%%%%%%%%%%%%%%%%%5
%%%%%%%%%%%%%%%%%%%%%%%%%%%%%%%%%%%%%%%%%%%%%%%%%%%%%%%%%%%%%%%%%%%%%%%%%%%%%%%%%%%%%%%%%%%%%
%%
\subsection{Cohomological invariants}
In this subsection, using the vertex algebra double complex construction \eqref{conde}--\eqref{deltacondition},  
 we provide invariants for the grading-restricted vertex algebra cohomology of codimension one foliations 
on complex curves.  
Recall (Appendix \ref{cohomological}) definitions of cohomological classes associated to grading-restricted vertex algebras.  
%%
%
%%%%%%%%%%%%%%%%%%%%%%%%%%%%%%%%%%%%%%%%%%%%%%%%%%%%%%%%%%%%%%%%%%%%%%%%%%%%%%%%%%%%%%%%%%%%%%%%%%%%%%%%%%%
%%%%%%%%%%%%%%%%%%%%%%%%%%%%%%%%%%%%%%%%%%%%%%%%%%%%%%%%%%%%%%%%%%%%%%%%%%%%%%%%%%%%%%%%%%%%%%%%%%%%%%%%%%%
%%
In this subection we consider the general classes of cohomological invariants which arrise from 
 Definition \ref{cvitochki} of a product of pairs of $C^n_m(V, \W, \F)$-spaces. 
%%
%%%%%%%%%%%%%%%%%%%%%%%%%%%%%%%%%%%%%%%%%%%%%%%%%%%%%%%%%%%%%%%%%%%%%%%%%%%%%%%%%%%%%%%%%%%%%%%%%%%%%%%%%
%%
Under a natural extra condition, the double complexes \eqref{hat-complex} and \eqref{hat-complex-half}  
allow us to establish relations among elements of $C^n_m(V, \W, \F)$ spaces. 
By analogy with the notion of integrability for  differential forms \cite{G}, 
 we use here the notion of orthogonality for spaces of a complex. 
%%
%%%%%%%%%%%%%%%%%%%%%%%%%%%%%%%%%%%%%%%%%%%%%%%%%%%%%%%%%%%%%%%%%%%%%%%%%%%%%%%%%%%%%%%%%%%%%%%%%%%%
\begin{definition}
For the double complexes \eqref{hat-complex} and \eqref{hat-complex-half}  
let us require that for a pair of double complex spaces $C_m^k(V, \W, \F)$ and $C_{m'}^n(V, \W, \F)$, 
 there exist subspaces 
\[
\widetilde{C}_m^k(V, \W, \F)\subset C_m^k(V, \W, \F),
\]
%%
 %and  
%%
\[
\widetilde{C}_{m'}^n(V, \W, \F)\subset C_{m'}^n(V, \W, \F),
\]
 such that for all $\Phi \in \widetilde{C}_m^k(V, \W, \F)$ and all $\Psi \in \widetilde{C}_{m'}^n(V, \W, \F)$, 
\begin{equation}
\label{ortho}
\Phi \cdot \delta^n_{m'} \Psi=0, 
\end{equation}
 namely,  
$\Phi$ supposed to be orthogonal to $\delta^n_{m'}\Psi$ with respect to the product  
\eqref{product}.  
We call this { the orthogonality condition} for mappings of double complexes 
\eqref{hat-complex} and \eqref{hat-complex-half}.
\end{definition}
Note that in the case of differential forms considered on a smooth manifold, 
the Frobenius theorem for a distribution provides the orthogonality condition \cite{G}.   
The fact that both sides of \eqref{uravnenie} (see below) belong to the same double complex space, apply limitations 
to possible combinations of $(k, m)$ and $(n, m')$.
Below
%In this Section
%%
 we derive algebraic relations occurring from the orthogonality condition on 
the double bicomplexes \eqref{hat-complex} and \eqref{hat-complex-half}. 
%%

%%%%%%%%%%%%%%%%%%%%%%%%%%%%%%%%%%%%%%%%%%%%%%%%%%%%%%%%%%%%%%%%%%%%%%%%%%%%%%%%%%%%%%%%%%%%%%%%%%%%%%%%%%%%%%
%

%%
Taking into account the correspondence %\eqref{cores} 
(see Subsection \ref{defcohomology})  
with the ${\rm   \check {C}}$ech-de Rham complex due to \cite{CM}, 
we reformulate the derivation of product-type invariant 
in vertex algebra terms. 
Recall that the Godbillon--Vey cohomological class \cite{G} is considered on codimension one foliations of 
three-dimensional smooth manifolds. % of dimension three. 
%%
%%%%%%%%%%%%%%%%%%%%%%%%%%%%%%%%%%%%%%%%%%%%%%%%%%%%%%%%%%%%%%%%%%%%%%%%%%%%%%%%%%%%%%%%%%%%%%%%%%%%
%%
%%
In this paper, we supply its analogue for complex curves.  
%%%%%%%%%%%%%%%%%%%%%%%%%%%%%%%%%%%%%%%%%%%%%%%%%%%%%%%%%%%%%%%%%%%%%%%%%%%%%%%%%%%%%%%%%%%%%%%%%%%%
%%
 According to Definition \ref{ourbicomplex} we 
have $k$-tuples of one-dimesional transversal sections. In each section we attach one vertex operator
$Y_W(u_k, w_k)$, $u_k \in V$, $w_k \in U_k$. 
Similarly to differential forms setup, a mapping $\Phi  \in  C_{k}^{m} (V, \W, \F)$ defines 
a codimension one foliation. As we see from \eqref{deltaproduct}, \eqref{fifi}, and \eqref{leibniz} 
it satisfyies properties 
simmilar to differential forms.

Now we show that the analog of the integrability condition provide a generalization of product-type invariant 
for codimension one foliations on complex curves. 
%%
%We formulate the main proposition of this paper: 
%%
%%
Here we formulate the main statement of this paper: 
%%
%%%%%%%%%%%%%%%%%%%%%%%%%%%%%%%%%%%%%%%%%%%%%%%%%%%%%%%%%%%%%%%%%%%%%%%%%%%%%%%%%%%%%%%%%%%%%%%%%%%%%%%
%%%%%%%%%%%%%%%%%%%%%%%%%%%%%%%%%%%%%%%%%%%%%%%%%%%%%%%%%%%%%%%%%%%%%%%%%%%%%%%%%%%%%%%%%%%%%5
\begin{proposition}
The set of commutation relations 
The product \eqref{Z2n_pt_epsss}, the differential operators \eqref{deltaproduct}, \eqref{halfdelta}, 
and the orthogonality condition \eqref{ortho} applied to the double complexes \eqref{hat-complex} and 
\eqref{hat-complex-half}  
generate non-vanishing cohomology classes 
\[
\left[\left(\delta^{1}_{2}  \Phi\right) \cdot \Phi \right], \quad   
  \left[\left(\delta^{0}_{3} \Lambda \right)\cdot \Lambda\right], \quad %and 
 \left[\left(\delta^{1}_{t} \Psi \right)\cdot \Psi\right], 
\]
independent on the choice of $\Phi \in   C^1_2(V, \W, \F)$, 
$ \Lambda \in   C^{0}_{3}(V, \W, \F)$, and $ \Psi \in   C^{1}_{t}(V, \W, \F)$.   
\end{proposition}
\begin{remark}
In this paper we provide results concerning complex curves, i.e., the case $n \le 1$, $n_0 \le 1$, $n_i \le 1$.
 They 
generalize to the case of higher dimensional complex manifolds. 
\end{remark}

\begin{proof}
%%
%%
%%%%%%%%%%%%%%%%%%%%%%%%%%%%%%%%%%%%%%%%%%%%%%%%%%%%%%%%%%%%%%%%%%%%%%%%%%%%%%%%%%%%%%%%%%%%%%%%%%%%%%
%%
%%
%%
%%
 Let us consider two maps $\Phi(v_1) \in C^{1}_{2}(V, \W, \F)$ and $\Lambda \in C^{0}_{3}(V, \W, \F)$. 
We require them to be orthogonal, i.e.,
\begin{equation}
\label{isxco}
\Phi \cdot \delta^{0}_{3} \Lambda=0.   
\end{equation}
Thus, there exists $\Psi(v_2) \in C^{n}_{m}(V, \W, \F)$, such that 
\begin{equation}
\label{uravnenie}
\delta^{0}_{3} \Lambda= \Phi \cdot \Psi, 
\end{equation} 
and 
$1=1 + n -r$, $2=2+m-t$, i.e., $n=r$,  which leads to $r=1$;  $m=t$, $0\le t \le 2$, i.e., 
$\Psi \in C^{1}_{t}(V, \W, \F)$. 
%%%%%%%%
Here $r$ and $t$ are correspondingly numbers of common vertex algebra element (and formal parameteres)
 and vertex operators 
a map composable with.  
All other orthogonality conditions for the short sequence \eqref{hat-complex-half} does not allow relations of the form
 \eqref{uravnenie}. 

%%%%%%%%%%%%%%%%%%%%%%%%%%%%%%%%%%%%%%%%%%%%%%%%%%%%%%%%%%%%%%%%%%%%%%%%%%%%%%%%%%%%%%%%%%%%%%%%%%%%%%%%%%%%%%
%%
Consider now \eqref{isxco}.
We obtain, using \eqref{leibniz}
\[
 \delta^{2-r'}_{4-t'} (\Phi \cdot \delta^{0}_{3} \Lambda)=  
\left(\delta^{1}_{2} \Phi\right) \cdot \delta^{0}_{3} \Lambda + \Phi \cdot \delta^{1}_{2} \delta^{0}_{3} \Lambda=
 \left(\delta^{1}_{2} \Phi\right) \cdot \delta^{0}_{3} \Lambda= \left(\delta^{1}_{2} \Phi\right) \cdot \Phi \cdot 
\Psi. 
\]
Thus 
\[
 0=\delta^{3-r'}_{3-t'} \delta^{2-r'}_{4-t'} (\Phi \cdot \delta^{0}_{3} \Lambda)=  
\delta^{3-r'}_{3-t'}  \left( \left(\delta^{1}_{2} \Phi\right) \cdot \Phi \cdot 
\Psi.  \right), 
\]
and $\big(\big(\delta^{1}_{2} \Phi\big) \cdot \Phi \cdot \Psi\big) \big)$ is closed. 
At the same time, 
from \eqref{isxco} %$\Phi \cdot \delta^{0}_{3} \Lambda=0$,
 it follows that 
\[
0=\delta^{1}_{2} \Phi \cdot \delta^{0}_{3} \Lambda- \Phi \cdot\delta^{1}_{2}\delta^{0}_{3} \Lambda 
= \big( \Phi\cdot \delta^{0}_{3} \Lambda \big).
\] 
Thus 
\[
\delta^{1}_{2} \Phi \cdot \delta^{0}_{3} \Lambda= \delta^{1}_{2} \Phi \cdot \Phi \cdot \Psi =0.
\] 
Consider \eqref{uravnenie}. 
Acting by $\delta^{1}_{2}$ and substituting back we obtain 
\[
0= \delta^{1}_{2} \delta^{0}_{3} \Lambda= \delta^{1}_{2}(\Phi \cdot \Psi)=  
\delta^{1}_{2}(\Phi) \cdot \Psi - \Phi \cdot \delta^{1}_{t} \Psi. 
\]
thus 
\[
\delta^{1}_{2}(\Phi) \cdot \Psi = \Phi \cdot \delta^{1}_{t} \Psi. 
\]
The last equality trivializes on applying $\delta^{3}_{t+1}$ to both sides. 
%%

%%%%%%%%%%%%%%%%%%%%%%%%%%%%%%%%%%%%%%%%%%%%%%%%%%%%%%%%%%%%%%%%%%%%%%%%%%%%%%%%%%%%%%%%%%%%%%%%%%%%%%%%%%%%
Let us show now the non-vanishing property of $\left(\left(\delta^{1}_{2} \Phi \right)\cdot \Phi\right)$. 
Indeed, suppose 
\[
\left(\delta^{1}_{2} \Phi \right)\cdot \Phi=0.
\]
 Then there exists $\Gamma \in C^{n}_{m}(V, \W, \F)$, 
such that 
\[
\delta^{1}_{2} \Phi =\Gamma \cdot \Phi.
\]
 Both sides of the last equality should belong to the same double complex 
space but one can see that it is not possible. 
Thus, $\left(\delta^{1}_{2} \Phi \right)\cdot \Phi$ is non-vanishing. 
 One proves in the same way that $\left(\delta^{0}_{3} \Lambda \right)\cdot \Lambda$ and 
$\left(\delta^{1}_{t} \Psi \right)\cdot \Psi$ do not vanish too.  
Now let us show  that $\left[\left(\delta^{1}_{2} \Phi \right)\cdot \Phi \right]$ 
 is invariant, i.e., it does not depend on the choice of $\Phi \in C^1_2(V, \W, \F)$.
 Substitute $\Phi$ by 
$\left(\Phi + \eta\right)\in C^{1}_{2}(V, \W, \F)$.   
We have 
\begin{eqnarray}
\label{pokaz}
\nonumber
\left(\delta^{1}_{2} \left( \Phi + \eta \right) \right) \cdot \left( \Phi  + \eta \right) &=& 
\left(\delta^{1}_{2} \Phi\right) \cdot \Phi 
 + \left( \left(\delta^{1}_{2} \Phi \right)\cdot \eta 
-  \Phi \cdot \delta^{1}_{2} \eta  \right)  
\nn
&+& \left( \Phi \cdot \delta^{1}_{2} \eta   + 
\delta^{1}_{2} \eta  \cdot \Phi \right) 
 +
\left(\delta^{1}_{2} \eta \right) \cdot \eta. 
\end{eqnarray}
Since
\[
\left( \Phi \cdot \delta^{1}_{2} \eta  + 
\left(\delta^{1}_{2} \eta\right)  \cdot \Phi \right)= 
\Phi \cdot_\epsilon \delta^{1}_{2} \eta -  (\delta^{1}_{2} \eta) \cdot_\epsilon \Phi 
+ \left(\delta^{1}_{2} \eta\right) \cdot_\epsilon \Phi - \Phi \cdot_\epsilon  \delta^{1}_{2} \eta=0, 
\]
then \eqref{pokaz} represents the same cohomology class 
$\left[ \left(\delta^{1}_{2}  \Phi \right) \cdot \Phi \right]$. 
The same folds for $\left[\left(\delta^{0}_{3} \Lambda \right)\cdot \Lambda \right]$, and 
$\left[\left(\delta^{1}_{t} \Psi \right)\cdot \Psi \right]$. 
%% 
%\bigskip 
%% 
%%
%%%%%%%%%%%%%%%%%%%%%%%%%%%%%%%%%%%%%%%%%%%%%%%%%%%%%%%%%%%%%%%%%%%%%%%%%%%%%%%%%%%%%%%%%%%%%%%%%%%%%%%%%%%%%
%
%%
%%
\end{proof}
%%

%%%%%%%%%%%%%%%%%%%%%%%%%%%%%%%%%%%%%%%%%%%%%%%%%%%%%%%%%%%%%%%%%%%%%%%%%%%%%%%%%%%%%%%%%%%%%%
\section*{Acknowledgements}
The author would like to thank A. Galaev, Y.-Zh. Huang, H. V. L\^e, A. Lytchak, and P. Somberg, 
 for related discussions. 
Research of the author was supported by the GACR project 18-00496S and RVO: 67985840. 
%% 
%%%%%%%%%%%%%%%%%%%%%%%%%%%%%%%%%%%%%%%%%%%%%%%%%%%%%%%%%%%%%%%%%%%%%%%%%%%%%%%%%%%%%%%%%%
%%%%%%%%%%%%%%%%%%%%%%%%%%%%%%%%%%%%%%%%%%%%%%%%%%%%%%%%%%%%%%%%%%%%%%%%%%%%%%%%%%%%%%%%%%
\section{Appendix: Grading-restricted vertex algebras and their modules}
\label{grading}
In this Section, following \cite{Huang} we recall basic properties of 
grading-restricted vertex algebras. 
 and their grading-restricted generalized 
modules, useful for our purposes in later sections. 
We work over the base field $\C$ of complex numbers. 
%%%%%%%%%%%%%%%%%%%%%%%%%%%%%%%%%%%%%%%%%%%%%%%%%%%%%%%%%%%%%%%%%%%%%%%%%%%%%%%%%%%%%%%%%%5
%%%%%%%%%%%%%%%%%%%%%%%%%%%%%%%%%%%%%%%%%%%%%%%%%%%%%%%%%%%%%%%%%%%%%%%%%%%%%%%%%%%%%%%%%
\subsection{Grading-restricted vertex algebras}
\begin{definition}
A vertex algebra   
$(V,Y_V,\mathbf{1}_V)$, (cf. \cite{K}),  consists of a $\Z$-graded complex vector space  
\[
V = \coprod_{n\in\Z}\,V_{(n)}, \quad \dim V_{(n)}<\infty, 
\]
 for each $n\in \Z$,   
and linear map 
\[
Y_V:V\rightarrow {\rm End \;}(V)[[z,z^{-1}]], 
\]
 for a formal parameter $z$ and a 
distinguished vector $\mathbf{1}_V\in V$.   
The evaluation of $Y_V$ on $v\in V$ is the vertex operator
\begin{equation}
\label{vop}
Y_V(v)\equiv Y_V(v,z) = \sum_{n\in\Z}v(n)z^{-n-1}, 
\end{equation}
with components $(Y_{V}(v))_{n}=v(n)\in {\rm End \;}(V)$, where $Y_V(v,z)\mathbf{1}_V = v+O(z)$.
\end{definition}
\begin{definition}
\label{grares}
A grading-restricted vertex algebra satisfies 
the following conditions:
\begin{enumerate}
\item {Grading-restriction condition}:
$V_{(n)}$ is finite dimensional for all $n\in \Z$, and $V_{(n)}=0$ for $n\ll 0$; 

\item { Lower-truncation condition}:
For $u$, $v\in V$, $Y_{V}(u, z)v$ contains only finitely many negative 
power terms, that is, 
\[
Y_{V}(u, z)v\in V((z)), 
\] 
(the space of formal 
Laurent series in $z$ with coefficients in $V$);   

\item { Identity property}: 
Let ${\rm Id}_{V}$ be the identity operator on $V$. Then
\[
Y_{V}(\mathbf{1}_V, z)={\rm Id}_{V};  
\]

\item { Creation property}: For $u\in V$, 
\[
Y_{V}(u, z)\mathbf{1}_V\in V[[z]],
\]  
and 
\[
\lim_{z\to 0}Y_{V}(u, z)\mathbf{1}_V= u;
\]

\item { Duality}: 
For $u_{1}, u_{2}, v\in V$, 
\[
v'\in V'=\coprod_{n\in \mathbb{Z}}V_{(n)}^{*}, 
\]
where 
 $V_{(n)}^{*}$ denotes
the dual vector space to $V_{(n)}$ and $\langle\, . ,  .\rangle$ the evaluation 
pairing $V'\otimes V\to \C$, the series 
\begin{eqnarray}
\label{porosyata}
& & \langle v', Y_{V}(u_{1}, z_{1})Y_{V}(u_{2}, z_{2})v\rangle,
\\
& & \langle v', Y_{V}(u_{2}, z_{2})Y_{V}(u_{1}, z_{1})v\rangle, %\quad\mbox{and}  
\\
& & \langle v', Y_{V}(Y_{V}(u_{1}, z_{1}-z_{2})u_{2}, z_{2})v\rangle, 
\end{eqnarray}
are absolutely convergent
in the regions 
\[
|z_{1}|>|z_{2}|>0,
\]
\[ 
|z_{2}|>|z_{1}|>0,
\]
\[
|z_{2}|>|z_{1}-z_{2}|>0,
\]
 respectively, to a common rational function 
in $z_{1}$ and $z_{2}$ with the only possible poles at $z_{1}=0=z_{2}$ and 
$z_{1}=z_{2}$; 

\item { $L_V(0)$-bracket formula}: Let $L_{V}(0): V\to V$,  
be defined by 
\[
L_{V}(0)v=nv, \qquad n=\wt(v),  
\]
 for $v\in V_{(n)}$.  
Then
\[
[L_{V}(0), Y_{V}(v, z)]=Y_{V}(L_{V}(0)v, z)+z\frac{d}{dz}Y_{V}(v, z), 
\]
for $v\in V$. 

\item { $L_V(-1)$-derivative property}:  
Let 
\[
L_{V}(-1): V\to V, 
\]
 be the operator given by 
\[
L_{V}(-1)v=\res_{z}z^{-2}Y_{V}(v, z)\one_V=Y_{(-2)}(v) \one_V,  
\]
for $v\in V$. Then for $v\in V$, 
\begin{equation}
\label{derprop}
\frac{d}{dz}Y_{V}(u, z)=Y_{V}(L_{V}(-1)u, z)=[L_{V}(-1), Y_{V}(u, z)].
\end{equation}
\end{enumerate}
\end{definition}
%%%%%%%%%%%%%%%%%%%%%%%%%%%%%%%%%%%%%%%%%%%%%%%%%%%%%%%%%%%%%%%%%%%%%%%%%%%%%%%%%%%%
In addition to that,  we recall here the following definition (cf. \cite{BZF}): 
\begin{definition}
 A grading-restricted vertex algebra $V$ is called conformal of central 
charge $c \in \C$,
 if there exists a non-zero conformal vector (Virasoro vector) $\omega \in V_{(2)}$ such that the
corresponding vertex operator 
\[
Y_V(\omega, z)=\sum_{n\in\Z}L_V(n)z^{-n-2}, 
\]
is determined by modes of Virasoro algebra $L_V(n): V\to V$ satisfying 
\[
[L_V(m), L_V(n)]=(m-n)L(m+n)+\frac{c}{12}(m^{3}-m)\delta_{m+b, 0}\; {\rm Id_V}. 
\]
\end{definition}
\begin{definition}
\label{primary}  
A vector $A$ which belongs to a module $W$ of a quasi-conformal 
 grading-restricted vertex algebra $V$ is called 
primary of conformal dimension $\Delta(A) \in  \mathbb Z_+$ if  
\begin{eqnarray*}
L_W(k) A &=& 0,\;  k > 0, 
\nn
 L_W(0) A &=& \Delta(A) A. 
\end{eqnarray*}
\end{definition}
%%
%%%%%%%%%%%%%%%%%%%%%%%%%%%%%%%%%%%%%%%%%%%%%%%%%%%%%%%%%%%%%%%%%%%%%%%%%%%%%%%%%%%
%%%%%%%%%%%%%%%%%%%%%%%%%%%%%%%%%%%%%%%%%%%%%%%%%%%%%%%%%%%%%%%%%%%%%%%%%%%%%%%%%%5
\subsection{Grading-restricted generalized $V$-module}
In this subsection we describe the grading-restricted generalized $V$-module
 for a grading-restricted vertex algebra $V$. 
\begin{definition}
A {grading-restricted generalized $V$-module} is a vector space 
$W$ equipped with a vertex operator map 
\begin{eqnarray*}
Y_{W}: V\otimes W&\to& W[[z, z^{-1}]],
\nn
u\otimes w&\mapsto & Y_{W}(u, w)\equiv Y_{W}(u, z)w=\sum_{n\in \Z}(Y_{W})_{n}(u,w)z^{-n-1}, 
\end{eqnarray*}
and linear operators $L_{W}(0)$ and $L_{W}(-1)$ on $W$ satisfying the following
conditions:
\begin{enumerate}
\item {Grading-restriction condition}:
The vector space $W$ is $\mathbb C$-graded, that is, 
\[
W=\coprod_{\alpha\in \mathbb{C}}W_{(\alpha)},
\]
 such that 
$W_{(\alpha)}=0$ when the real part of $\alpha$ is sufficiently negative; 

\item { Lower-truncation condition}:
For $u\in V$ and $w\in W$, $Y_{W}(u, z)w$ contains only finitely many negative 
power terms, that is, $Y_{W}(u, z)w\in W((z))$; 

\item { Identity property}: 
Let ${\rm Id}_{W}$ be the identity operator on $W$. 
Then 
\[
Y_{W}(\mathbf{1}_V, z)={\rm Id}_{W}; 
\] 

\item { Duality}: For $u_{1}, u_{2}\in V$, $w\in W$, 
\[
w'\in W'=\coprod_{n\in \mathbb{Z}}W_{(n)}^{*}, 
\]
 $W'$ denotes 
the dual $V$-module to $W$ and $\langle\, .,. \rangle$ their evaluation 
pairing, the series 
\begin{eqnarray}
\label{porosyataw}
&& \langle w', Y_{W}(u_{1}, z_{1})Y_{W}(u_{2}, z_{2})w\rangle,
\\ %%\quad 
&& \langle w', Y_{W}(u_{2}, z_{2})Y_{W}(u_{1}, z_{1})w\rangle, 
\\
&& \langle w', Y_{W}(Y_{V}(u_{1}, z_{1}-z_{2})u_{2}, z_{2})w\rangle, 
\end{eqnarray}
are absolutely convergent
in the regions
\[ 
|z_{1}|>|z_{2}|>0,
\]
\[ 
|z_{2}|>|z_{1}|>0, 
\]
\[
|z_{2}|>|z_{1}-z_{2}|>0,
\]
 respectively, to a common rational function 
in $z_{1}$ and $z_{2}$ with the only possible poles at $z_{1}=0=z_{2}$ and 
$z_{1}=z_{2}$. 
\item { $L_{W}(0)$-bracket formula}: For  $v\in V$,
\begin{equation}
\label{locomm}
[L_{W}(0), Y_{W}(v, z)]=Y_{W}(L_V(0)v, z)+z\frac{d}{dz}Y_{W}(v, z); 
\end{equation}
\item { $L_W(0)$-grading property}: For $w\in W_{(\alpha)}$, there exists
$N \in \Z_{+}$ such that 
\begin{equation}
\label{gradprop}
(L_{W}(0)-\alpha)^{N}w=0;  
\end{equation}
\item { 
$L_W(-1)$-derivative property}: For $v\in V$,
\begin{equation}
\label{derprop}
\frac{d}{dz}Y_{W}(u, z)=Y_{W}(L_{V}(-1)u, z)=[L_{W}(-1), Y_{W}(u, z)].
\end{equation}
\end{enumerate}
\end{definition}
The translation property of vertex operators 
\begin{equation}
\label{transl}
 Y_{W}(u, z) = e^{-z' L_{W}(-1)} Y_{W}(u, z+z') e^{z' L_{W}(-1)}, 
\end{equation}
for $z' \in \C$, follows from from \eqref{derprop}.  %the last requirement of Definition \ref{modul}.  
%%
%%%%%%%%%%%%%%%%%%%%%%%%%%%%%%%%%%%%%%%%%%%%%%%%%%%%%%%%%%%%%%%%%%%%%%%%%%%%%%%%%%%%%%%
%%%%%%%%%%%%%%%%%%%%%%%%%%%%%%%%%%%%%%%%%%%%%%%%%%%%%%%%%%%%%%%%%%%%%%%%%%%%%%%%%%%%%%%
%%
For $v\in V$, and $w \in W$, the intertwining operator 
\begin{eqnarray}
\label{interop}
&& Y_{WV}^{W}: V\to W,  
\nn
&&
v   \mapsto  Y_{WV}^{W}(w, z) v,    
\end{eqnarray}
 is defined by 
\begin{eqnarray}
\label{wprop}
Y_{WV}^{W}(w, z) v= e^{zL_W(-1)} Y_{W}(v, -z) w. 
\end{eqnarray}
%%%
%%
%%%%%%%%%%%%%%%%%%%%%%%%%%%%%%%%%%%%%%%%%%%%%%%%%%%%%%%%%%%%%%%%%%%%%%%%%%%%%%%%%%%%%%%%%%
For $a\in \C$, 
 the conjugation property with respect to the grading operator $L_W{(0)}$ is given by 
%%
%%%%%%%%
%%
%
\begin{equation}
\label{aprop}
 a^{ L_W{(0)} } \; Y_W(v,z) \; a^{-L_W{(0)} }= Y_W (a^{ L_W{(0)} } v, az).    
\end{equation}
%%
%%
%\]
%%
%%
%%%%%%%%%%%%%%%%%%%%%%%%%%%%%%%%%%%%%%%%%%%%%%%%%%%%%%%%%%%%%%%%%%%%%%%%%%%%%%%%%%%%%%%%%%%%%%%
%%%%%%%%%%%%%%%%%%%%%%%%%%%%%%%%%%%%%%%%%%%%%%%%%%%%%%%%%%%%%%%%%%%%%%%%%%%%%%%%%%%%%%%%%%%%%%
\subsection{Generators of Virasoro algebra and the group of 
automorphisms}
Let us recall some further facts from \cite{BZF} relating generators of Virasoro algebra with the group of 
automorphisms in complex dimension one. 
 Let us represent an element of ${\rm Aut}_z \; \Oo^{(1)}$ by the map  
\begin{equation}
\label{lempa}
z \mapsto \rho=\rho(z),
\end{equation}
given by the power series
\begin{equation}
\label{prostoryad}
\rho(z) = \sum\limits_{k \ge 1} a_k z^k, 
\end{equation}
%%
%%%%%%%%%%%%%%%%%%%%%%%%%%%%%%%%%%%%%%%%%%%%%%%%%%%%%%%%%%%%%%%%%%%%%%%%%%%%%%%%%%%%%%%%%%%%%%%%%%%
%%
%
$\rho(z)$ can be represented in an exponential form 
\begin{equation}
\label{rog}
f(z) = \exp \left(  \sum\limits_{k > -1} \beta_{k }\; z^{k+1} \partial_{z} \right) 
\left(\beta_0 \right)^{z \partial_z}.z, 
\end{equation}
where we express $\beta_k \in \mathbb C$, $k \ge 0$, through combinations of $a_k$, $k\ge 1$.  %\cite{BZF}. 
 A representation of Virasoro algebra modes in terms of differenatial operators is given by \cite{K} 
\begin{equation}
\label{repro}
L_W(m) \mapsto - \zeta^{m+1}\partial_\zeta, 
\end{equation}
for $m \in \Z$. 
 By expanding \eqref{rog} and comparing to \eqref{prostoryad} we obtain a system of equations which, 
 can be solved recursively for all $\beta_{k}$.
In \cite{BZF}, $v \in V$, they derive the formula 
\begin{eqnarray}
\label{infaction}
&&
\left[L_W(n), Y_W (v, z) \right] 
%% 
%%\\
%%&&
=  \sum_{m \geq -1}  
 \frac{1}{(m+1)!} \left(\partial^{m+1}_{z} z^{m+1}  \right)\;  
Y_W (L_V(m) v, z),    
\end{eqnarray}
of a Virasoro generator commutation with a vertex operator. 
Given a vector field 
\begin{equation}
\label{top}
\beta(z)\partial_z= \sum_{n \geq -1} \beta_n z^{n+1} \partial_z, 
\end{equation}
which belongs to local Lie algebra of ${\rm Aut}_z\; \Oo^{(1)}$, 
 one introduces the operator 
\[
\beta = - \sum_{n \geq -1} \beta_n L_W(n). 
\]
We conlclude from \eqref{top} with the following 
\begin{lemma}
\begin{eqnarray}
\label{infaction000}
&&
\left[\beta, Y_W (v, z) \right] 
=  \sum_{m \geq -1}  
 \frac{1}{(m+1)!} \left(\partial^{m+1}_{z} \beta(z)  \right)\;  
Y_W (L_V(m) v, z).   
\end{eqnarray}
\end{lemma}
The formula \eqref{infaction000}  is used in \cite{BZF} (Chapter 6) in order to prove invariance of 
vertex operators multiplied by conformal weight differentials in case of primary states, and 
in generic case.  
%%

%%%%%%%%%%%%%%%%%%%%%%%%%%%%%%%%%%%%%%%%%%%%%%%%%%%%%%%%%%%%%%%%%%%%%%%%%%%%%%%%%%%%%%%%%%%%%%%%%%%%%%%%
%%%%%%%%%%%%%%%%%%%%%%%%%%%%%%%%%%%%%%%%%%%%%%%%%%%%%%%%%%%%%%%%%%%%%%%%%%%%%%%%%%%%%%%%%%%%%%%%%%%%%%%%
Let us give some further definition: 
\begin{definition}
\label{quasiconf}
A grading-restricted vertex algebra $V$-module $W$ is called quasi-conformal if 
  it carries an action of local Lie algebra of ${\rm Aut}_z\; \Oo$  
such that commutation formula
 \eqref{infaction000}    
 holds for any 
$v \in  V$, the element 
\[
L_W(-1) = - \partial_{z}, 
\]
as the translation operator $T$,
\[
L_W(0) = - z \partial_{z},  
\]
  acts semi-simply with integral
eigenvalues, and the Lie subalgebra of the positive part of local Lie algebra of ${\rm Aut}_z\; \Oo^{(n)}$
 acts locally nilpotently. 
\end{definition}
%%
%%%%%%%%%%%%%%%%%%%%%%%%%%%%%%%%%%%%%%%%%%%%%%%%%%%%%%%%%%%%%%%%%%%%%%%%%%%%%%%%%%%%%%%%%%%%%%%%%%%%%%%%%%
%%
Recall \cite{BZF} the exponential form $f(\zeta)$ \eqref{rog} of the coordinate transformation \eqref{lempa} 
$\rho(z) \in {\rm Aut}_z\; \Oo^{(1)}$.
A quasi-conformal vertex algebra posseses the formula \eqref{infaction000}, thus 
it is possible 
by using the identification \eqref{repro}, to introduce the linear operator 
 representing $f(\zeta)$ \eqref{rog} on $\W_{z_1, \ldots, z_n}$,  
\begin{equation}
\label{poperator}
 P\left(f (\zeta) \right)= 
\exp \left( \sum\limits_{m >  0} (m+1) \; \beta_m \; L_V(m)\right) \beta_0^{L_W(0)},  
\end{equation}
 (note that we have a different normalization in it). 
In \cite{BZF} (Chapter 6) it was shown that  
the action of an operator similar to \eqref{poperator} 
on a vertex algebra element $v\in V_n$ contains finitely meny terms, and 
subspaces 
\[
V_{\le m} = \bigoplus_{ n \ge K}^m V_n, 
\]
 are stable under all operators $P(f)$, $f \in  {\rm Aut}_z\; \Oo^{(1)}$. 
 In \cite{BZF} they proved the following 
\begin{lemma}
 The assignment
\[
 f \mapsto P(f), 
\]
 defines a representation of ${\rm Aut}_z\; \Oo^{(1)}$
on $V$, 
\[
 P(f_1 * f_2) = P(f_1) \; P(f_2),  
\]
 which is the inductive limit of the
representations $V_{\le m}$, $m\ge K$.
\end{lemma}
Similarly, \eqref{poperator} provides a representation operator on $\W_{z_1, \ldots, z_n}$. 
%%
 
%%
%%%%%%%%%%%%%%%%%%%%%%%%%%%%%%%%%%%%%%%%%%%%%%%%%%%%%%%%%%%%%%%%%%%%%%%%%%%%%%%%%%%%%%%%%%%%%%%%%%
%%%
\subsection{Non-degenerate  invariant bilinear form on $V$} %The Li--Zamolodchikov (Li--Z) Metric}
\label{liza}
The subalgebra 
\[
\{L_V(-1),L_V(0),L_V(1)\}\cong SL(2,\mathbb{C}), 
\]
 associated with M\"{o}bius transformations on  
$z$ naturally acts on $V$,   (cf., e.g. \cite{K}). 
In particular, 
\begin{equation}
\gamma_{\lambda}=\left(
\begin{array}{cc}
0 & \lambda\\
-\lambda & 0\\	
\end{array}
\right)
:z\mapsto w=-\frac{\lambda^{2}}{z},
 \label{eq: gam_lam}
\end{equation}
is generated by 
\[
T_{\lambda }= \exp\left(\lambda L_V{(-1)}\right) 
%5
\; \exp\left({\lambda}^{-1}L_V(1)\right) \; \exp\left(\lambda L_V(-1)\right),   
\]
 where
\begin{equation}
T_{\lambda }Y(u,z)T_{\lambda }^{-1}=
Y\left(\exp \left(-\frac{z}{\lambda^{2}}L_V(1)\right)
\left(-\frac{z}{\lambda}\right)^{-2L_V(0)}u,-\frac{\lambda^{2}}{z}\right).  \label{eq: Y_U}
\end{equation}
In our considerations (cf. Appendix \ref{sphere}) of Riemann sphere %surface 
sewing, we use in particular, 
the M\"{o}bius map 
\[
z\mapsto z'= \epsilon/z,
\] 
 associated with the sewing condition \eqref{pinch} with 
\begin{equation}
\lambda=-\xi\epsilon^{\frac{1}{2}},
\label{eq:lamb_eps}
\end{equation}  
with $\xi\in\{\pm \sqrt{-1}\}$. % as previously introduced in \eqref{dz1dz2}.
The adjoint vertex operator \cite{K, FHL}  
is defined by 
\begin{equation}
Y^{\dagger }(u,z)=\sum_{n\in \Z}u^{\dagger }(n)z^{-n-1}= T_{\lambda}Y(u,z)T_{\lambda}^{-1}. \label{eq: adj op}
\end{equation}
%%
%%%%%%%%%%%%%%%%%%%%%%%%%%%%%%%%%%%%%%%%%%%%%%%%%%%%%%5
%
A bilinear form $\langle . , . \rangle_{\lambda}$ on $V$ is 
invariant if for all $a$, $b$, $u\in V$, 
%5
if  
\begin{equation}
\langle Y(u,z)a,b\rangle_{\lambda} =%(-1)^{p(u)p(a)}
\langle a,Y^{\dagger }(u,z)b\rangle_{\lambda}, 
\label{eq: inv bil form}
\end{equation}%
i.e.
\[
 \langle u(n)a,b\rangle_{\lambda} =%(-1)^{p(u)p(a)}
\langle a,u^{\dagger }(n)b\rangle_{\lambda}.
\] 
Thus it follows that 
\begin{equation}
\label{dubay}
\langle L_V(0)a,b\rangle_{\lambda} =\langle a,L_V(0)b\rangle_{\lambda}, 
\end{equation}
 so that 
\begin{equation}
\label{condip}
\langle a,b\rangle_{\lambda} =0, 
\end{equation}
  if $wt(a)\not=wt(b)$ for homogeneous $a,b$.
 One also finds 
\[
\langle a,b\rangle_{\lambda} = \langle b,a \rangle_{\lambda}.  
\]
The form 
$\langle . , .\rangle_{\lambda}$ is unique up to normalization if $L_V(1)V_{1}=V_{0}$.  
 Given any $V$ basis $\{ u^{\alpha}\}$ we define the %Li--Z 
dual $V$ basis $\{ \overline{u}^{\beta}\}$ where 
\[
\langle u^{\alpha} ,\overline{u}^{\beta}\rangle_{\lambda}=\delta^{\alpha\beta}.
\] 
%%

%%
%%%%%%%%%%%%%%%%%%%%%%%%%%%%%%%%%%%%%%%%%%%%%%%%%%%%%%%%%%%%%%%%%%%%%%%%%%%%%%%%%%%%%%%%%%%%%%%%%%%%%%%%%%%%%%
%%%%%%%%%%%%%%%%%%%%%%%%%%%%%%%%%%%%%%%%%%%%%%%%%%%%%%%%%%%%%%%%%%%%%%%%%%%%%%%%%%%%%%%%%%%%%%%%%%%%%%%%%%%%%%
\section{Appendix: $\W_{z_1, \ldots, z_n}$-valued rational functions} 
\label{valued}
%%

%%
%%%%%%%%%%%%%%%%%%%%%%%%%%%%%%%%%%%%%%%%%%%%%%%%%%%%%%%%%%%%%%
%%%%%%%%%%%%%%%%%%%%%%%%%%%%%%%%%%%%%%%%%%%%%%%%%%%%%%%%%%%%%%
 Recall the definition of shuffles. %\cite{Huang}. 
\begin{definition}
\label{shuffles}
Let  $S_{q}$ be the permutation group. 
For $l \in \N$ and $1\le s \le l-1$, let $J_{l; s}$ be the set of elements of 
$S_{l}$ which preserve the order of the first $s$ numbers and the order of the last 
$l-s$ numbers, that is,
\[
J_{l, s}=\{\sigma\in S_{l}\;|\;\sigma(1)<\cdots <\sigma(s),\;
\sigma(s+1)<\cdots <\sigma(l)\}.
\]
The elements of $J_{l; s}$ are called shuffles, and we use the notation 
\[
J_{l; s}^{-1}=\{\sigma\;|\; \sigma\in J_{l; s}\}.
\]
\end{definition}
%%
%%%%%%%%%%%%%%%%%%%%%%%%%%%%%%%%%%%%%%%%%%%%%%%%%%%%%%%%%%%%%%%%%%%%%%%%%%%%%%%%%%%%%%%%%%%%
%%
\begin{definition}
\label{confug}
 We define the configuration spaces: %are \cite{Huang}: 
\[
F_{n}\C=\{(z_{1}, \dots, z_{n})\in \C^{n}\;|\; z_{i}\ne z_{j}, i\ne j\},
\] 
for $n\in \Z_{+}$.
\end{definition}
%%
%%%%%%%%%%%%%%%%%%%%%%%%%%%%%%%%%%%%%%%%%%%%%%%%%%%%%%%%%%%%%%%%%%%%%%%%%%%%%%
Let $V$ be a grading-restricted 
vertex algebra, and $W$ a a grading-restricted generalized $V$-module.  %(see Appendix \ref{graded}). 
By $\overline{W}$ we denote the algebraic completion of $W$, %\cite{Huang}
\[
\overline{W}=\prod_{n\in \mathbb C}W_{(n)}=(W')^{*}.
\]
%%
%%%%%%%%%%%%%%%%%%%%%%%%%%%%%%%%%%%%%%%%%%%%%%%%%%%%%%%%%%%%%%%%%%%%%%%%%%%%%%%
\begin{definition}
 A $\overline{W}$-valued rational function in $(z_{1}, \dots, z_{n})$ 
with the only possible poles at 
$z_{i}=z_{j}$, $i\ne j$, 
is a map 
\begin{eqnarray*}
 f:F_{n}\C &\to& \overline{W},   
\\
 (z_{1}, \dots, z_{n}) &\mapsto& f(z_{1}, \dots, z_{n}),   
\end{eqnarray*} 
such that for any $w'\in W'$,  
\[
R(z_1, \ldots, z_n)= \langle w', f(z_{1}, \dots, z_{n}) \rangle,
\] 
is a rational function in $(z_{1}, \dots, z_{n})$  
with the only possible poles  at 
$z_{i}=z_{j}$, $i\ne j$. 
In this paper, such a map is called $\overline{W}$-valued rational function 
in $(z_{1}, \dots, z_{n})$ with possible other poles.
 The space of $\overline{W}$-valued rational functions is denoted by $\overline{W}_{z_{1}, \dots, z_{n}}$.
\end{definition} 
%%
%%%%%%%%%%%%%%%%%%%%%%%%%%%%%%%%%%%%%%%%%%%%%%%%%%%%%%%%%%%%%%%%%%%%%%%%%%
%%%%%%%%%%%%%%%%%%%%%%%%%%%%%%%%%%%%%%%%%%%%%%%%%%%%%%%%%%%%%%%%%%%%%%%%%%%%
\begin{definition}
One defines an action of $S_{n}$ on the space $\hom(V^{\otimes n}, 
\overline{W}_{z_{1}, \dots, z_{n}})$ of linear maps from 
$V^{\otimes n}$ to $\overline{W}_{z_{1}, \dots, z_{n}}$ by 
\begin{equation}
\label{sigmaction}
\sigma(\Phi)(v_{1}, z_1; \cdots; v_{n}, z_n)  
=\Phi (v_{\sigma(1)}, v_{\sigma(1)};  \cdots v_{\sigma(n)}, z_{\sigma(n)}), 
\end{equation}   
for $\sigma\in S_{n}$, and $v_{1}, \dots, v_{n}\in V$.
\end{definition}
We will use the notation $\sigma_{i_{1}, \dots, i_{n}}\in S_{n}$, to denote the 
the permutation given by $\sigma_{i_{1}, \dots, i_{n}}(j)=i_{j}$,  
for $j=1, \dots, n$.
In \cite{Huang} one finds:
\begin{proposition}%{\rm \cite{Huang}}  
The subspace of $\hom ( V^{\otimes n}, \W_{z_1, \ldots, z_n} )$ consisting of linear maps
having
the $L(-1)$-derivative property, having the $L(0)$-conjugation property
or being composable with $m$ vertex operators is invariant under the 
action of $S_{n}$.
\end{proposition}
%%
%%%%%%%%%%%%%%%%%%%%%%%%%%%%%%%%%%%%%%%%%%%%%%%%%%%%%%%%%%%%%%%%%%%%%%%
%%
%%  

%%
%%%%%%%%%%%%%%%%%%%%%%%%%%%%%%%%%%%%%%%%%%%%%%%%%%%%%%%%%%%%%%%%%%%
%%%%%%%%%%%%%%%%%%%%%%%%%%%%%%%%%%%%%%%%%%%%%%%%%%%%%%%%%%%%%%%%%%%
%%
Let us introduce another definition: 
\begin{definition}
\label{wspace}
We define the space $\W_{z_1, \dots, z_n}$ of 
  $\overline{W}_{z_{1}, \dots, z_{n}}$-valued rational forms $\Phi$ 
with each vertex algebra element entry $v_i$, $1 \le i \le n$
of a quasi-conformal grading-restricted vertex algebra $V$ tensored with power $\wt(v_i)$-differential of 
corresponding formal parameter $z_i$, i.e., 
\begin{equation}
\label{bomba}
\Phi \left(dz_1^{\wt(v_1)} \otimes v_{1}, z_1; \cdots;
 dz_n^{\wt(v_n)} \otimes  v_{n}, z_n\right) \in \W_{z_1, \dots, z_n}. 
\end{equation}
We assume also that \eqref{bomba} %$\Phi \in \W_{z_1, \dots, z_n}$   
satisfy $L_V(-1)$-derivative \eqref{lder1}, $L_V(0)$-conjugation 
\eqref{loconj} properties, and the symmetry property 
with respect to action of the symmetric group $S_n$: 
%%
%%%%%%%%%%%%%%%%%%%%%%%%%%%%%%%%%%%%%%%%%%%%%%%%%%%%%%%%%%%%%%%%%%%%%%5
%%  
%%
\begin{equation}
\label{shushu} 
\sum_{\sigma\in J_{l; s}^{-1}}(-1)^{|\sigma|}
 \left(\Phi(v_{\sigma(1)}, z_{\sigma(1)}; \cdots; v_{\sigma(l)},  z_{\sigma(1)}) \right)=0.
\end{equation}
\end{definition}
In Section \ref{spaces} we prove that \eqref{bomba} is invariant with respect to changes of formal parameters 
$(z_1, \dots, z_n)$. 
%%
%%%%%%%%%%%%%%%%%%%%%%%%%%%%%%%%%%%%%%%%%%%%%%%%%%%%%%%%%%%%%%%%%%%%%%%%%%%%%%%%%%%%%%%%%%%%%
%%%%%%%%%%%%%%%%%%%%%%%%%%%%%%%%%%%%%%%%%%%%%%%%%%%%%%%%%%%%%%%%%%%%%%%%%%%%%%%%%%%%%%%%%%%%%
\section{Appendix: Properties of matrix elements for a grading-restricted vertex algebra}
\label{properties}
Let $V$ be a grading-restricted vertex algebra and $W$ a grading-restricted generalized $V$-module. 
Let us recall some definitions and facts about matrix elements for a grading-restricted vertex algebra \cite{Huang}. 
If a meromorphic function $f(z_{1}, \dots, z_{n})$ on a domain in $\mathbb C^{n}$ is 
 analytically extendable to a rational function in $z_{1}, \dots, z_{n}$, 
we denote this rational function  
by $R(f(z_{1}, \dots, z_{n}))$.
%%
%%%%%%%%%%%%%%%%%%%%%%%%%%%%%%%%%%%%%%%%%%%%%%%%%%%%%%%%%%%%%%%%%%%%%%%%%%%%%%%%%%%%%%%%%%%%%%%%%%
%% 
%%
%%%%%%%%%%%%%%%%%%%%%%%%%%%%%%%%%%%%%%%%%%%%%%%%%%%%%%%%%%%%%%%%%%%%%%%%%
%%
%% 
Let us recall a few definitions from \cite{Huang}
\begin{definition}
\label{}
For $n\in \Z_{+}$,  
a linear map 
\[
\Phi(v_{1}, z_{1};  \ldots ; v_{n},  z_{n})
= V^{\otimes n}\to 
\W_{z_{1}, \dots, z_{n}}, 
\]
 is said to have
the  $L(-1)$-derivative property if
%%%
\begin{equation}
\label{lder1}
(i) \qquad \langle w', \partial_{z_{i}} \Phi(v_{1}, z_{1};  \ldots ; v_{n},  z_{n})\rangle= 
\langle w',  \Phi(v_{1}, z_{1};  \ldots; L_{V}(-1)v_{i}, z_i; \ldots ; v_{n},  z_{n}) \rangle, 
\end{equation}
for $i=1, \dots, n$, $v_{1}, \dots, v_{n}\in V$, $w'\in W$, 
and  
\begin{eqnarray}
\label{lder2}
(ii) \qquad \sum\limits_{i=1}^n\partial_{z_{i}} \langle w', \Phi(v_{1}, z_{1};  \ldots ; v_{n},  z_{n})\rangle=
\langle w', L_{W}(-1).\Psi(v_{1}, z_{1};  \ldots ; v_{n},  z_{n}) \rangle, 
\end{eqnarray}
with some action $.$ of $L_{W}(-1)$  on $\Phi(v_{1}, z_{1};  \ldots ; v_{n},  z_{n})$, and 
and  $v_{1}, \dots, v_{n}\in V$. 
\end{definition}
%%
%%%%%%%%%%%%%%%%%%%%%%%%%%%%%%%%%%%%%%%%%%%%%%%%%%%%%%%%%%%%%%%%%
%% 
%%
Note that since $L_{W}(-1)$ is a weight-one operator on $W$, 
for any $z\in \C$,  $e^{zL_{W}(-1)}$ is a well-defined linear operator
on $\overline{W}$. 
%%

%%%%%%%%%%%%%%%%%%%%%%%%%%%%%%%%%%%%%%%%%%%%%%%%%%%%%%%%%%%%%%%%%%%%%%
In \cite{Huang} we find the following 
\begin{proposition}
\label{n-comm}
Let $\Phi$ be a linear map having the $L(-1)$-derivative property. Then for $v_{1}, \dots, v_{n}\in V$,
$w'\in W'$, $(z_{1}, \dots, z_{n})\in F_{n}\C$, $z\in \C$ such that
$(z_{1}+z, \dots,  z_{n}+z)\in F_{n}\C$,
\begin{eqnarray}
\label{ldirdir}
%%&&
 \langle w', e^{zL_{W}(-1)} 
\Phi \left(v_{1}, z_1; \ldots; v_{n}, z_{n} \right) \rangle 
%%\\
%%&& \qquad 
%
=\langle w', 
\Phi(v_{1}, z_{1}+z ; \ldots; v_{n},  z_{n}+z) \rangle,  
\end{eqnarray}
and for $v_{1}, \dots, v_{n}\in V$,
$w'\in W'$, $(z_{1}, \dots, z_{n})\in F_{n}\C$, $z\in \C$,  and $1\le i\le n$ such that
\[
(z_{1}, \dots, z_{i-1}, z_{i}+z, z_{i+1}, \dots, z_{n})\in F_{n}\C,
\]
the power series expansion of 
\begin{equation}\label{expansion-fn}
\langle w', 
 \Phi(v_{1}, z_1;  \ldots; v_{i-1}, z_{i-1};  v_i, z_{i}+z; v_{v+1}, z_{i+1}; \ldots  v_{n}, z_n) \rangle, 
\end{equation}
in $z$ is equal to the power series
\begin{equation}\label{power-series}
\langle w', 
\Phi(v_{1} z_1; \ldots;  v_{i-1}, z_{i-1}; e^{zL(-1)}v_{i}, z_i;  
 v_{i+1}, z_{i+1};  \ldots; v_{n}, z_n) \rangle,  
\end{equation}
in $z$.
In particular, the power series (\ref{power-series}) in $z$ is absolutely convergent
to (\ref{expansion-fn}) in the disk $|z|<\min_{i\ne j}\{|z_{i}-z_{j}|\}$. 
\end{proposition}
%%

%%%%%%%%%%%%%%%%%%%%%%%%%%%%%%%%%%%%%%%%%%%%%%%%%%%%%%%%%%%%%%%%%%%%%%%%%%%%%%%%%
Finally, we have 
\begin{definition}
A linear map 
\[
\Phi: V^{\otimes n} \to \W_{z_{1}, \dots, z_{n}}
\]
 has the  $L(0)$-conjugation property if for $v_{1}, \dots, v_{n}\in V$,
$w'\in W'$, $(z_{1}, \dots, z_{n})\in F_{n}\C$ and $z\in \C^{\times}$ so that 
$(zz_{1}, \dots, zz_{n})\in F_{n}\C$,
\begin{eqnarray}
\label{loconj}
\langle w', z^{L_{W}(0)}   
\Phi \left(v_{1}, z_1; \cdots; v_{n}, z_{n} \right) \rangle 
%%
%%\nn
%%
%%&&
%%
=\langle w', 
 \Phi(z^{L(0)} v_{1}, zz_{1};  \cdots ;  z^{L(0)} v_{n},  zz_{n})\rangle.
\end{eqnarray} 
\end{definition}
%%
%%%%%%%%%%%%%%%%%%%%%%%%%%%%%%%%%%%%%%%%%%%%%%%%%%%%%%%%%%%%%%%%%%%%%%%%%%%%%%%%%%%%%%%%%

%%%%%%%%%%%%%%%%%%%%%%%%%%%%%%%%%%%%%%%%%%%%%%%%%%%%%%%%%%%%%%%%%%%%%%%%%%%%%%%%%%%%%%%%%%%%%%
\subsection{$E$-elements}
 For $w\in W$, the $\overline{W}$-valued function 
is given by 
$$
E^{(n)}_{W}(v_{1}, z_1; \cdots; v_{n}, z_n; w)
= E(\omega_{W}(v_{1}, z_{1}) \ldots \omega_{W}(v_{n}, z_{n})w),   
$$
where an element $E(.)\in \overline{W}$ is given by (see notations for $\omega_W$ in Section \ref{spaces}) 
\[
\langle w', E(.) \rangle =R(\langle w', . \rangle), 
\]
 and $R(.)$ denotes the following (cf. \cite{Huang}).   
Namely, 
if a meromorphic function $f(z_{1}, \dots, z_{n})$ on a region in $\C^{n}$
can be analytically extended to a rational function in $(z_{1}, \dots, z_{n})$, 
then the notation $R(f(z_{1}, \dots, z_{n}))$ is used to denote such rational function. 
One defines  
\[
E^{W; (n)}_{WV}(w; v_{1}, z_1 ; \ldots; v_{n}, z_n) 
=E^{(n)}_{W}(v_{1}, z_1; \ldots;  v_{n}, z_n; w),
\]
where 
$E^{W; (n)}_{WV}(w; v_{1}, z_1 ; \ldots; v_{n}, z_n)$ is  
an element of $\overline{W}_{z_{1}, \dots, z_{n}}$.
One defines
\[
 \Phi\circ \left(E^{(l_{1})}_{V;\;\one}\otimes \cdots \otimes E^{(l_{n})}_{V;\;\one}\right): 
V^{\otimes m+n}\to 
\overline{W}_{z_{1},  \dots, z_{m+n}},
\] 
by
\begin{eqnarray*}
\lefteqn{(\Phi\circ (E^{(l_{1})}_{V;\;\one}\otimes \cdots \otimes 
E^{(l_{n})}_{V;\;\one}))(v_{1}\otimes \cdots \otimes v_{m+n-1})}
\nn
&&=E(\Phi(E^{(l_{1})}_{V; \one}(v_{1}\otimes \cdots \otimes v_{l_{1}})\otimes \cdots
\nn 
&&\quad\quad\quad\quad\quad \otimes 
E^{(l_{n})}_{V; \one}
(v_{l_{1}+\cdots +l_{n-1}+1}\otimes \cdots 
\otimes v_{l_{1}+\cdots +l_{n-1}+l_{n}}))),  
\end{eqnarray*}
and 
%%%%%%%%%%%%%%%%%%%%%%%%%%%%%%%%%%%%%%%%%%%%%%%%%%%%%%%%%%%
%% 
%%
\[
E^{(m)}_{W} \circ_{0%m+1
} \Phi: V^{\otimes m+n}\to 
\overline{W}_{z_{1}, \dots, 
z_{m+n-1}},
\]
 is given by 
\begin{eqnarray*}
\lefteqn{
(E^{(m)}_{W}\circ_{0%m+1
}\Phi)(v_{1}\otimes \cdots \otimes v_{m+n})
}\nn
&&
=E(E^{(m)}_{W}(v_{1}\otimes \cdots\otimes v_{m};
\Phi(v_{m+1}\otimes \cdots\otimes v_{m+n}))). 
\end{eqnarray*}
Finally,  
\[
E^{W; (m)}_{WV}\circ_{m+1 %0
}\Phi: V^{\otimes m+n}\to 
\overline{W}_{z_{1}, \dots, 
z_{m+n-1}},
\]
 is defined by 
\begin{eqnarray*}
(E^{W; (m)}_{WV}\circ_{m+1 %0
}\Phi)(v_{1}\otimes \cdots \otimes v_{m+n})
%%\nn
%%&&
 =E(E^{W; (m)}_{WV}(\Phi(v_{1}\otimes \cdots\otimes v_{n})
; v_{n+1}\otimes \cdots\otimes v_{n+m})). 
\end{eqnarray*}
In the case that $l_{1}=\cdots=l_{i-1}=l_{i+1}=1$ and $l_{i}=m-n-1$, for some $1 \le i \le n$,
we will use $\Phi\circ_{i} E^{(l_{i})}_{V;\;\one}$ to 
denote $\Phi\circ (E^{(l_{1})}_{V;\;\one}\otimes \cdots 
\otimes E^{(l_{n})}_{V;\;\one})$.
Note that our notations differ with that of \cite{Huang}. 
%%
%%%%%%%%%%%%%%%%%%%%%%%%%%%%%%%%%%%%%%%%%%%%%%%%%%%%%%%%%%%%%%%%%%%%%%%%%%%%
%%%%%%%%%%%%%%%%%%%%%%%%%%%%%%%%%%%%%%%%%%%%%%%%%%%%%%%%%%%%%%%%%%%%%%%%%%%%
\section{Appendix: Maps composable with vertex operators}
\label{composable}
%%

%%%%%%%%%%%%%%%%%%%%%%%%%%%%%%%%%%%%%%%%%%%%%%%%%%%%%%%%%%%%%%%%%%%%%%%%5
%% 

%%
In the construction of double complexes in Section \ref{coboundary}
 we would like to use linear maps from tensor powers of $V$ to the space 
$\W_{z_{1}, \dots, z_{n}}$ 
to define cochains in vertex algebra cohomology theory.
 For that purpose, in particular, to define the coboundary operator, we have to compose cochains 
with vertex operators. However, as mentioned in \cite{Huang},
 the images of vertex operator maps in general do not belong to 
algebras or thier modules.
 They belong to corresponding algebraic completions which constitute 
  one of the most subtle features of the theory of vertex algebras.
 Because of this, we might not be able to compose vertex operators directly. 
In order to overcome this problem \cite{H2}, we first write a series by projecting
an element of the algebraic completion of an algebra or a module 
to its homogeneous components. 
Then we compose these homogeneous components with vertex operators,  
and take formal sums.
 If such formal sums are absolutely convergent, then these operators 
can be composed and can be used in constructions. 

%%%%%%%%%%%%%%%%%%%%%%%%%%%%%%%%%%%%%%%%%%%%%%%%%%%%%%%%%%%%%%%%%%%
%%  
%
Another question that appears is the question of associativity. 
Compositions of maps are usually associative.
 But for compositions of maps defined by sums of 
absolutely convergent series the existence of does not provide associativity in general.  
 Nevertheless, the requirement of analyticity provides the associativity \cite{Huang}.  
%%
%%%%%%%%%%%%%%%%%%%%%%%%%%%%%%%%%%%%%%%%%%%%%%%%%
\begin{definition}
\label{composabilitydef}
For a $V$-module 
\[
W=\coprod_{n\in \C} W_{(n)}, 
\]
 and $m\in \C$, 
let
\[
P_{m}: \overline{W}\to W_{(m)}, 
\]
 be the projection from 
$\overline{W}$ to $W_{(m)}$. 
Let 
\[
\Phi: V^{\otimes n} \to \W_{z_{1}, \dots, z_{n}}, 
\]
 be a linear map. For $m\in \N$, 
$\Phi$ is called \cite{Huang, FQ} to be composable with $m$ vertex operators if  
the following conditions are satisfied:
%%
%%%%%%%%%%%%%%%%%%%%%%%%%%%%%%%%%%%%%%%%%%%%%%%%%%%%%%%%%%%%%%%
%%
%% 
%

\medskip 
1) Let $l_{1}, \dots, l_{n}\in \Z_+$ such that $l_{1}+\cdots +l_{n}=m+n$,
$v_{1}, \dots, v_{m+n}\in V$ and $w'\in W'$. Set 
 \begin{eqnarray}
\label{psii}
\Xi_{i}
&
=
&
E^{(l_{i})}_{V}(v_{k_1}, z_{k_1}- \zeta_{i};  %{l_{1}+\cdots +l_{i-1}+1}
%%
%\otimes
 \ldots; %\otimes 
v_{k_i}, z_{k_i}- \zeta_{i} %{l_{1}+\cdots +l_{i-1}+l_{i}}
 ; \one_{V}),    
\end{eqnarray}
where
\begin{eqnarray}
\label{ki}
 {k_1}={l_{1}+\cdots +l_{i-1}+1}, \quad  \ldots, \quad  {k_i}={l_{1}+\cdots +l_{i-1}+l_{i}}, 
\end{eqnarray} 
for $i=1, \dots, n$. Then there exist positive integers $N^n_m(v_{i}, v_{j})$
depending only on $v_{i}$ and $v_{j}$ for $i, j=1, \dots, k$, $i\ne j$ such that the series 
\begin{eqnarray}
\label{Inm}
\mathcal I^n_m(\Phi)=
\sum_{r_{1}, \dots, r_{n}\in \Z}\langle w', 
\Phi(P_{r_{1}}\Xi_{1}; \zeta_1; %\otimes
 \ldots; %\otimes
P_{r_n} \Xi_{n}, \zeta_{n}) %, \dots, \zeta_{n}) 
\rangle,
\end{eqnarray} 
is absolutely convergent  when 
\begin{eqnarray}
\label{granizy1}
|z_{l_{1}+\cdots +l_{i-1}+p}-\zeta_{i}| 
+ |z_{l_{1}+\cdots +l_{j-1}+q}-\zeta_{i}|< |\zeta_{i}
-\zeta_{j}|,
\end{eqnarray} 
for $i,j=1, \dots, k$, $i\ne j$ and for $p=1, 
\dots,  l_i$ and $q=1, \dots, l_j$. 
The sum must be analytically extended to a
rational function
in $(z_{1}, \dots, z_{m+n})$,
 independent of $(\zeta_{1}, \dots, \zeta_{n})$, 
with the only possible poles at 
$z_{i}=z_{j}$, of order less than or equal to 
$N^n_m(v_{i}, v_{j})$, for $i,j=1, \dots, k$,  $i\ne j$.

\medskip 
%%%%%%%%%%%%%%%%%%%%%%%%%%%%%%%%%%%%%%%%%%%%%%%%%%%%%%%%%%%%%%%%%%%%%%%%%%%%%%%%%%%%%%
%%
2) 
 For $v_{1}, \dots, v_{m+n}\in V$, there exist 
positive integers $N^n_m(v_{i}, v_{j})$, depending only on $v_{i}$ and 
$v_{j}$, for $i, j=1, \dots, k$, $i\ne j$, such that for $w'\in W'$, and 
\[
{\bf v}_{n,m}=(v_{1+m}\otimes \cdots\otimes v_{n+m}),
\]
\[
  {\bf z}_{n,m}=(z_{1+m}, \dots, z_{n+m}),
\]
 such that 
\begin{eqnarray}
\label{Jnm}
\mathcal J^n_m(\Phi)=  
\sum_{q\in \C}\langle w', E^{(m)}_{W} \Big(v_{1}\otimes \cdots\otimes v_{m}; 
P_{q}( \Phi({\bf v}_{n,m} ) ({\bf z}_{n,m})\Big)\rangle, 
\end{eqnarray}
is absolutely convergent when 
\begin{eqnarray}
\label{granizy2}
z_{i}\ne z_{j}, \quad i\ne j, \quad 
\nn
|z_{i}|>|z_{k}|>0, 
\end{eqnarray}
 for $i=1, \dots, m$, and $k=m+1, \dots, m+n$, and the sum can be analytically extended to a
rational function 
in $(z_{1}, \dots, z_{m+n})$ with the only possible poles at 
$z_{i}=z_{j}$, of orders less than or equal to 
$N^n_m(v_{i}, v_{j})$, for $i, j=1, \dots, k$, $i\ne j$,. 
\end{definition}
%%
%%%%%%%%%%%%%%%%%%%%%%%%%%%%%%%%%%%%%%%%%%%%%%%%%%%%%%%%%%%%%%%%%%%%%%%%%%%%%%%%%%%%%%%
In \cite{Huang}, we the following useful proposition is proven: 
\begin{proposition}\label{comp-assoc}
Let $\Phi: V^{\otimes n}\to 
\overline{W}_{z_{1}, \dots, z_{n}}$
be composable with $m$ vertex operators. Then we have:
\begin{enumerate}
\item For $p\le m$, $\Phi$ is 
composable with $p$ vertex operators and for 
$p, q\in \Z_{+}$ such that $p+q\le m$ and 
$l_{1}, \dots, l_{n} \in \Z_+$ such that $l_{1}+\cdots +l_{n}=p+n$,
$\Phi\circ (E^{(l_{1})}_{V;\;\one}\otimes 
\cdots \otimes E^{(l_{n})}_{V;\;\one})$ and $E^{(p)}_{W}\circ_{p+1}\Phi$
are
composable with $q$ vertex operators. 

\item For $p, q\in \Z_{+}$ such that $p+q\le m$,
$l_{1}, \dots, l_{n} \in \Z_+$ such that $l_{1}+\cdots +l_{n}=p+n$ and
$k_{1}, \dots, k_{p+n} \in \Z_+$ such that $k_{1}+\cdots +k_{p+n}=q+p+n$,
we have
\begin{eqnarray*}
&(\Phi\circ (E^{(l_{1})}_{V;\;\one}\otimes 
\cdots \otimes E^{(l_{n})}_{V;\;\one}))\circ 
(E^{(k_{1})}_{V;\;\one}\otimes 
\cdots \otimes E^{(k_{p+n})}_{V;\;\one})&\nn
&=\Phi\circ (E^{(k_{1}+\cdots +k_{l_{1}})}_{V;\;\one}\otimes 
\cdots \otimes E^{(k_{l_{1}+\cdots +l_{n-1}+1}+\cdots +k_{p+n})}_{V;\;\one}).&
\end{eqnarray*}

\item For $p, q\in \Z_{+}$ such that $p+q\le m$ and
$l_{1}, \dots, l_{n} \in \Z_+$ such that $l_{1}+\cdots +l_{n}=p+n$,
we have
$$E^{(q)}_{W}\circ_{q+1} (\Phi\circ (E^{(l_{1})}_{V;\;\one}\otimes 
\cdots \otimes E^{(l_{n})}_{V;\;\one}))
=(E^{(q)}_{W}\circ_{q+1} \Phi)\circ (E^{(l_{1})}_{V;\;\one}\otimes 
\cdots \otimes E^{(l_{n})}_{V;\;\one}).$$

\item For $p, q\in \Z_{+}$ such that $p+q\le m$, we have
$$E^{(p)}_{W}\circ_{p+1} (E^{(q)}_{W}\circ_{q+1}\Phi)
=E^{(p+q)}_{W}\circ_{p+q+1}\Phi.$$
\end{enumerate}
\end{proposition}
%%
%%
%%%%%%%%%%%%%%%%%%%%%%%%%%%%%%%%%%%%%%%%%%%%%%%%%%%%%%%%%%%%%%%%%%%%%%%%%%%%%%%%%%%%%%%%%%%%%%%%%%%%%%%%%
%%%%%%%%%%%%%%%%%%%%%%%%%%%%%%%%%%%%%%%%%%%%%%%%%%%%%%%%%%%%%%%%%%%%%%%%%%%%%%%%%%%%%%%%%%%%%%%%%%%%%%%%
\section{Appendix: Proofs of Lemmas \ref{empty}, \ref{nezu}, \ref{subset} and Proposiiton \ref{nezc}}
\label{proof}
In this Appendix we provide proofs of Lemma \ref{nezu} and Proposiiton \ref{nezc}
%%
%%%%%%%%%%%%%%%%%%%%%%%%%%%%%%%%%%%%%%%%%%%%%%%%%%%%%%%%%%%%%%%%%%%%%%%%%%%%%%%%%%%%%%%%%%%%%%%%%%%%%%%%%%
\subsection{Proof of Lemma \ref{empty}}
We start with the proof of Lemma \ref{empty}. 
%%%%%%%%%%%%%%%%%%%%%%%%%%%%%%%%%%%%%%%%%%%%%%%%%%%%%%%%%%%%%%%%%%%%%%%%%%%%%%%%%%%%%%%%%%%%%%%%%%%%%%%%%
\begin{proof} 
From the construction %in \cite{Huang} 
 of spaces for double complex for a grading-restricted vertex algebra cohomology, 
it is clear that the spaces $C^n (V, \W, \U, \F)(U_j)$,  $1 \le s \le m$ in Definition \ref{initialspace} are non-empty.
On each transversal section $U_s$,  $1 \le s \le m$, $\Phi(v_1, c_j(p_1);  \ldots; v_n, c_j(p_n))$
 belongs to the space 
$\W_{c_j(p_1),  \ldots, c_j(p_n)}$, and satisfy the $L(-1)$-derivative \eqref{lder1} and $L(0)$-conjugation
 \eqref{loconj}
properties. 
 A map $\Phi(v_1, c_j(p_1)$;  $\ldots$; $v_n, c_j(p_n))$ 
is composable with $m$ vertex operators 
with formal parameters identified with local coordinates $c_j(p'_j)$, on each transversal section $U_j$.
Note that on each transversal section, $n$ and $m$ the spaces \eqref{ourbicomplex} remain the same.   
The only difference may be constituted by the composibility conditions \eqref{Inm} and \eqref{Jnm} for  $\Phi$.
%%

%%%%%%%%%%%%%%%%%%%%%%%%%%%%%%%%%%%%%%%%%%%%%%%%%%%%%%%%%%%%%%%%%%%%%%%%%%%%%%%%%%%%%%%%%%%%%%%%%%%%5
%% etot kusok vynesti v otdel'noe obyasnenie i potom iz vseh kuskov na nego ssylat'sya dlya sokrascheniya 
%%
%%
In particular, 
%%
%%%%%%%%%%%%%%%%%%%%%%%%%%%%%%%%%%%%%%%%%%%%%%%%%%%%%%%%%%%%%%%%%%%%%%%%%%%%%%%%%%%%%%%%%%%%%%%
%%
for  $l_{1}, \dots, l_{n}\in \Z_+$ such that $l_{1}+\cdots +l_{n}=n+m$, 
$v_{1}, \dots, v_{m+n}\in V$ and $w'\in W'$, recall \eqref{psii} that 
 \begin{eqnarray}
\label{psiii}
\Xi_{i}
&
=
&
\omega_V(v_{k_1}, c_{k_1}(p_{k_1})-\zeta_i)  \ldots  \omega_V(v_{k_i}, c_{k_i}(p_{k_i})-\zeta_i) \;\one_{V},     
\end{eqnarray}
where $k_i$ is defined in \eqref{ki}, 
for $i=1, \dots, n$, 
 depend on coordinates of points on transversal sections. 
At the same time, in the first composibility condition 
 \eqref{Inm} depends on projections $P_r(\Xi_i)$, $r\in \C$, 
 of $\W_{c(p_1), \ldots, c(p_n)}$ to $W$, and 
on arbitrary variables $\zeta_i$, $1 \le i \le m$.  
On each transversal connection $U_s$, $1 \le s \le m$,  
the absolute convergency is assumed for the series  \eqref{Inm} (cf. Appendix \ref{composable}). 
Positive integers $N^n_m(v_{i}, v_{j})$,
(depending only on $v_{i}$ and $v_{j}$) as well as $\zeta_i$, 
 for $i$, $j=1, \dots, k$, $i\ne j$, 
may vary for transversal sections $U_s$. 
Nevertheless, the domains of convergency determined by the conditions \eqref{granizy1} which have the form 
\begin{equation}
\label{popas}
|c_{m_i}(p_{m_i})-\zeta_{i}| 
+ |c_{n_i}(p_{n_i})-\zeta_{i}|< |\zeta_{i}-\zeta_{j}|,
\end{equation} 
for $m_i= l_{1}+\cdots +l_{i-1}+p$, $n=l_{1}+\cdots +l_{j-1}+q$,   
 $i$, $j=1, \dots, k$, $i\ne j$ and for $p=1, 
\dots,  l_i$ and $q=1, \dots, l_j$, 
are limited by $|\zeta_{i}-\zeta_{j}|$ in \eqref{popas} from above. 
Thus, for the intersection variation of sets of homology embeddings in \eqref{ourbicomplex}, 
 the absolute convergency condition for \eqref{Inm} is still fulfilled. 
%%
%%%%%%%%%%%%%%%%%%%%%%%%%%%%%%%%%%%%%%%%%%%%%%%%%%%%%%%%%%%%%%%%%%%%%%%%%%%%%%%%%%%%%%%%%%%%%%%%%%%%
%%
Under intersection in \eqref{ourbicomplex}
by choosing appropriate $N^n_m(v_{i}, v_{j})$, 
one can analytically extend \eqref{Inm}  
to a rational function in $(c_1(p_{1}), \dots, c_{n+m}(p_{n+m}))$,
 independent of $(\zeta_{1}, \dots, \zeta_{n})$, 
with the only possible poles at 
$c_{i}(p_i)=c_{j}(p_j)$, of order less than or equal to 
$N^n_m(v_{i}, v_{j})$, for $i,j=1, \dots, k$,  $i\ne j$. 

%%%%%%%%%%%%%%%%%%%%%%%%%%%%%%%%%%%%%%%%%%%%%%%%%%%%%%%%%%%%%%%%%%%%%%%%%%%%%%%%%%%%%%%%%%%%%%%%%%%%%%%%%%%%
%%
As for the second condition in Definition of composibility, we note that, on each transversal section, 
the domains of absolute convergensy 
$c_{i}(p_i)\ne c_{j}(p_j)$, $i\ne j$
\[
|c_{i}(p_i)| > |c_{k}(p_j)|>0,
\]
 for 
 $i=1, \dots, m$, 
and 
$k=1+m, \dots, n+m$, 
for 
\begin{eqnarray}
\mathcal J^n_m(\Phi) &=&  
\sum_{q\in \C}\langle w',
\omega_W(v_{1}, c_1(p_1)) \ldots  \omega_W(v_{m}, c_m(p_m)) \;    
\nn
&& \qquad 
P_{q}(\Phi( v_{1+m}, c_{1+m}(p_{1+m}); \ldots; v_{n+m}, c_{n+m}(p_{n+m})) \rangle, 
\end{eqnarray}
are limited from below by the same set ot absolute values of local coordinates on 
transversal section. 
Thus, under intersection in \eqref{ourbicomplex} this condition is preserved, and 
  the sum \eqref{Jnm} can be analytically extended to a
rational function 
in $(c_{1}(p_1)$, $\dots$, $c_{m+n}(p_{m+n}))$ with the only possible poles at
$c_{i}(p_i)=c_{j}(p_j)$, 
of orders less than or equal to 
$N^n_m(v_{i}, v_{j})$, for $i, j=1, \dots, k$, $i\ne j$. 
Thus, we proved the lemma.   % \eqref{ourbicomplex} is non-empty. 
\end{proof}
%%
%%%%%%%%%%%%%%%%%%%%%%%%%%%%%%%%%%%%%%%%%%%%%%%%%%%%%%%%%%%%%%%%%%%%%%%%%%%%%%%%%%%%%%%%%%%%%
\subsection{Proof of Lemma \ref{nezu}}
Next we give proof of Lemma \ref{nezu}. 
%%
%%%%%%%%%%%%%%%%%%%%%%%%%%%%%%%%%%%%%%%%%%%%%%%%%%%%%%%%%%%%%%%%%%%%%%%%%%%%%%%%%%%%%%%%%%%%%
\begin{proof}
%%
% 
%%
%%
%%%%%%%%%%%%%%%%%%%%%%%%%%%%%%%%%%%%%%%%%%%%%%%%%%%%%%%%%%%%%%%%%%%%%%%%%%%%%%%%%%%%%%%%%%%%%%%%%%%%
%%
Suppose we consider another transversal basis $\U'$ for $\F$.  
According to the definition,   
for each transversal section $U_i$ which belong to the original basis $\U$ in \eqref{ourbicomplex}   
 there exists a holonomy embedding 
\[
{h}'_i: U_i \hookrightarrow U_j', 
\]
i.e., it embeds $U_i$ into a section $U_j'$ of our new transversal basis $\U'$.  
Then consider the sequnce of holonomy embeddings $\left\{ h'_k \right\}$ such that 
\[
U'_{0} \stackrel{h'_1}{\hookrightarrow}   \ldots \stackrel{h'_k}{\hookrightarrow} U'_k.   
\]
For the combination of embeddings $\left\{ %\widetilde 
{h}'_i, i \ge 0 \right\}$ and  
\[
U_{0}\stackrel{h_1}{\hookrightarrow}  \ldots \stackrel{h_k}{\hookrightarrow} U_k,    
\]
we obtain commutative diagrams. 
Since the intersection in \eqref{ourbicomplex} is performed over all sets of homology mappings, 
then it is independent on the choice of a  transversal basis.
\end{proof}
%%
%%%%%%%%%%%%%%%%%%%%%%%%%%%%%%%%%%%%%%%%%%%%%%%%%%%%%%%%%%%%%%%%%%%%%%%%%%%%%%%%%%%%%%%%%%%%%%%
\subsection{Proof of Proposition \ref{nezc}}
Next, we prove Proposition \ref{nezc}. 
%%
%%%%%%%%%%%%%%%%%%%%%%%%%%%%%%%%%%%%%%%%%%%%%%%%%%%%%%%%%%%%%%%%%%%%%%%%%5555
\begin{proof}
%%
%%%%%%%%%%%%%%%%%%%%%%%%%%%%%%%%%%%%%%%%%%%%%%%%%%%%%%%%%%%%%%%%%%%%%%%%%%%%%%%%%%%%%%%%%%%%%%%%%%%%%%%%
Here we prove that for generic elements of a quasi-conformal grading-restricted vertex algebra 
 $\Phi$ and $\omega_W$ $\in \W_{z_1, \ldots, z_n}$ and  
   are canonical, i.e., independent on 
 changes 
\begin{eqnarray}
\label{zwrho}
z_i \mapsto w_i=\rho(z_i), \quad 1 \le i \le n, 
\end{eqnarray}
 of local coordinates of  $c_i(p_i)$ and $c_j(p'_j)$ at points $p_i$ and $p'_j$, $1 \le i \le n$, 
 $1 \le j \le k$. 
Thus the construction of the double complex spaces \eqref{ourbicomplex} is proved to be canonical too.
%%
%%%%%%%%%%%%%%%%%%%%%%%%%%%%%%%%%%%%%%%%%%%%%%%%%%%%%%%%%%%%%%%%%%%%%%%%%%%%%%%%%%%%%%%%%%%%%%%%%%%%%%%%%%%%%
%%%%%%%%%%%%%%%%%%%%%%%%%%%%%%%%%%%%%%%%%%%%%%%%%%%%%%%%%%%%%%%%%%%%%%%%%%%%%%%%%%%%
%%
%% 
%%
Let us denote by 
\[
\xi_i = \left( \beta_0^{-1} \; dw_i \right)^{\wt(v_i)}. 
\]
Recall the linear operator \eqref{poper} (cf. Appendix \ref{grading}).   
Define introduce the action of the transformations \eqref{zwrho} as 
\begin{eqnarray}
\label{deiv}
&& \Phi \left(dw_1^{\wt(v_1)} \otimes  v_1, w_1;
 \ldots ; dw_n^{\wt(v_n)} \otimes v_n, w_n \right)  %\in  C^n_k(V, \W, \F), 
\nn
&& \qquad =\left( \frac{d f(\zeta)}{d\zeta} \right)^{-L_W(0)}\; P(f(\zeta))  
  \; \Phi \left( \xi_1 %\left(\beta_0^{-1} \; dw_1\right)^{\wt(v_1)} 
\otimes  v_1, z_1;  
\ldots ; \xi_n %\left(\beta_0^{-1} \; dw_n\right)^{\wt(v_n)}  
\otimes v_n, z_n \right). 
\end{eqnarray}
We then obtain 
%%
%% 
%%
%
%%%%%%%%%%%%%%%%%%%%%%%%%%%%%%%%%%%%%%%%%%%%%%%%%%%%%%%%%%%%%%%%%%%%%%%%%%%%%%%%%%%%%
\begin{lemma}
An element \eqref{bomba} 
\[
\Phi \left(dz_1^{\wt(v_1)} \otimes  v_1, z_1;
 \ldots ; dz_n^{\wt(v_n)} \otimes v_n, z_n \right), 
\]
of $\W_{z_1, \ldots, z_n}$ is canonical 
is invariant under transformations \eqref{zwrho} of $\left({\rm Aut}\; \Oo^{(1)}\right)^{\times n}_{z_1,\ldots, z_n }$.
\end{lemma}
\begin{proof}
Consider \eqref{deiv}. 
%%%%%%%%%%%%%%%%%%%%%%%%%%%%%%%%%%%%%%%%%%%%%%%%%%%%%%%%%%%%%%%%%%%%%%%%%%%%%%%%%%%%%%%%%%%%%%%%%%%%%%%%%%%%%%%%%%%%%%%%
%%%%%%%%%%%%%%%%%%%%%%%%%%%%%%%%%%%%%%%%%%%%%%%%%%%%%%%%%%%%%%%%%%%%%%%%%%%%%%%%%%%%%%%%%%%%%%%%%%%%%%%%%%%%%%%%%%%%%%%%
%%
%%
First, note that 
\[
 f'(\zeta) =\frac{df(\zeta)}{d\zeta}= \sum\limits_{m \ge 0} (m+1) \; \beta_m \zeta^m. 
\]
By using the identification \eqref{repro} and 
 and the $L_W(-1)$-properties \eqref{lder1} and \eqref{loconj} we  
obtain 
\begin{eqnarray*}
&& 
 \langle w',
 \Phi \left( dw_1^{\wt(v_1)} \otimes  v_1, w_1; 
 \ldots; dw_n^{\wt(v_n)} \otimes v_n, w_n \right) 
\rangle 
\end{eqnarray*}
\begin{eqnarray*}
%\nn
%%
&& = 
  \langle w', f'(\zeta)^{-L_W(0)}\;  
%%
 %%\qquad 
%%
P(f(\zeta)) 
  \; \Phi \left( \xi_1%\left( \beta_0^{-1} \; dw_1 \right)^{\wt(v_1)} 
\otimes  v_1, z_1;
 \ldots ; \xi_n 
 \otimes v_n, z_n \right)  
\rangle   
%%
%\nn
%%%%%%%%%%%%%%%%%%%%%%%%%%%%%%%%%%%%%%%%%%%%%%%%%%%%%%%%%%%%%%%%%%%%%%%%%%%%%%
%%
\end{eqnarray*}

%%%%%%%%%%%%%%%%%%%%%%%%%%%%%%%%%%%%%%%%%%%%%%%%%%%%%%%%%%%%%%%%%%%%%%%%%%%%%%%%
\begin{eqnarray*}
&&  \qquad = 
 \langle w', 
\left(
\frac{d f(\zeta)} {d\zeta} \right)^{-L_W(0)} \; 
\Phi \left( dw_1^{\wt(v_1)} \otimes  v_1, 
 \sum\limits_{m \ge 0} (m+1)\; \beta_m z_1^{m+1}; 
 \ldots; \right. 
\nn
&& 
\left.  
\qquad \qquad \qquad \qquad \qquad \qquad \qquad \qquad  dw_n^{\wt(v_n)} \otimes v_n, 
 \sum\limits_{m \ge 0} (m+1)\; \beta_m z_n^{m+1} \right)  \rangle %\right) 
\end{eqnarray*}
\begin{eqnarray*}
&& 
%%%%%%%%%%%%%%%%%%%%%%%%%%%%%%%%%
%\qquad 
=  %R\left( 
 \langle w',  
\left( \frac{d f(\zeta)}{d \zeta} \right)^{-L_W(0)} \; 
\Phi \left(dw_1^{\wt(v_1)} \otimes  v_1, 
\left( \frac{d f(z_1)}{dz_1} \right) z_1;      \right. 
\nn
&& 
\left.  
\qquad \qquad \qquad \qquad \qquad \qquad \qquad \qquad  
 \ldots;  dw_n^{\wt(v_n)} \otimes v_n, 
 \left( \frac{d f(z_n)}{dz_n} \right) z_n\right)  \rangle %\right)
\end{eqnarray*}
%%
%%%%%%%%%%%%%%%%%%%%%%%%%%%%%%%%%%%%%
\begin{eqnarray*}
&&
%%%%%%%%%%%%%%%%%%%%%%%%%%%
\qquad =  
 \langle w',  
\Phi \left( \left( \frac{d f(z_1)}{dz_i}  \; dw_1 \right)^{-\wt(v_1)} \otimes  v_1, z_1;     
\right. 
\nn
&&
\left.  
\qquad \qquad \qquad \qquad \qquad \qquad \qquad \qquad \qquad
 \ldots; \left( \frac{d f(z_n)}{dz_n} \; dw_n \right)^{-\wt(v_n)}   \otimes v_n, z_n 
 \right) 
 \rangle %\right) 
%%
%%%%%%%%%%%%%%%%%%%%%%%%%%%%%%%%%
\nn
&& 
\qquad = %R\left( 
  \langle w',  \Phi 
\left(dz_1^{\wt(v_1)} \otimes  v_1, z_1;
 \ldots; dz_n^{\wt(v_n)} \otimes v_n, z_n \right) \rangle. %\right). 
\end{eqnarray*}
Thus we proved the Lemma.  
%%
%%%%%%%%%%%%%%%%%%%%%%%%%%%%%%%%%%%%%%%%%%%%%%%%%%%%%%%%%%%%%%%%%%%%%%%%%%%%%%%%%%%%%%%%%%%%%%%%
%%
\end{proof}
The elements $\Phi(v_1, z_1; \ldots; v_n, z_n)$ of $C^n_k(V, \W, \F)$ belong to the space $\W_{z_1, \ldots, z_n}$ 
and assumed to be composable with a set of vertex operators $\omega_W(v_j', c_j(p_j'))$, $1 \le j \le k$. 
Vertex operators $\omega_W(dc(p)^{\wt(v')} \otimes v_j', c_j(p_j'))$ constitute particular examples \
of mapping of $C^1_\infty(V, \W, \F)$ and, therefore, are invariant with respect to \eqref{zwrho}.  
Thus, the construction of spaces \eqref{ourbicomplex} is invariant under the action of the group
$\left( {\rm Aut} \; \Oo \right)^{\times n}_{z_1, \ldots, z_n}$. 
\end{proof}
%%
%%
%%%%%%%%%%%%%%%%%%%%%%%%%%%%%%%%%%%%%%%%%%%%%%%%%%%%%%%%%%%%%%%%%%%%%%%%%%%%%%%%%%%%%%%%%%%%%%%%%%
\subsection{Proof of Lemma \ref{subset}}
Finally, we give a proof of Lemma \ref{subset}. 
%%
%%%%%%%%%%%%%%%%%%%%%%%%%%%%%%%%%%%%%%%%%%%%%%%%%%%%%%%%%%%%%%%%%%%%%%%%%%%%%%%
\begin{proof}
Since $n$ is the same for both spaces in \eqref{susus}, it only remains 
to check that the conditions for \eqref{Inm} and \eqref{Jnm} for $\Phi(v_1, c_j(p_1);$  $\ldots;$ $v_n, c_j(p_n))$
of  composibility Definition \ref{composable} 
 with vertex operators are stronger for ${C}_{m}^{n}(V, \W, \U, \F)$ 
then for ${C}_{m-1}^{n}(V, \W, \U, \F)$.  
%%
%%%%%%%%%%%%%%%%%%%%%%%%%%%%%%%%%%%%%%%%%%%%%%%%%%%%%%%%%%%%%%%%%%%%%%%%%%%%%%%%%%%%%%%%%%%%%%%%%%%%%%%
%%
%%
In particular, 
%%
%%%%%%%%%%%%%%%%%%%%%%%%%%%%%%%%%%%%%%%%%%%%%%%%%%%%%%%%%%%%%%%
%%
in the first condition for \eqref{Inm} in definition of composability \ref{composabilitydef}
the difference between the spaces in \eqref{susus} is in indeces.  
Consider 
\eqref{psiii}. 
For ${C}_{m-1}^{n}(V, \W, \U, \F)$, the summations in idexes  
\[
{k_1}={l_{1}+\cdots +l_{i-1}+1}, \; ..., \; {k_i}={l_{1}+\cdots +l_{i-1}+l_{i}}, 
\]
for the coordinates $c_j(p_1)$, ..., $c_j(p_n)$ 
with  $l_{1}, \dots, l_{n}\in \Z_+$, such that $l_{1}+\cdots +l_{n}=n+(m-1)$, 
and vertex algebra elements 
$v_{1}, \dots, v_{n+(m-1)}$ 
are included in summation for indexes for ${C}_{m}^{n}(V, \W, \U, \F)$.  
%%%%
The conditions for the domains of absolute convergency for $\mathcal M$, i.e.,  
\[
|c_{l_{1}+\cdots +l_{i-1}+p}-\zeta_{i}| 
+ |c_{l_{1}+\cdots +l_{j-1}+q}-\zeta_{i}|< |\zeta_{i}
-\zeta_{j}|, 
\]
for $i$, $j=1, \dots, k$, $i\ne j$, and for $p=1, \dots,  l_i$ and $q=1, \dots, l_j$, 
 for the series \eqref{Inm}  
are more restrictive then for $m-1$. 
The conditions for $\mathcal I^n_{m-1}(\Phi)$ to be extended analytically  
to a rational function in $(c_1(p_{1})$,
 $\dots$, $c_{n+(m-1)} (p_{n+(m-1)}))$,
 with positive integers $N^n_{m-1}(v_{i}, v_{j})$,
depending only on $v_{i}$ and $v_{j}$ for $i, j=1, \dots, k$, $i\ne j$, 
are included in the conditions for $\mathcal I^n_m(\Phi)$. 

%%
%%%%%%%%%%%%%%%%%%%%%%%%%%%%%%%%%%%%%%%%%%%%%%%%%%%%%%%%%%%%%%%%%%%%%%%%%%%%%%%%%%%%%%%%%%%%%%%%
%%%%%%%%%%%%%%%%%%%%%%%%%%%%%%%%%%%%%%%%%%%%%%%%%%%%%%%%%%%%%%%%%%%%%%%%%%%%%%%%%%%%%%%%%%%%%%%%
Similarly, the second condition for \eqref{Jnm}, 
of is absolute convergency
and analytical extension to a
rational function 
in $(c_{1}(p_1), \dots, c_{m+n}(p_{m+n}))$, with the only possible poles at
$c_{i}(p_i)=c_{j}(p_j)$, 
of orders less than or equal to 
$N^n_m(v_{i}, v_{j})$, for $i, j=1, \dots, k$, $i\ne j$,  
for \eqref{Jnm}  
when 
\[
c_{i}(p_i)\ne c_{j}(p_j), \; i\ne j, \; 
|c_{i}(p_i)| > |c_{k}(p_k)|>0, 
\]
 for
 $i=1, \dots, m$, 
and 
$k=m+1, \dots, m+n$
includes the same condition for $\mathcal J^n_{m-1}(\Phi)$. 
Thus we obtain the conclusion of Lemma. 
\end{proof}
%%
%%%%%%%%%%%%%%%%%%%%%%%%%%%%%%%%%%%%%%%%%%%%%%%%%%%%%%%%%%%%%%%%%%%%%%%%%%%%%%%%%%%%%%%%%%%%%
%%%%%%%%%%%%%%%%%%%%%%%%%%%%%%%%%%%%%%%%%%%%%%%%%%%%%%%%%%%%%%%%%%%%%%%%%%%%%%%%%%%%%%%%%%%%%

%%%%%%%%%%%%%%%%%%%%%%%%%%%%%%%%%%%%%%%%%%%%%%%%%%%%%%%%%%%%%%%%%%%%%%%%%%%%%%%%%%%%%%%%%%%%%%%
%%%%%%%%%%%%%%%%%%%%%%%%%%%%%%%%%%%%%%%%%%%%%%%%%%%%%%%%%%%%%%%%%%%%%%%%%%%%%%%%%%%%%%%%%%%%%%%
\section{Appendix: Cohomological classes and connections}
\label{cohomological}
%%
%%%%%%%%%%%%%%%%%%%%%%%%%%%%%%%%%%%%%%%%%%%%%%%%%%%%%%%%%%%%%%%%%%%%%%%%%%%%%%%%%%%%%%%%%%%%%%
\subsection{Classes of grading-restricted vertex alebra cohomology}
%%

%%%%%%%%%%%%%%%%%%%%%%%%%%%%%%%%%%%%%%%%%%%%%%%%%%%%%%%%%%%%%%%%%%%%%%%%%%%%%%%%%%%%%%%%%%%%%%%
%%
In this section we describe certain classes associated to the first and the second vertex algebra cohomologies
 for codimension one foliations.   Let us give some further definitions. 
Usually, the cohomology classes for codimension one foliations \cite{G, CM, Ko} are introduced  by means of  
an extra condition (in particular, the orthogonality condition) applied to differential forms, and leading to 
the integrability condition. 
As we mentioned in Section \ref{coboundary}, it is a separate problem to introduce a product defined on one 
or among various spaces 
  $C^n_m(V, \W, \F)$ of \eqref{ourbicomplex}.  
Note that elements of $\mathcal E$ in \eqref{deltaproduct} and $\mathcal E_{ex}$ in \eqref{mathe2} can be seen as elements 
of spaces $C^1_\infty(V, \W, \F)$, i.e., maps composable with an infinite number of vertex operators. 
Though the actions of coboundary operators $\delta^n_m$ and $\delta^2_{ex}$ in \eqref{deltaproduct} and 
\eqref{halfdelta} are written in form of a product 
 (as in Frobenius theorem \cite{G}), and, in contrast to the case of differential forms, 
 it is complicated to use these products 
 for further formulation 
of cohomological invariants and derivation of analogues of the product-type invariants.   
Nevertheless, even with such a product yet missing, it is possible to introduce the lower-level cohomological 
classes of the form $\left[ \delta \eta \right]$ which are counterparts of the Godbillon class \cite{Galaev}.    
Let us give some further definitions. 
By analogy with differential forms, let us introduce
%%
%%%%%%%%%%%%%%%%%%%%%%%%%%%%%%%%%%%%%%%%%%%%%%%%%%%%%%%%%%%%%%%%%%%%%%%%%%55
\begin{definition}
\label{probab}
 We call a map 
\[
\Phi \in C_{k}^{n}(V, \W, \F),  
\] 
%%  pomenyat mozhet byt nazcanie exact na closed 
%%
%%
closed if it is a closed connection: 
\[
\delta^{n}_{k} \Phi=G(\Phi)=0.
\]  
For $k \ge 1$, we call it exact if there exists 
\[
\Psi \in  C_{k-1}^{n+1}(V, \W, \F), 
\]  
such that 
\[
\Psi=\delta^{n}_{k} \Phi,
\]
 i.e., $\Psi$ is a form of connection.  
\end{definition}
\begin{definition}
For $\Phi \in {C}^{n}_{k}(V, \W, \F)$ we call the cohomology class of mappings 
 $\left[ \Phi \right]$ the set of all closed forms that differ from $\Phi$ by an 
exact mapping, i.e., for $\Lambda \in {C}^{n-1}_{k+1}(V, \W, \F)$,  
\[
\left[ \Phi \right]= \Phi + \delta^{n-1}_{k+1} \Lambda. 
\]
\end{definition}
%%
%% 
%%%%%%%%%%%%%%%%%%%%%%%%%%%%%%%%%%%%%%%%%%%%%%%%%%%%%%%%%%%%%%%%%%%%%%%%%%%%%%%%%%%%%%%%%%%%%%%%%%%%%%%%
%%
As we will see in this section, there are cohomological classes, 
(i.e., $\left[\Phi \right]$, $\Phi \in {C}^{1}_{m}(V, \W, \F)$, $m\ge 0$),  
 associated with two-point connections and the first cohomology ${H}^{1}_{m}(V, \W, \F)$, and classes  
(i.e., $\left[\Phi \right]$, $\Phi \in {C}^{2}_{ex}(V, \W, \F)$),  
associated with transversal connections and the second cohomology ${H}^{2}_{ex}(V, \W, \F)$, 
of $\mathcal M/\F$. 
The cohomological classes we obtain are vertex algebra cohomology counterparts of the Godbillon class 
\cite{Ko, Galaev}   
for codimension one foliations.  
\begin{remark}
As it was discovered in \cite{BG, BGG}, 
 it is a usual situation when the existence of a connection (affine of projective) 
for codimension one foliations on smooth manifolds prevents corresponding cohomology classes from vanishing.  
 Note also, that for a few examples of codimension one foliations, the cohomology class $\left[ d\eta \right]$  
is always zero.  
\end{remark}
%%
%%%%%%%%%%%%%%%%%%%%%%%%%%%%%%%%%%%%%%%%%%%%%%%%%%%%%%%%%%%%%%%%%%%%%%%%%%%%%%%%%%%%%%%%%%%
%%
\begin{remark}
In contrast to \cite{BG}, our cohomological class is a functional of $v\in V$. 
That means that the actual functional form of $\Phi(v, z)$ (and therefore $\langle w', \Phi\rangle$, for 
$w'\in W'$)    
 varies with various choices of $v\in V$. 
That allows one to use it in order to distinguish types of leaves of $\mathcal M/\F$. 
\end{remark}
%%
%%%%%%%%%%%%%%%%%%%%%%%%%%%%%%%%%%%%%%%%%%%%%%%%%%%%%%%%%%%%%%%%%%%%%%%%%%%%%%%%%%%%%%%%%%%%%%%%%%%%%%%%
\subsection{Cohomology in terms of connections}
\label{connections}
In various situations it is sometimes effective to use an interpretation of cohomology in terms of 
connections. 
In particular in our supporting example of vertex algebra cohomology of codimension one foliations. 
It is convenient to introduce multi-point connections over a graded space  
and to express coboundary operators and cohomology in terms of connections: 
\[
\delta^n \phi  \in G^{n+1}(\phi),  
\]
\[
 \delta^n \phi= G(\phi). 
\]
Then the cohomology is defined as the factor space 
\[
H^n= {\mathcal Con}^n_{cl }/G^{n-1},   
\]
of closed multi-point connections with respect to the space of connection forms defined below. 
%%
%%%%%%%%%%%%%%%%%%%%%%%%%%%%%%%%%%%%%%%%%%%%%%%%%%%%%%%%%%%%%%%%%%%%%%%%%%%%%%%%%%%%%%%%%%%%%%%%%%%%%%%%
%%%%%%%%%%%%%%%%%%%%%%%%%%%%%%%%%%%%%%%%%%%%%%%%%%%%%%%%%%%%%%%%%%%%%%%%%%%%%%%%%%%%%%%%%%%%%%%%%%%%%%%%
\subsection{Multi-point holomorphic connections} 
%%%

%%%%%%%%%%%%%%%%%%%%%%%%%%%%%%%%%%%%%%%%%%%%%%%%%%%%%%%%%%%%%%%%%%%%%%%%%%%%%%%%%%%%%%%%%%%%%%%%%%%%%%%%%%%%%%%%%%%
We start this section with definitions of holomorphic multi-point connections on a smooth complex variety.  
Let $\mathcal X$ be a smooth complex variety and $\V \to \mathcal X$ a holomorphic vector bundle over $\mathcal X$. 
Let $E$ be the sheaf of holomorphic sections of $\V$.
 Denote by $\Omega$
 the sheaf of differentials on $\mathcal X$. 
A holomorphic connection $\nabla$ on $E$ is a $\C$-linear map 
\[
\nabla: E \to E \otimes \Omega, 
\]
 satisfying the Leibniz formula
\[
\nabla(f\phi)= \nabla f \phi + \phi \otimes dz, 
\]
for any holomorphic function $f$.  
Motivated by the definition of the holomorphic connection $\nabla$ 
 defined for a vertex algebra bundle (cf. Section 6, \cite{BZF}) over  
a smooth complex variety $\mathcal X$, we introduce the definition of 
the multiple point holomorphic connection over $\mathcal X$. 
%%
%%%%%%%%%%%%%%%%%%%%%%%%%%%%%%%%%%%%%%%%%%%%%%%%%%%%%%%%%%%%%%%%%%%%%%%%%%%%%%%%%%%%%%%%%%%
\begin{definition}
\label{mpconnection}
Let $\V$ be a holomorphic vector bundle over $\mathcal X$, and $\mathcal X_0$ its subvariety.  
A holomorphic multi-point 
connection $\mathcal G$ on $\V$    
is a $\C$-multi-linear map 
\[
\mathcal G: E  \to  E \otimes \Omega,
\]  
such that for any holomorphic function $f$, and two sections $\phi(p)$ and $\psi(p')$ at points  
$p$ and $p'$ on $\mathcal X_0$ correspondingly, we have 
\begin{equation}
\label{locus}
\sum\limits_{q, q' \mathcal X_0 \subset \mathcal X} \mathcal G\left( f(\psi(q)).\phi(q') \right) = f(\psi(p')) \; 
\mathcal G \left(\phi(p) \right) + f(\phi(p)) \; \mathcal G\left(\psi(p') \right),    
\end{equation}
where the summation on left hand side is performed over a locus of points $q$, $q'$ on $\mathcal X_0$. 
We denote by ${\mathcal Con}_{\mathcal X_0}(\mathcal S)$ the space of such connections defined over a
smooth complex variety $\mathcal X$. 
We will call $\mathcal G$ satisfying \eqref{locus}, a closed connection, and denote the space of such connections 
by ${\mathcal Con}^n_{\mathcal X_0; cl}$.  
\end{definition}
%%
%% eto tozhe perenesti i perepisat 
%%
%In particular, for the case $\mathcak S=M/\F$ we have also dependence on $z$. 
%% 
%%%%%%%%%%%%%%%%%%%%%%%%%%%%%%%%%%%%%%%%%%%%%%%%%%%%%%%%%%%%%%%%%%%%%%%%%%%%%%%%%%%%%%%%%%%%%%%%%%%%%%%%%%%%%%
%% tut nado kak sleduet napisat 
%%
Geometrically, for a vector bundle $\V$ defined over a complex variety $\mathcal X$,
a multi-point holomorphic connection \eqref{locus} relates two sections $\phi$ and $\psi$ of $E$ at points $p$ and $p'$
with a number of sections at a subvariety $\mathcal X_0$ of $\mathcal X$.  
\begin{definition}
We call
\begin{equation}
\label{gform}
G(\phi, \psi) = f(\phi(p)) \; \mathcal G\left(\psi(p') \right)  + f(\psi(p')) \; \mathcal G \left(\phi(p) \right)   
- \sum\limits_{q, q' \mathcal X_0 \subset \mathcal X} \mathcal G\left( f(\psi(q')).\phi(q) \right), 
\end{equation}
the form of a holomorphic connection $\mathcal G$. 
The space of form for $n$-point holomorphic connection forms will be denoted by $G^n(p, p', q, q')$.  
\end{definition}
Let us formulate another definition which we use in the next section: 
\begin{definition}
\label{transcon}
We call a multi-point holomorphic connection $\mathcal G$ 
 the transversal connection, i.e., when it satisfies 
\begin{equation}
\label{transa0}
f(\psi(p'))\; \mathcal G(\phi(p)) + f(\phi(p)) \; \mathcal G(\psi(p'))=0. 
\end{equation}
We call 
\begin{equation}
\label{transa}
G_{tr}(p, p')=(\psi(p'))\; \mathcal G(\phi(p)) + f(\phi(p)) \; \mathcal G(\psi(p')), 
\end{equation}
the form of a transversal connection. The space of such connections is denoted by $G^2_{tr}$.  
\end{definition}
%%

%%%%%%%%%%%%%%%%%%%%%%%%%%%%%%%%%%%%%%%%%%%%%%%%%%%%%%%%%%%%%%%%%%%%%%%%%%%%%%%%%%%%%%%%%%%%%
%%
%%%%%%%%%%%%%%%%%%%%%%%%%%%%%%%%%%%%%%%%%%%%%%%%%%%%%%%%%%%%%

%%%%%%%%%%%%%%%%%%%%%%%%%%%%%%%%%%%%%%%%%%%%%%%%%%%%%%%%%%%%%%%%%%%%%%%%%%%%%%%%%%%%%%%%%%%%%%%%%%
%%%%%%%%%%%%%%%%%%%%%%%%%%%%%%%%%%%%%%%%%%%%%%%%%%%%%%%%%%%%%%%%%%%%%%%%%%%%%%%%%%%%%%%%%%%%%%%%%%
%%%%%%%%%%%%%%%%%%%%%%%%%%%%%%%%%%%%%%%%%%%%%%%%%%%%%%%%%%%%%%%%%%%%%%%%%%%%%%%%%%%%%%%%%%%%%%%%%%%%%%%%%%%
\section{Appendix: 
A sphere formed from sewing of two spheres}
\label{sphere}
%%
%

%%%%%%%%%%%%%%%%%%%%%%%%%%%%%%%%%%%%%%%%%%%%%%%%%%%%%%%%%%%%%%%%%%%%%%%%%%%%%%%%%%%%%%%%%%%%%%%%%%%%%%%%%%%%%%%%%%%%%%%

%%
The matrix element for a number of vertex operators of a vertex algebra is usually associated \cite{FHL, FMS, TUY} 
with a vertex algebra character on a sphere. We extrapolate this notion to the case of $\W_{z_1, \ldots, z_n}$ spaces. 
In Section \ref{product} we explained that a space $\W_{z_1, \ldots, z_n}$ can be associated with a Riemann sphere with marked points, 
while the product of two such spaces is then associated with a sewing of such two spheres with a number 
of marked  
points 
and extra points with local coordinates identified with formal parameters of $\W_{x_1, \ldots, x_k}$ and $\W_{y_1, \ldots, y_n}$. 
In order to supply an appropriate geometric construction for the product,  
 we use the $\epsilon$-sewing procedure (described in this Appendix) for two initial spheres to obtain a matrix element associated with \eqref{gendef}. 

%%%%%%%%%%%%%%%%%%%%%%%%%%%%%%%%%%%%%%%%%%%%%%%%%%%%%%%%%%%%%%%%%%%%%%%%%%%%%%%%%%%%%%%%%%%%%%%%%%%%%%%%%%
\begin{remark}
In addition to the $\epsilon$-sewing procedure of two initial spheres, one can alternatively use 
the self-sewing procedure \cite{Y} for the sphere to get, at first, the torus, and then by sending parameters 
to appropriate limit by shrinking genus to zero. As a result, one obtains again the sphere but with a  
different parameterization.  In the case of spheres, such a procedure 
consideration of the product of $\W$-spaces so we focus in this paper on the $\epsilon$-formalizm only.  
\end{remark}
In our particular case of $\W$-values rational functions obtained from matrix elements \eqref{def}  
two initial 
auxiliary
 spaces we take Riemann spheres $\Sigma^{(0)}_a$, $a=1$, $2$, and the resulting 
space is formed by 
the sphere $\Sigma^{(0)}$ obtained by the procedure of sewing $\Sigma^{(0)}_a$. 
The formal parameters $(x_1, \ldots, x_k)$ and $(y_{1}, \ldots, y_n)$ are identified with 
local coordinates of $k$ and $n$ points on two initial spheres $\Sigma^{(0)}_a$, $a=1$, $2$ correspondingly.  
In the $\epsilon$ sewing procedure, some $r$ points 
among 
$(p_1, \ldots, p_k)$ 
may coincide with 
points 
among $(p'_{1}, \ldots, p'_n)$ %$(y_{k+1}, \ldots, y_n)$ 
when we identify the annuluses \eqref{zhopki}.   
This corresponds to the singular case of coincidence of $r$ formal parameters. 
%%
%%%%%%%%%%%%%%%%%%%%%%%%%%%%%%%%%%%%%%%%%%%%%%%%%%%%%%%%%%%%%%%%%%%%%%%%%%%%%
%%%%%%%%%%%%%%%%%%%%%%%%%%%%%%%%%%%%%%%%%%%%%%%%%%%%%%%%%%%%%%%%%%%%%%%%%%%%%
%%%%%%%%%%%%%%%%%%%%%%%%%%%%%%%%%%%%%%%%%%%%%%%%%%%%%%%%%%%%%%%%%%%%%%%

%%%%%%%%%%%%%%%%%%%%%%%%%%%%%%%%%%%%%%%%%%%%%%%%%%%%%%%%%%%%%%%%%%%%%%%%%%%%%%
%%
Consider the sphere formed by sewing together two initial spheres in the sewing scheme referred to 
as the $\epsilon$-formalism in \cite{Y}. 
Let $\Sigma_a^{(0)}$,  %=\mathbb{C}/{\Lambda}_{a}$ for 
$a=1$, $2$ 
be 
to initial spheres.  
Introduce a complex sewing
parameter $\epsilon$ where 
\[
|\epsilon |\leq r_{1}r_{2},
\]
%%
%%%%%%%%%%%%%%%%%%%%%%%%%%%%%%%%%%%%%%%%%%%%%%%%%%%%%%%%%%%%%%%%%%%%%%%%%%%%%%%%%%%%%%%%%%
Consider $k$ distinct points on $p_i \in  \Sigma_{1}^{(0)}$, $i=1, \ldots, k$, 
with local coordinates $(x_1, \ldots, x_{k}) \in F_{k}\C$,  
and distinct points $p_j \in  \Sigma_{2}^{(0)}$, $j=1, \ldots, n$,
with local coordinates $(y_{1},\ldots ,y_{n})\in F_{n}\C$,  
with
\[
\left\vert x_{i}\right\vert
\geq |\epsilon |/r_{2}, 
\]
%%
 %and
%
\[
\left\vert y_{i}\right\vert \geq |\epsilon |/r_{1}. 
\] 
%%%%%%%%%%%%%%%%%%%%%%%%%%%%%%%%%%%%%%%%%%%%%%%%%%%%%%%%%%%%%%%%%%%%%%%%
%% 
Choose a local coordinate $z_{a}\in \mathbb{C}$ %/{\Lambda}_{a}$ 
on $\Sigma^{(0)}_a$ in the
neighborhood of points $p_{a}\in\Sigma^{(0)}_a$, $a=1$, $2$. 
%%
 %and 
Consider the closed disks 
\[
\left\vert \zeta_{a} \right\vert \leq r_{a}, 
\]
 and excise the disk 
\begin{equation}
\label{disk}
\{
\zeta_{a}, \; \left\vert \zeta_{a}\right\vert \leq |\epsilon |/r_{\overline{a}}\}\subset 
\Sigma^{(0)}_a, 
\end{equation}
to form a punctured sphere  
\begin{equation*}
\widehat{\Sigma}^{(0)}_a=\Sigma^{(0)}_a \backslash \{\zeta_{a},\left\vert 
\zeta_{a}\right\vert \leq |\epsilon |/r_{\overline{a}}\}.
\end{equation*}
We use the convention 
\begin{equation}
\overline{1}=2,\quad \overline{2}=1.  
\label{bardef}
\end{equation}
Define the annulus
\begin{equation}
\label{zhopki}
\mathcal{A}_{a}=\left\{\zeta_{a},|\epsilon |/r_{\overline{a}}\leq \left\vert
\zeta_{a}\right\vert \leq r_{a}\right\}\subset \widehat{\Sigma}^{(0)}_a,
\end{equation}
and identify $\mathcal{A}_{1}$ and $\mathcal{A}_{2}$ as a single region 
$\mathcal{A}=\mathcal{A}_{1}\simeq \mathcal{A}_{2}$ via the sewing relation 
\begin{equation}
\zeta_{1}\zeta_{2}=\epsilon.   \label{pinch}
\end{equation}
In this way we obtain a genus zero compact Riemann surface 
\[
\Sigma^{(0)}=\left\{ \widehat {\Sigma}^{(0)}_1
\backslash \mathcal{A}_{1} \right\}
\cup \left\{\widehat{\Sigma}^{(0)}_2 \backslash 
\mathcal{A}_{2}\right\}\cup \mathcal{A}.
\] 
This sphere form a suitable geometrical model for the construction of a product of $\W$-valued rational forms
in Section \ref{product}. 
%%
%%%%%%%%%%%%%%%%%%%%%%%%%%%%%%%%%%%%%%%%%%%%%%%%%%%%%%%%%%%%%%%%%%%%%%%%%%%%%%%%%%

%%
%%%%%%%%%%%%%%%%%%%%%%%%%%%%%%%%%%%%%%%%%%%%%%%%%%%%%%%%%%%%%%%%%%%%%%%%%%%%%%%%%%%%%%%%%%%%%%%
%%%%%%%%%%%%%%%%%%%%%%%%%%%%%%%%%%%%%%%%%%%%%%%%%%%%%%%%%%%%%%%%%%%%%%%%%%%%%%%%%%%%%%%%%%%%%%%
\section{Proofs of Proposition \ref{katas}, Proposition \ref{pupa}  
Lemma \ref{functionformpropcor}, Lemma \ref{tarusa} 
}
\label{rasto}
%%

%%%%%%%%%%%%%%%%%%%%%%%%%%%%%%%%%%%%%%%%%%%%%%%%%%%%%%%%%%%%%%%%%%%%%%%%%%%%%%%%%%%%%%%%%%%%%%%%%%%%%%%%%
%%%%%%%%%%%%%%%%%%%%%%%%%%%%%%%%%%%%%%%%%%%%%%%%%%%%%%%%%%%%%%%%%%%%%%%%%%%%%%%%%%%%%%%%%%%%%%%%%%%%%%%%%
\subsection{Proof of Proposition \ref{katas}}
\begin{proof}
%%%%%%%%%%%%%%%%%%%%%%%%%%%%%%%%%%%%%%%
%%
%% 
%%
 By using \eqref{lder1} for $\Phi(v_{1},  x_{1};  \ldots; v_{k}, x_{k})$
and $\Psi(v'_{1},  y_{1}; \ldots;  v'_{n},  y_{n})$,  
 we consider 
%%%%%%%%%%%%%%%%%%%%%%%%%%%%%%%%%%%%%%%%%%%%%%%%%%%%%%%%%%%%%%%%%%%%%
%% 

%%%%%%%%%%%%%%%%%%%%%%%%%%%%%%%%%%%%%%%%%%%%%%%%%%%%%%%%%%%%%%%%%%%%%
%%
\begin{eqnarray}
\label{Z2n_pt_eps1q00000}
 & &
\langle w',\partial_{
l} %
 \Theta( v_{1},  x_{1};  \ldots;  v_k, x_{k};   v'_1, y_1; \ldots;  v'_{n},  y_{n}; \epsilon) \rangle 
\nn
&& 
%%
%%
%%%%%%%%%%%%%%%%%%%%%%%%%%%%%%%%%%%%%%%%%%%%%%%%%%%%%%%%%%%%%%%%%%%%%%%%%%%%%%%%%%%%5
\nn
 & &  = \sum_{l\in \mathbb{Z}
} \epsilon^{l} \sum_{u\in V_l}   
\langle w', \partial^{\delta_{l,i}}_{x_i} 
Y^{W}_{WV}\left( 
   \Phi ( v_{1},  x_{1};  \ldots; v_{k}, x_{k}), \zeta_1 \right) \; u \rangle  
%%
%\notag 
%%
\nn
& &
\qquad   \qquad 
\langle w', \partial^{\delta_{l,j}}_{y_j} 
Y^{W}_{WV}\left( 
\Psi
(v'_{1},  y_{1}; \ldots;  v'_{n},  y_{n}),  \zeta_2 \right) \overline{u} \rangle    
%%
%%%%%%%%%%%%%%%%%%%%%%%%%%%%%%%%%%%%%%%%%%%%%%%%%%%%%%%%%%%%%%%%%%%%%%%%%%%%%%%%%%%%%%%
%%%%%%%%%%%%%%%%%%%%%%%%%%%%%%%%%%%%%%%%%%%%%%%%%%%%%%%%%%%%%%%%%%%%%%%%%%%%%%%%%%%%5
\nn
 & &  = \sum_{l\in \mathbb{Z}
} \epsilon^{l} \sum_{u\in V_l}   
\langle w', \partial^{\delta_{l,i}}_{x_i} 
Y_{W}\left( u, - \zeta_1 \right) 
   \Phi ( v_{1},  x_{1};  \ldots; v_{k}, x_{k})) \; u \rangle  
%%
%\notag 
%%
\nn
& &
\qquad   \qquad %\cdot 
\langle w', \partial^{\delta_{l,j}}_{y_j}  
Y_{W}\left(\overline{u}, - \zeta_2  \right) 
\Psi
(v'_{1},  y_{1}; \ldots;  v'_{n},  y_{n})  \rangle    
%%
%%%%%%%%%%%%%%%%%%%%%%%%%%%%%%%%%%%%%%%%%%%%%%%%%%%%%%%%%%%%%%%%%%%%%%%%%%%%%%%%%%%%%%%
\nn
 & &  = \sum_{l\in \mathbb{Z}}  
\epsilon^{l} \sum_{u\in V_l}   
\langle w', 
Y^{W}_{WV}\left( 
 \partial^{\delta_{l,i}}_{x_i}  \Phi ( v_{1},  x_{1};  \ldots; v_{k}, x_{k}), \zeta_1 \right) \; u \rangle  
\nn
& &
\qquad   \qquad %\cdot 
\langle w', 
%%%
Y^{W}_{WV}\left( \partial^{\delta_{l,j}}_{y_j} 
\Psi 
(v'_{1},  y_{1}; \ldots;  v_{n},  y_{n}),  \zeta_2 \right) \overline{u} \rangle   
\nn
& &  
=  \sum_{l\in \mathbb{Z}
} \epsilon^{l} \sum_{u\in V_l}   
\langle w',
 Y^{W}_{WV}\left(  
\Phi (v_{1}, x_{1}; \ldots; \left(L_V{(-1)}\right)^{\delta_{l,i}} v_i, x_i; \ldots; v_{k}, x_{k}), \zeta_1 \right) \; u \rangle 
%%
%\notag 
%%
\nn
& &
\qquad   \qquad %\cdot 
\langle w', Y^{W}_{WV}\left( 
\Psi
(v'_{1}, y_{1}; \ldots; \left(L_V{(-1)}\right)^{\delta_{l,j}} v'_j, y_j;   \ldots; v'_{n}, y_{n}),  \zeta_2 \right) \overline{u} \rangle   
\nn
& &  
=    
\langle w',
\Theta (
v_{1},  x_{1};  \ldots;  \left(L_V{(-1)}\right)_l;  \ldots;  v'_{n},  y_{n};
\epsilon) \rangle,   
\end{eqnarray}
where $\left(L_V{(-1)}\right)_l$ acts on the $l$-th entry of  
$(v_{1},  \ldots;  v_k;   v'_1,  \ldots,   v'_{n})$. 
%%
%%%%%%%%%%%%%%%%%%%%%%%%%%%%%%%%%%%%%%%%%%%%%%%%%%%%%%%%%%%%%%%%%%%%%%%%%%%%%%%%%%%%%5
%%
Summing over $l$ we obtain
\begin{eqnarray}
\label{lder2f0}
&&
\sum\limits_{l=1}^{k+n} \partial_{
{l}} 
\Theta (v_{1}, x_{1}; \ldots; 
v_{k}, x_{k};
v'_{1}, y_{1}; \ldots; 
 v'_{n}, y_{n}; \epsilon) \rangle 
\nn
&&= 
\sum\limits_{l=1}^{k+n} \langle w',
\Theta (v'_{1}, x_{1}; \ldots; \left(L_V{(-1)}\right);    \ldots;  v'_{n}, y_{n}; \epsilon) \rangle 
\nn
&&
=\langle w',
 L_{W}{(-1)}.\Theta(v_{1}, x_{1};  \ldots ; v_{k}, x_{k};   v'_{1},  y_{1}; \ldots; v'_{n},  y_{n}; \epsilon) \rangle.   
\end{eqnarray}
%%
%%%%%%%%%%%%%%%%%%%%%%%%%%%%%%%%%%%%%%%%%%%%%%%%%%%%%%%%%%%%%%%%%%555
%%%%%%%%%%%%%%%%%%%%%%%%%%%%%%%%%%%%%%%%%%%%%%%%%%%%%%%%%%%%%%%%%%%%%
%
%%
%%
%
%% 
%%
%%%%%%%%%%%%%%%%%%%%%%%%%%%%%%%%%%%%%%%%%%%%%%%%%%%%%%%%%%%%%%%%%%%%%%%%%%%%%%%%%%%%%%%%%%%%%%%%%%%%%%
Due to \eqref{loconj}, \eqref{locomm}, \eqref{dubay}, \eqref{condip}, and \eqref{aprop}, %Z2n_pt_eps}, 
we have 
%%
%%%%%%%%%%%%%%%%%%%%%%%%%%%%%%%%%%%%%%%%%%%%%%%%%%%%%%%%%%%%%%%%%%%%
%%%%%%%%%%%%%%%%%%%%%%%%%%%%%%%%%%%%%%%%%%%%%%%%%%%%%%%%%%%%%%%%%%%%
%%
\begin{eqnarray*}
&& \langle w', 
 \Theta ( z^{L_V{(0)} } v_{1}, z \; x_{1};  \ldots;  z^{L_V{(0)}} v_{k},  z\; x_{k};
z^{L_V{(0)}} v'_{1}, z \; y_{1};  \ldots;  z^{L_V{(0)}} v'_{n},  z\; y_{n}; 
 \epsilon) \rangle  
\end{eqnarray*}
\begin{eqnarray*}
%\nn
%%
%%%%%%%%%%%%%%%%%%%%%%%%%%%%%%%%%%%%%%%%%%%%%%%%
 & &  = \sum_{l\in \mathbb{Z}
} \epsilon^{l} \sum_{u\in V_l}   
\langle w', Y^{W}_{WV}\left( 
 \Phi (z^{L_V{(0)}} v_{1}, z\; x_{1};  \ldots; z^{L_V{(0)}} v_{k}, z\; x_{k}), \zeta_1 \right) \; u \rangle  
%%
%\notag 
%%
\nn
& &
\qquad   \qquad %\cdot 
\langle w', Y^{W}_{WV}\left( 
\Psi
(z^{L_V{(0)}} v'_{1}, z \; y_{1}; \ldots; z^{L_V{(0)}} v'_{n}, z\; y_{n}),  \zeta_2 \right) \overline{u} \rangle   
\end{eqnarray*}
\begin{eqnarray*}
%\nn
%%
& &  
=  \sum_{l\in \mathbb{Z}
} \epsilon^{l} \sum_{u\in V_l}   
\langle w', Y^{W}_{WV}\left( z^{L_V{(0)}} 
\Phi (v_{1}, x_{1};  \ldots; v_{k}, x_{k}), \zeta_1 \right) \; u \rangle 
%%
%\notag 
%%
\nn
& &
\qquad   \qquad %\cdot 
\langle w', Y^{W}_{WV}\left( z^{L_V{(0)}}
\Psi
(v'_{1}, y_{1}; \ldots; v'_{n}, y_{n}),  \zeta_2 \right) \overline{u} \rangle   
%%
%%%%%%%%%%%%%%%%%%%%%%%
\end{eqnarray*}
%%%%%%%%%%%%%%%%%%%%%%%%%%%%%%%%%%%%%%%%
\begin{eqnarray*}
%\nn
& &  
=  \sum_{l\in \mathbb{Z}
} \epsilon^{l} \sum_{u\in V_l}   
\langle w', e^{\zeta_1 L_W{(-1)}}  Y_{W} \left(  u, -\zeta_1 \right) 
z^{L_V{(0)}}  
\Phi (v_{1}, x_{1};  \ldots; v_{k}, x_{k})  \rangle 
%%
%\notag 
%%
\nn
& &
\qquad   \qquad %\cdot 
\langle w', e^{\zeta_2 L_W{(-1)}} \; Y_{W}\left( \overline{u}, -\zeta_2  \right) z^{L_V{(0)}}  
\Psi
(v'_{1}, y_{1}; \ldots; v'_{n}, y_{n})   \rangle    
\end{eqnarray*}
\begin{eqnarray*}
%\nn
%&&
%%
%\nn
%%%%%%%%%%%%%%%%%%%%%%%%%%%%%%%%%%%%%%%%%%%%%%%%%%%%%%%%%%%%%%%5
%%
& &  
=  \sum_{l\in \mathbb{Z}
} \epsilon^{l} \sum_{u\in V_l}   
\langle w', e^{\zeta_1 L_W{(-1)}}  z^{L_V{(0)}} Y_{W} \left(  z^{-L_V{(0)}} u, -z \; \zeta_1 \right) 
\Phi (v_{1}, x_{1};  \ldots; v_{k}, x_{k})  \rangle 
%%
%\notag 
%%
\nn
& &
\qquad   \qquad %\cdot 
\langle w', e^{\zeta_2 L_W{(-1)}}\; z^{L_W{(0)}}\;   Y_{W}\left( z^{-L_V{(0)}} \overline{u}, -z \; \zeta_2  \right)  
\Psi
(v'_{1}, y_{1}; \ldots; v'_{n}, y_{n})   \rangle   
%%
%%%%%%%%%%%%%%%%%%%%%%%%%%%%%%%%%%%%%%%%%%%%%%%%%%%%%%%%%%%%%%%%%%%%%%%%
\nn
& &  
=  \sum_{l\in \mathbb{Z}
} \epsilon^{l} \sum_{u\in V_l}   
\langle w', e^{\zeta_1 L_W{(-1)}}  z^{L_W{(0)}} z^{-{\rm wt} u} \; Y_{W} \left(   u, -z \; \zeta_1 \right) 
\Phi (v_{1}, x_{1};  \ldots; v_{k}, x_{k})  \rangle 
%%
%\notag 
%%
\nn
& &
\qquad   \qquad %\cdot 
\langle w', e^{\zeta_2 L_W{(-1)}}\; z^{L_W(0)}\;  z^{-{\rm wt}\overline{u} } \; Y_{W}\left(  \overline{u}, -z \; \zeta_2  \right)  
\Psi
(v'_{1}, y_{1}; \ldots; v'_{n}, y_{n})   \rangle 
%%%%%%%%%%%%%%%%%%%%%%%%%%%%%%%%%%%%%%%%%%%%%%%%%%%%%%%%%%%%%%%
\end{eqnarray*}
\begin{eqnarray*}
%%
%\nn
& &  
=  \sum_{l\in \mathbb{Z}
} \epsilon^{l} \sum_{u\in V_l}   
\langle w', z^{L_W(0)} e^{\zeta_1 L_W{(-1)}} Y_{W} \left(  u, -z \zeta_1 \right) 
\Phi (v_{1}, x_{1};  \ldots; v_{k}, x_{k})  \rangle 
%%
%\notag 
%%
\nn
& &
\qquad   \qquad %\cdot 
\langle w', z^{L_W(0)} e^{\zeta_2 L_W{(-1)}}  Y_{W}\left( \overline{u}, -z \zeta_2  \right)  
\Psi
(v'_{1}, y_{1}; \ldots; v'_{n},y_{n}),   \rangle   
%%
%%
%\nn
%&&
%%%%%%%%%%%%%%%%%%%%%%%%
\end{eqnarray*}
\begin{eqnarray*}
%\nn
%%
& &  
=  \sum_{l\in \mathbb{Z}
} \epsilon^{l} \sum_{u\in V_l}   
\langle w', z^{L_W(0)} \; Y^{W}_{WV}\left( 
\Phi (v_{1}, x_{1};  \ldots; v_{k}, x_{k}), z \zeta_1 \right) \; u \rangle 
\notag 
\nn
& &
\qquad   \qquad %\cdot 
\langle w', z^{L_W(0)} \;  Y^{W}_{WV}\left( 
\Psi
(v'_{1}, y_{1}; \ldots; v'_{n},y_{n}),  z \zeta_2 \right) \overline{u} \rangle   
%%
%%%%%%%%%%%%%%%%%%%%%%%%%
%%%%%%%%%%%%%%%%%%%%%%%%
\end{eqnarray*}
\begin{eqnarray*}
%\nn
%%
& &  
=  \sum_{l\in \mathbb{Z}
} \epsilon^{l} \sum_{u\in V_l}   
\langle w', z^{L_W(0)} \; Y^{W}_{WV}\left( 
\Phi (v_{1}, x_{1};  \ldots; v_{k}, x_{k}),  \zeta'_1 \right) \; u \rangle 
%%
%\notag 
%%
\nn
& &
\qquad   \qquad %\cdot 
\langle w', z^{L_W(0)} \;  Y^{W}_{WV}\left( 
\Psi
(v'_{1}, y_{1}; \ldots; v'_{n},y_{n}),  \zeta'_2 \right) \overline{u} \rangle   
\end{eqnarray*}
\begin{eqnarray*}
%\nn
&&
=\langle w', \left( z^{L_W(0)} \right). 
\Theta (v_{1}, x_{1};  \ldots;  v_{k}, x_{k}; v'_{1}, y_{1}; \ldots; v'_{n}, y_{n}; \epsilon) \rangle.  
%%
%\nn
%&& 
%%
\end{eqnarray*}
%%
%
%%
%%%%%%%%%%%%%%%%%%%%%%%%%%%%%%%%%%%%%%%%%%%%%%%%%%%%%%%%%%%%%%%%%%%%%%%%%%%%%%%%%%%%%%%
With \eqref{pinch}, 
we obtain \eqref{loconj} for \eqref{Z2n_pt_eps1q1}.   
\end{proof}
%%

%%%%%%%%%%%%%%%%%%%%%%%%%%%%%%%%%%%%%%%%%%%%%%%%%%%%%%%%%%%%%%%%%%%%%%%%%%%%%%%%%%%%%%%%%%%%%%%%%%%%%%%%
%%%%%%%%%%%%%%%%%%%%%%%%%%%%%%%%%%%%%%%%%%%%%%%%%%%%%%%%%%%%%%%%%%%%%%%%%%%%%%%%%%%%%%%%%%%%%%%%%%%%%%%%
\subsection{Proof of Proposition \ref{pupa}}
\begin{proof}
%%
%%%%%%%%%%%%%%%%%%%%%%%%%%%%%%%%%%%%%%%%%%%%%%%%%%%
%%
%%%%%%%%%%%%%%%%%%%%%%%%%%%%%%%%%%%%%%%%%%%%%%%%%%%%%%%%%%%%%%%%%%%%%%%%%%%%%%%%%%%%%%%%%%%%%%%%%%%%%%%%%
Note that due to Proposition \ref{pupa} 
\begin{eqnarray*}
\Phi (v_{1}, x'_{1};  \ldots; v_{k}, x'_{k}) &=&
\Phi (v_{1}, x_{1};  \ldots; v_{k}, x_{k}), 
\nn
\Psi (v_{1}, y'_{1};  \ldots; v_{n}, y'_{n}) &=&
\Psi (v_{1}, y_{1};  \ldots; v_{n}, y_{n}).   
\end{eqnarray*}
Thus,  
%%
%%%%%%%%%%%%%%%%%%%%%%%%%%%%%%%%%%%%%%%%%%%%%%%%%%%%%%%%
%%%%%%%%%%%%%%%%%%%%%%%%%%%%%%%%%%%%%%%%%%%%%%%%%%%%%%%%%%%%%%%%%%%%%%%%
%%
\begin{eqnarray*}
&& \langle w', \Theta(v_1, x'_1; \ldots,;  v_k, x'_k; v'_1, y'_1; \ldots; v'_n, y'_{n}; \epsilon) \rangle
\nn
 & &  =  \sum_{l\in \mathbb{Z}
} \epsilon^{l} \sum_{u\in V_l}   
\langle w', Y^{W}_{WV}\left( 
\Phi (v_{1}, x'_{1};  \ldots; v_{k}, x'_{k}), \zeta_1 \right) \; u \rangle 
%%
%\notag 
%%
\nn
& &
\qquad   \qquad %\cdot 
\langle w', Y^{W}_{WV}\left( 
\Psi
(v'_{1}, y'_{1}; \ldots; v'_{n}, y'_{n}),  \zeta_2 \right) \overline{u} \rangle   
\nn
& &  
=  \sum_{l\in \mathbb{Z}
} \epsilon^{l} \sum_{u\in V_l}   
\langle w', Y^{W}_{WV}\left( 
\Phi (v_{1}, x_{1};  \ldots; v_{k}, x_{k}), \zeta_1 \right) \; u \rangle 
%%
%\notag 
%%
\nn
& &
\qquad   \qquad %\cdot 
\langle w', Y^{W}_{WV}\left( 
\Psi
(v'_{1}, y_{1}; \ldots; v'_{n}, y_{n}),  \zeta_2 \right) \overline{u} \rangle   
\nn
&&
=
\langle w', \Theta(v_1, x_1; \ldots,;  v_k, x_k; v'_1, y_1; \ldots; v'_n, y_{n}; \epsilon) \rangle. 
\end{eqnarray*}
%%%%%%%%%%%%%%%%%%%%%%%%%%%%%%%%%%%%%%%%%%%%%%%%%%%%%%%%%%%%%%%%%%%%%%%%%%%%%%%%
Thus, the product \eqref{Z2n_pt_eps1q1} is invariant under \eqref{zwrho}. 
%%
%Thus the construction of the double complex spaces \eqref{ourbicomplex} is proved to be canonical too.
%%
%%
\end{proof}
%%%
%%%%%%%%%%%%%%%%%%%%%%%%%%%%%%%%%%%%%%%%%%%%%%%%%%%%%%%%%%%%%%%%%%%%%%%%%%%%%%%%%%%%%%%%%%%%%%
%%%%%%%%%%%%%%%%%%%%%%%%%%%%%%%%%%%%%%%%%%%%%%%%%%%%%%%%%%%%%%%%%%%%%%%%%%%%%%%%%%%%%%%%%%%%%%
\subsection{Proof of Lemma \ref{functionformpropcor}}
\begin{proof}
Let $\widetilde{v}_{i} \in V$, $1 \le i \le k$,  
$\widetilde{v}_{j} \in V$, $1 \le j \le k$,   %and $v'_{j}\in V$, $1 \le j \le n$,
and 
 $z_{i}$, 
$z_j$ are corresponding formal parameters.   
We show that  
the $\epsilon$-product of 
$\Phi(\widetilde{v}_{1}, z_1;  \ldots;  \widetilde{v}_{k}, z_k)$ and 
$\Psi(\widetilde{v}_{k+1}, z_{k+1}$;   $\ldots$; $\widetilde{v}_{n}, z_n)$, i.e.,
 the $\W_{z_1, \ldots, z_{k+n-r}}$-valued  %formal  
differential form 
\begin{eqnarray}
& & {\Theta}
 ((\widetilde{v}_{1}, z_1;  \ldots;  \widetilde{v}_{k}, z_k); 
(\widetilde{v}_{k+1}, z_{k+1};  \ldots; \widetilde{v}_{n}, z_n);  
\zeta_{1}, \zeta_{2}; \epsilon )    
\end{eqnarray}
is independent of the choice of $0 \le k \le n$.  
 Consider  
\begin{eqnarray}
\label{ishodnoe}
&& 
\langle w',  
\Theta (\widetilde{v}_{1}, z_{1};\ldots ;  \widetilde{v}_{k}, z_{k}; \widetilde{v}_{k+1}, z_{k+1};
 \ldots ; \widetilde{v}_{n}, z_{n}; \zeta_1, \zeta_2; \epsilon)  \rangle    
\nn
%%%%%%%%%%%%%%%%%%%%%%%%%
&&\qquad  = 
 \sum_{l\in \mathbb{Z} 
} \epsilon^{l} \sum_{u \in V_l}   
\langle w', Y^{W}_{WV}\left( 
\Phi (\widetilde{v}_{1}, z_{1};  \ldots; \widetilde{v}_{k}, z_{k}), \zeta_1 \right) \; u \rangle 
%%
%\notag 
%%
\nn
& &
\qquad   \qquad  \qquad %\cdot 
\langle w', Y^{W}_{WV}\left( 
\Psi
(\widetilde{v}_{k+1}, z_{k+1}; \ldots; \widetilde{v}_{n}, z_{n}),  \zeta_2 \right) \overline{u} \rangle.    
%%
%\nn
%%
%&&
%%
\end{eqnarray}
%%
%%%%%%%%%%%%%%%%%%%%%%%%%%%%%%%%%%%%%%%%%%%%%%%%%%%%%%%%%%%%%%%%%%%%%%%%%%%%%%%%%%%%%%%%%%%
%%%%%%%%%%%%%%%%%%%%%%%%%%%%%%%%%%%%%%%%%%%%%%%%%%%%%%%%%%%%%%%%%%%%%%%%%%%%%%%%%%%%%%%%%%%
On the other hand, for $0 \le m \le k$,  consider
%%%
%%
\begin{eqnarray*}
%%
%%%%%%%%%%%%%%%%%%%%%%%%%
&& 
 \sum_{l\in \mathbb{Z} 
} \epsilon^{l} \sum_{u \in V_l}   
\langle w', Y^{W}_{WV}\left( 
\Phi (\widetilde{v}_{1}, z_{1};  \ldots; \widetilde{v}_{m}, z_{m}), \zeta_1 \right) \; u \rangle 
%%
%\notag 
%%
\nn
& &
\qquad   \qquad  \qquad %\cdot 
\langle w', Y^{W}_{WV}\left( 
\Psi
(   \widetilde{v}_{m+1}, z'_{m+1}; \ldots; \widetilde{v}_{k}, z'_{k};  
\widetilde{v}_{k+1}, z_{1}; \ldots; \widetilde{v}_{n}, z_{n}),  \zeta_2 \right) \overline{u} \rangle   
\nn
&&
\qquad =\langle w',  
\Theta (\widetilde{v}_{1}, z_{1};\ldots ;  \widetilde{v}_{m}, z_{m};
\widetilde{v}_{m+1}, z'_{m+1}; \ldots; \widetilde{v}_{k}, z'_{k};
 \widetilde{v}_{k+1}, z_{k+1};
 \ldots ; \widetilde{v}_{n}, z_{n})  \rangle. 
\end{eqnarray*}
The last is the $\epsilon$-product \eqref{Z2n_pt_eps1q1} of 
$\Phi (\widetilde{v}_{1}, z_{1};  \ldots; \widetilde{v}_{m}, z_{m})  \in 
\W_{z_{1},  \ldots, z_{m}}$ and  $\Psi
(\widetilde{v}_{m+1}, z'_{m+1} $; $ \ldots $; $ \widetilde{v}_{k}, z'_{k} $; $   
\widetilde{v}_{k+1}, z_{1} $; $ \ldots $;  $\widetilde{v}_{n}, z_{n}) \in \W_{z'_{m+1}, \ldots,  z'_{k};  
 z_{1}, \ldots, z_{n}}$. 
Let us apply the invariance with respect to a subgroup of 
$\left({\rm Aut}\; \Oo^{(1)}\right)^{\times (k+n)}_{z_1, \ldots, z_{k+n}}$,   
with $(z_1, \ldots, z_m)$ and $(z_{k+1}, \ldots, z_n)$ remaining unchanged. 
Then we obtain the same product \eqref{ishodnoe}. 
%%
%%%%%%%%%%%%%%%%%%%%%%%%%%%%%%%%%%%%%%%%%%%%%%%%%%%%%%%%%%%%%%%%%%%%%%%%%%%%%%%%%%%%%%%%%%%%%%%%%%%%%
%%
\end{proof}
%%

%%%%%%%%%%%%%%%%%%%%%%%%%%%%%%%%%%%%%%%%%%%%%%%%%%%%%%%%%%%%%%%%%%%%%%%%%%%%%%%%%%%%%%%%%
%%%%%%%%%%%%%%%%%%%%%%%%%%%%%%%%%%%%%%%%%%%%%%%%%%%%%%%%%%%%%%%%%%%%%%%%%%%%%%%%%%%%%%%%%
\subsection{Proof of Lemma \ref{tarusa}}
\begin{proof}
 For arbitrary $w' \in W'$, we have % consider 
\begin{eqnarray*}
&&
 \sum_{\sigma\in J_{k+n; s}^{-1}}(-1)^{|\sigma|}
 \langle w', 
\Theta\left(v_{\sigma(1)}, x_{\sigma(1)}; \ldots; v_{\sigma(k)}, x_{\sigma(k)};  
v'_{\sigma(1)},  y_{\sigma(1)};  \ldots; v'_{\sigma(n)},  y_{\sigma(n)})\right)
 \rangle 
%%
%%%%%%%%%%%%%%%%%%%%%%%%%%%%%%%%%%%%%%%%%%%%%%%%%%%%%%%%%%%%%%%%%%%%%%%%%%%%%%%%%%%%
%%%%%%%%%%%%%%%%%%%%%%%%%%%%%%%%%%%%%%%%%%%%%%%%%%%%%%%%%%%%%%%%%%%%%%%%%%%%%%%%%%%%%
%%
\nn
&&
\nn
&&
%%%%%%%%%%%%%%%%%%%%%%%%%%%%%%%%%%%%%%%%%%%%%%%%%%%%%%%%%%%%%%%%%%%%%%%%%%%%%%%%%%%%
=
\sum_{\sigma\in J_{k+n; s}^{-1}}(-1)^{|\sigma|}  \;  %\sum\limits_{u\in V} 
 \sum_{l \in \Z }  \epsilon^l 
 \sum_{u\in V_l } 
  \langle w', Y^{W}_{WV} 
\left(\Phi(v_{\sigma(1)}, x_{\sigma(1)}; \ldots; v_{\sigma(k)},  x_{\sigma(k)}), \zeta_1 \right) u  
\rangle 
\nn
&&
 \qquad \qquad \qquad  \langle w', Y^{W}_{WV} 
\left(\Psi(v'_{\sigma(1)}, y_{\sigma(1)}; \ldots; v'_{\sigma(n)},  y_{\sigma(n)}), \zeta_2  
\right) \overline{u} \rangle 
\nn
&&
%%%%%%%%%%%%%%%%%%%%%%%%%%%%%%%%%%%%%%%%%%%%%%%%%%%%%%%%%55
=%
\sum_{l \in \Z }  \epsilon^l 
 \sum_{u\in V_l } 
\sum_{\sigma\in J_{k+n; s}^{-1}}(-1)^{|\sigma|}   
  \langle w', e^{\zeta_1 L_W{(-1)}} \; Y_{W}(u, -\zeta_1) \; 
\Phi(v_{\sigma(1)}, x_{\sigma(1)}; \ldots; v_{\sigma(k)},  x_{\sigma(k)})   
\rangle 
\nn
&&
 \qquad \qquad \qquad  \langle w',  e^{\zeta_2 L_W{(-1)}} \; Y_{W}(\overline{u}, -\zeta_2) \;   
 \Psi(v'_{\sigma(1)}, y_{\sigma(1)}; \ldots; v'_{\sigma(n)},  y_{\sigma(n)})
 \rangle 
%%
%%%%%%%%%%%%%%%%%%%%%%%%%%%%%%%%%%%%%%%%%%%%%%%%%%%%%%%%%%%%%%%%%%%%%%%%%%%%%%%%%%%%%%%%%%%%
%%%%%%%%%%%%%%%%%%%%%%%%%%%%%%%%%%%%%%%%%%%%%%%%%%%%%%%%%%%%%%%%%%%%%%%%%%%%%%%%%%%%%%%%%%%%
\nn
&&
=%\sum\limits_{u\in V}
\sum_{l \in \Z }  \epsilon^l 
 \sum_{u\in V_l } 
  \langle w', e^{\zeta_1 L_W{(-1)}} \; Y_{W}(u, -\zeta_1) \; 
\sum_{\sigma\in J_{k; s}^{-1}}(-1)^{|\sigma|} \Phi(v_{\sigma(1)}, x_{\sigma(1)}; \ldots; v_{\sigma(k)},  x_{\sigma(k)})   
\rangle 
\nn
&&
 \qquad \qquad  \langle w',  e^{\zeta_2 L_W{(-1)}} \; Y_{W}(\overline{u}, -\zeta_2) \;   
 \Psi(v'_{\sigma(1)}, y_{\sigma(1)}; \ldots; v'_{\sigma(n)},  y_{\sigma(n)})
 \rangle 
\nn
&& 
+
\sum_{l \in \Z }  \epsilon^l 
 \sum_{u\in V_l } 
  \langle w', e^{\zeta_1 L_W{(-1)}} \; Y_{W}(u, -\zeta_1) \; 
 \Phi(v_{\sigma(1)}, x_{\sigma(1)}; \ldots; v_{\sigma(k)},  x_{\sigma(k)})   
\rangle 
\nn
&&
 \qquad   \langle w',  e^{\zeta_2 L_W{(-1)}} \; Y_{W}(\overline{u}, -\zeta_2) \;   
\sum_{\sigma\in J_{n; s}^{-1}}(-1)^{|\sigma|}  \Psi(v'_{\sigma(1)}, y_{\sigma(1)}; \ldots; v'_{\sigma(n)},  y_{\sigma(n)})
 \rangle=0,  
\end{eqnarray*}
since,
%%
%\[
%%
$J^{-1}
_{k+n; s}= J^{-1} 
_{k;s} \times J^{-1}
_{n;s}$,  
and 
%%%%%%%%%%%%%%%%%%%%%%%%%%%%%%%%%%%%%%%%%%%%%%%%%%%%%%%%%%%%%%%%%%%%%%%%%%%%%%%%%5
%%
  due to the fact that 
 $\F(v_{1}, x_{1}; \ldots; v_{k},  x_{k})$
 and 
$\F(v'_{1}, y_{1} $; $ \ldots $; $ v'_{n},  y_{n})$ 
satisfy \eqref{sigmaction}. 
\end{proof}
%%
%%%%%%%%%%%%%%%%%%%%%%%%%%%%%%%%%%%%%%%%%%%%%%%%%%%%%%%%%%%%%%%%%%%%%%%%%%%%%%%%%%%%%%%%
%%%%%%%%%%%%%%%%%%%%%%%%%%%%%%%%%%%%%%%%%%%%%%%%%%%%%%%%%%%%%%%%%%%%%%%%%%%%%%%%%%%%%%%%

%%%%%%%%%%%%%%%%%%%%%%%%%%%%%%%%%%%%%%%%%%%%%%%%%%%%%%%%%%%%%%%%%%%%%%%%%%%%%%%%%%%%%%
%%%%%%%%%%%%%%%%%%%%%%%%%%%%%%%%%%%%%%%%%%%%%%%%%%%%%%%%%%%%%%%%%%%%%%%%%%%%%%%%%%%%%%
\subsection{Proof of Proposition \ref{ccc}}
\begin{proof}
Recall that 
$\Phi(v_{1}, x_{1}; \ldots; v_{k},  x_{k})$
is composable with $m$ vertex operators, and 
 $\Psi(v'_{1}, y_{1}; \ldots; v'_{n},  y_{n})$ 
is composable with $m'$ vertex operators. 
%%
%%%%%%%%%%%%%%%%%%%%%%%%%%%%%%%%%%%%%%%%%%%%%%%%%%%%%%%%%%%%%%%%%%%%%%%%%%%%%%%%%%%%%%%%%%%%%%
%%%%%%%%%%%%%%%%%%%%%%%%%%%%%%%%%%%%%%%%%%%%%%%%%%%%%%%%%%%%%%%%%%%%%%%%%%%%%%%%%%%%%%%%%%%%%%
%%
%%
%%
For $\Phi(v_{1}, x_{1}; \ldots; v_{k},  x_{k})$ we have: 

1) Let $l_{1}, \dots, l_{k}\in \Z_+$ such that $l_{1}+\ldots +l_{k}= k+m$,  and 
$v_{1}, \dots, v_{k+m} \in V$, and arbitrary $w'\in W'$. Set  
 \begin{eqnarray}
\label{psii}
\Xi_{i}
&
=
&
E^{(l_{i})}_{V}(v_{k_1}, x_{k_1}- \zeta_{i};  
%%
%\otimes
 \ldots; 
v_{k_i}, x_{k_i}- \zeta_{i} 
 ; \one_{V}),    
\end{eqnarray}
where
\begin{eqnarray}
\label{ki}
 {k_1}={l_{1}+\ldots +l_{i-1}+1}, \quad  \ldots, \quad  {k_i}={l_{1}+\ldots +l_{i-1}+l_{i}}, 
\end{eqnarray} 
for $i=1, \dots, k$. 
%%%%
%%%%%%%%%%%%%%%%%%%%%%%%%%%%%%%%%%%%%%%%%%%%%%%%%%%%%%%%%%%%%%%%%%%%%%%%%%%%%%%%%%%%%%%%%%%%%
Then the series 
%%%%
%%%%%%%%%%%%%%%%%%%%%%%%%%%%%%%%%%%%%%%%%%%%%%%%%%%%%%%%%%%%%%%%%%%%%%%%%%%%%%%%%%%%%%%%%%%%%%
%%%%%%%%%%%%%%%%%%%%%%%%%%%%%%%%%%%%%%%%%%%%%%%%%%%%%%%%%%%%%%%%%%%%%%%%%%%%%%%%%%%%%%%%%%%%%%
%%
\begin{eqnarray}
\label{Inms}
\mathcal I^k_m(\Phi)=
\sum_{r_{1}, \dots, r_{k}\in \Z}\langle w', 
\Phi(P_{r_{1}}\Xi_{1}; \zeta_1; 
 \ldots; %\otimes
P_{r_k} \Xi_{k}, \zeta_{k}) 
\rangle,
\end{eqnarray} 
%% 
%%%%%%%%%%%%%%%%%%%%%%%%%%%%%%%%%%%%%%%%%%%%%%%%%%%%%%%%
%%
is absolutely convergent  when 
\begin{eqnarray}
\label{granizy1}
|x_{l_{1}+\ldots +l_{i-1}+p}-\zeta_{i}| 
+ |x_{l_{1}+\ldots +l_{j-1}+q}-\zeta_{i}|< |\zeta_{i}
-\zeta_{j}|, 
\end{eqnarray} 
for $i$, $j=1, \dots, k$, $i\ne j$ and for $p=1, 
\dots,  l_i$ and $q=1, \dots, l_j$. 
%%
%%%%%%%%%%%%%%%%%%%%%%%%%%%%%%%%%%%%%%%%%%%%%%%%%%%%%%%%%%%%%%%%%%%%%%%%%%%%%%%%%%%%%%%%%%%5
There exist positive integers $N^k_m(v_{i}, v_{j})$, 
depending only on $v_{i}$ and $v_{j}$ for $i, j=1, \dots, k$, $i\ne j$, such that 
the sum is analytically extended to a
rational function
in $(x_{1}, \dots, x_{k+m})$, 
 independent of $(\zeta_{1}, \dots, \zeta_{k})$,  
with the only possible poles at 
$x_{i}=x_{j}$, of order less than or equal to 
$N^k_m(v_{i}, v_{j})$, for $i$, $j=1, \dots, k$,  $i\ne j$.  
%%

%%%%%%%%%%%%%%%%%%%%%%%%%%%%%%%%%%%%%%%%%%%%%%%%%%%%%%%%%%%%%5
%%%%%%%%%%%%%%%%%%%%%%%%%%%%%%%%%%%%%%%%%%%%%%%%%%%%%%%%%%%%%%%%%%%%%%%%%%%%%%%%%%%%%%%%%%%%%
%%%%%%%%%%%%%%%%%%%%%%%%%%%%%%%%%%%%%%%%%%%%%%%%%%%%%%%%%%%%%%%%%%%%%%%%%%%%%%%%%%%%%%%%%%%%%
For $\Psi(v'_{1}, y_{1}; \ldots; v'_{n},  y_{n})$  we have: 

\medskip 
1') Let $l'_{1}, \dots, l'_{n}\in \Z_+$ such that $l'_{1}+\ldots +l'_{n}= n+m'$,  
$v'_{1}, \dots, v_{n+m'} \in V$ and arbitrary $w'\in W'$. 
Set  
 \begin{eqnarray}
\label{psii}
\Xi'_{i'} 
&
=
&
E^{(l'_{i'})}_{V}(v'_{k'_1}, y_{k'_1}- \zeta'_{i'};  
%%
%\otimes
 \ldots; 
v'_{k'_{i'}}, y_{k'_{i'}}- \zeta'_{i'} 
 ; \one_{V}),    
\end{eqnarray}
where
\begin{eqnarray}
\label{ki}
 {k'_{1}}={l'_{1}+\ldots +l'_{i'-1}+1}, \quad  \ldots, \quad  {k'_{i'}}={l'_{1}+\ldots +l'_{i'-1}+l'_{i'}},  
\end{eqnarray} 
for $i'=1, \dots, n$. 
%%
%%%%%%%%%%%%%%%%%%%%%%%%%%%%%%%%%%%%%%%%%%%%%%%%%%%%%%%%%%%%%%%%%%%%%%%%%%%%%%%%%%%%%%%%%%%%%%
  Then the series 
%%%%%%%%%%%%%%%%%%%%%%%%%%%%%%%%%%%%%%%%%%%%%%%%%%%%%%%%%%%%%%%%%%%%%%%%%%%%%%%%%%%%%%%%%%%%%%
%%
\begin{eqnarray}
\label{Jnms}
\mathcal I^{n}_{m'}(\Psi)=  
\sum_{r'_{1}, \dots, r'_{n}\in \Z}\langle w', 
\Psi(P_{r'_{1}}\Psi'_{1}; \zeta'_1; 
 \ldots; %\otimes
P_{r'_{n}} \Psi'_{n}, \zeta'_{n})  
\rangle,
\end{eqnarray} 
%%
%%%%%%%%%%%%%%%%%%%%%%%%%%%%%%%%%%%%%%%%%%%%%%%%%%%%%%%%
%%
is absolutely convergent  when 
\begin{eqnarray}
\label{granizy2}
|y_{l'_{1}+\ldots +l'_{i'-1}+p'}-\zeta'_{i'}| 
+ |y_{l'_{1}+\ldots +l'_{j'-1}+q'}-\zeta'_{i'}|< |\zeta'_{i'}
-\zeta'_{j'}|, 
\end{eqnarray} 
for $i'$, $j'=1, \dots, n$, $i'\ne j'$ and for $p'=1, 
\dots,  l'_i$ and $q'=1, \dots, l'_j$. 
There exist positive integers $N^{n}_{m'}(v'_{i'}, v'_{j'})$,  
depending only on $v'_{i'}$ and $v'_{j'}$ for $i$, $j=1, \dots, n$, $i'\ne j'$, such that 
the sum is analytically extended to a
rational function
in $(y_{1}, \dots, y_{n+m'})$,   
 independent of $(\zeta'_{1}, \dots, \zeta'_{n})$,   
with the only possible poles at 
$y_{i'}=y_{j'}$, of order less than or equal to 
$N^{n}_{m'}(v'_{i'}, v'_{j'})$, for $i'$, $j'=1, \dots, n$,  $i'\ne j'$.  
%%

%%%%%%%%%%%%%%%%%%%%%%%%%%%%%%%%%%%%%%%%%%%%%%%%%%%%%%%%%%%%%%%%%%%%%%%%%%%%%%%%%%%%%
%%%%%%%%%%%%%%%%%%%%%%%%%%%%%%%%%%%%%%%%%%%%%%%%%%%%%%%%%%%%%%%%%%%%%%%%%%%%%%%%%%%%%%%%%%%%
Now let us consider the first condition of Definition \ref{composabilitydef} of composability for the product 
\eqref{Z2n_pt_epsss} of 
$\Phi(v_{1}, x_{1}; \ldots; v_{k},  x_{k})$ 
 and 
 $\Psi(v'_{1}, y_{1}; \ldots; v'_{n},  y_{n})$ with a number of vertex operators. 
Then we obtain for $\Theta\left(v_{1}, x_{1}; \ldots; v_k, x_k; v'_1, y_1; 
\ldots; v'_{n}, y_{n}; \epsilon \right)$ the following. 
%%
%%
%%%%%%%%%%%%%%%%%%%%%%%%%%%%%%%%%%%%%%%%%%%%%%%%%%%%%%%%%%%%%%%%%%%%%%%%%%%%%%%%%%%%%%%%%%%%%%
%%%%%%%%%%%%%%%%%%%%%%%%%%%%%%%%%%%%%%%%%%%%%%%%%%%%%%%%%%%%%%%%%%%%%%%%%%%%%%%%%%%%%%%%%%%%%%
%1'') 
 We redefine the notations for the 
%Let us form the 
%%
set
\begin{eqnarray*}
  && (v''_{1}, \ldots, v''_k; v''_{k+1}, \ldots, v''_{k+m}; v''_{k+m+1}, \dots, v''_{k+n+m+m'};
 v_{n+1}, \ldots, v'_{n+m'})
\nn
&&
 \qquad \qquad 
=(v_{1}, \ldots, v_k; v_{k+1}, \ldots, v_{k+m}; v'_1, \dots, v'_{n};
 v'_{n+1}, \ldots, v'_{n+m'}), %\in V, 
\nn
&&
(z_{1}, \ldots, z_k; z_{k+1},  \dots, z_{k+n-r}) = 
  (x_1, \ldots, x_k; y_1, \ldots, y_{n}), 
\end{eqnarray*}
 of vertex algebra $V$ elements.  
Introduce $l''_{1}, \dots, l''_{k+n} \in \Z_+$,     
 such that $l''_{1}+\ldots +l''_{k+n}= k+n+m+m'$.  
%%
%
%%%%%%%%%%%%%%%%%%%%%%%%%%%%%%%%%%%%%%%%%%%%%%%%%%%%%%%%%%%%%%%%%%%%%%%%
Define  
 \begin{eqnarray}
\label{psiinew}
\Xi''_{i}
&
=
&
E^{(l''_{i''})}_{V}(v''_{k''_1}, z_{k''_1}- \zeta''_{i''};  
%%
%\otimes
 \ldots; 
v''_{k''_{i''}}, z_{k''_{i''}}- \zeta''_{i''}   
 ; \one_{V}),    
\end{eqnarray}
where
\begin{eqnarray}
\label{ki}
 {k''_1}={l''_{1}+\ldots +l''_{i''-1}+1}, \quad  \ldots, \quad  {k''_{i''}}={l''_{1}+\ldots +l''_{i''-1}+l''_{i''}},   
\end{eqnarray} 
for $i''=1, \dots, k+n$, 
and we take 
\[
(\zeta''_1, \ldots, \zeta''_{k+n})= (\zeta_1, \ldots, \zeta_{k}; \zeta'_1, \ldots, \zeta'_{n}). 
\]  
%%%
Then we consider 
%%
%%%%%%%%%%%%%%%%%%%%%%%%%%%%%%%%%%%%%%%%%%%%%%%%%%%%%%%%%%%%%%%%%%%%%%%%%%%%%%%%%%%%%%
\begin{eqnarray}
\label{Inmdvadva}
 && \mathcal I^{k+n}_{m+m'}(\Theta)=  
\sum_{r''_{1}, \dots, r''_{k+n}\in \Z}
 \langle w',  
\Theta(P_{r''_{1}}\Psi''_{1}; \zeta''_1; 
 \ldots; 
P_{r''_{k+n}} \Psi''_{k+n}, \zeta''_{k+n})  
\rangle,
\end{eqnarray} 
and prove it is absolutely convergent with some conditions. 
%%

%%%%%%%%%%%%%%%%%%%%%%%%%%%%%%%%%%%%%%%%%%%%%%%%%%%%%%%%%%%%%%%%%%%%%%%%%%%%%%%%%%%%%%
%%
%%
 The 
condition   
\begin{eqnarray}
\label{granizy1000000}
&&
|z_{l''_{1}+\ldots +l''_{i-1}+p''}-\zeta''_{i}| 
+ |z_{l''_{1}+\ldots +l''_{j-1}+q''}-\zeta''_{i}|< |\zeta''_{i} -\zeta''_{j}|,  
\end{eqnarray} 
of absolute convergence for \eqref{Inmdvadva} for $i''$, $j''=1, \dots, k+n$, $i\ne j$ and for $p''=1,  
\dots,  l''_i$ and $q''=1, \dots, l''_j$, follows from the conditions \eqref{granizy1} and \eqref{granizy2}.  
%%
%%%%%%%%%%%%%%%%%%%%%%%%%%%%%%%%%%%%%%%%%%%%%%%%%%%%%%%%%%%%%%%%%%%%%%%%%%%%%%%5
%% 
%%%%%%%%%%%%%%%%%%%%%%%%%%%%%%%%%%%%%%%%%%%%%%%%%%%%%%%%%%%%%%%%%%%%%%%%%%%%%%%%%%%%%%%%%%%%%%%%%%%%%%%
 The action of $e^{\zeta L_W{(-1)} } \;  Y_W(.,.)$, $a=1$, $2$, in 
\[
\langle w',  e^{\zeta_1 L_W{(-1)} } \; Y_W({u}, -\zeta)\sum_{r_{1}, \dots, r_{k}\in \Z} 
\Phi(P_{r_{1}}\Xi_{1}; \zeta_1; 
 \ldots; 
P_{r_{k}} \Xi_{k}, \zeta_{k})   \rangle, 
\]
%%
%and 
%%
\[
\langle w',  e^{\zeta_2 L_W{(-1)} } \; Y_W(\overline{u}, -\widetilde{\zeta}) \sum_{r'_{1}, \dots, r'_{n}\in \Z} 
\Psi(P_{r'_{1}}\Xi'_{1}; \zeta_1; 
 \ldots; 
P_{r'_{k}} \Xi'_{n}, \zeta'_{n})  \rangle,  
\]
 does not affect the absolute convergency of \eqref{Inms} and \eqref{Jnms}. %Using Lemma \ref{persvo},
 We obtain 
%% 
%%%%%%%%%%%%%%%%%%%%%%%%%%%%%%%%%%%%%%%%%%%%%%%%%%%%%%%%%%%%%%%%%%%%%%%%%%%%%%%%%%%%%%%%%%%%%%
%%%%%%%%%%%%%%%%%%%%%%%%%%%%%%%%%%%%%%%%%%%%%%%%%%%%%%%%%%%%%%%%%%%%%%%%%%%%%%%%%%%%%%%%%%%%%%
%%
%% 
\begin{eqnarray*}
 && \left|\mathcal I^{k+n}_{m+m'}(\Theta)\right|= 
%% 
%%%%%%%%%%%%%%%%%%%%%%%%%%%%%%%%%%%%%%%%%%%%%%%%%%%%%%%%%%%
\nn
&&
=\left|\sum_{r''_{1}, \dots, r''_{k+n}\in \Z}
 \langle w',  
\Theta(P_{r''_{1}}\Xi''_{1}; \zeta''_1; 
 \ldots; 
P_{r''_{k+n}} \Xi''_{k+n}, \zeta''_{k+n})  
\rangle\right| 
\nn
&&
%%%%%%%%%%%%%%%%%%%%%%%%%%%%%%%%%%%%%%%%%%%%%%%%%%%%%%%%%%%
=\left| 
\sum_{l \in \Z }  \epsilon^l 
 \sum_{u\in V_l } 
\langle w',  Y^W_{VW}(\sum_{r_{1}, \dots, r_{k}\in \Z} 
\Phi(P_{r_{1}}\Xi_{1}; \zeta_1; 
 \ldots; 
P_{r_{k}} \Xi_{k}, \zeta_{k}), \zeta) u  \rangle \right.
%%
%%%%%%%%%%%%%%%%%%%%%%%%%%%%%%%%%%%%%%%%%%%%%%%%%%%%%%%
\nn
&&
\left. 
\qquad \qquad \langle w',  Y^W_{VW}(\sum_{r'_{1}, \dots, r'_{n}\in \Z} 
\Psi(P_{r'_{1}}\Xi'_{1}; \zeta'_1; 
 \ldots; 
P_{r'_{n}} \Xi'_{n}, \zeta'_{n}), \widetilde{\zeta}) \overline{u}
\rangle \right|
\nn
&&
 \qquad \qquad \qquad \le \left|\mathcal I^{k}_{m}(\Phi)\right| \; \left|\mathcal I^{n}_{m'}(\Psi)\right|.  
\end{eqnarray*} 
%%
%%%%%%%%%%%%%%%%%%%%%%%%%%%%%%%%%%%%%%%%%%%%%%%%%%%%%%%%%%%%%%%%%%%%%%%%%%%%%%%%%%%%%%%%%%%%%%%%%%%%%%%%
%%%%%%%%%%%%%%%%%%%%%%%%%%%%%%%%%%%%%%%%%%%%%%%%%%%%%%%%%%%%%%%%%%%%%%%%%%%%%%%%%%%%%%%%%%%%%%%%%%%%%%%%
%
%%
Thus, we infer that \eqref{Inmdvadva}
is absolutely convergent. 
Recall that 
the maximal orders of possible poles of \eqref{Inmdvadva} are $N^{k}_{m}(v_{i}, v_{j})$, $N^{n}_{m'}(v'_{i'}, v'_{j'})$
 at $x_{i}=x_{j}$, $y_{i'}=y_{j'}$. 
From the last expression we infer that %Since 
there exist 
  positive integers $N^{k+n}_{m+m'}(v''_{i''}, v''_{j''})$ 
%% 
%%%%%%%%%%%%%%%%%%%%%%%%%%%%%%%%%%%%%%%%%%%%%%%%%%%%%%%%%%%%%%%%%%%%%%%%%%%%%%%%%%%%%%%%%%%%%%%%
%%
 for $i$, $j=1, \dots, k$, $i\ne j$,   $i'$, $j'=1, \dots, n$, $i'\ne j$, %such that   
%%
%%%%%
%  
%%
depending only on $v''_{i''}$ and $v''_{j''}$ for $i''$,  $j''=1, \dots, k+n$, $i''\ne j''$
%5
such that 
 the series \eqref{Inmdvadva} 
%%%%
can be analytically extended to a
rational function
in $(x_{1}, \dots, x_k; y_1, \ldots, y_n)$,    
 independent of $(\zeta''_{1}, \dots, \zeta''_{k+n})$,    
with extra 
 possible poles at  
 and $x_{i}=y_{j}$,  of order less than or equal to 
$N^{k+n}_{m+m'}(v''_{i''}, v''_{j''})$, for $i''$, $j''=1, \dots, n$,  $i''\ne j''$.  
Let us proceed with the second condition of composability. 
For $\Phi(v_1, x_1; \ldots; v_{k}, x_k)  \in  C^{k}_{m}(V, \W, \F)$, and 
$(v_1, \ldots, 
v_{k+m}) \in V$, $(x_1, \ldots, %x_{k}; x_{k+1}, \ldots, 
x_{k+m})\in \C$,     
 we have

\medskip 
2)
For arbitrary  $w'\in W'$, the series 
\begin{eqnarray}
\label{Jnm2}
\mathcal J^k_m(\Phi)=  
\sum_{q\in \C}\langle w', 
E^{(m)}_{W} \Big(v_{1}, x_1;  \ldots; %\otimes \ldots\otimes 
v_{m}, x_m;  
%\nn
%%&&\quad\quad\quad 
%%
P_{q}( \Phi(v_{m+1}, x_{m+1}; \ldots; %  \otimes \ldots\otimes 
v_{m+k}, x_{m+k}
\Big)
\rangle, 
\nn
&&
\end{eqnarray}
is absolutely convergent when 
\begin{eqnarray}
\label{granizy2}
x_{i}\ne x_{j}, \quad i\ne j, \quad 
\nn
|x_{i}|>|x_{k'}|>0, 
\end{eqnarray}
 for $i=1, \dots, m$, and $k'=m+1, \dots, k+m$, and the sum can be analytically extended to a
rational function 
in $(x_{1}, \dots,  %x_{k}; x_{k+1}, \ldots, 
x_{k+m})$ with the only possible poles at 
$x_{i}=x_{j}$, of orders less than or equal to 
$N^k_m(v_{i}, v_{j})$, for $i, j=1, \dots, k$, $i\ne j$.  
%%
%
%%
 
%%%%%%%%%%%%%%%%%%%%%%%%%%%%%%%%%%%%%%%%%%%%%%%%%%%%%%%%%%%%%%%%%%%%%%
%%%%%%%%%%%%%%%%%%%%%%%%%%%%%%%%%%%%%%%%%%%%%%%%%%%%%%%%%%%%%%%%%%%%%%
%%
2')  For 
 $\Psi(v'_{1}, y_{1}; \ldots; v'_{n}, y_{n})  \in   C_{m'}^{n}(V, \W, \F)$,
$(v'_1, \ldots, 
v'_{n+m'})\in V$,  and 
$(y_1, \ldots, 
y_{n+m'})\in \C$,  
 the series 
 %%
%%%%%%%%%%%%%%%%%%%%%%%%%%%%%%%%%%%%%%%%%%%%%%%%%%%%%%%%%%%%%%%%%%%%%%%%%%%%%%%
%%%%%%%%%%%%%%%%%%%%%%%%%%%%%%%%%%%%%%%%%%%%%%%%%%%%%%%%%%%%%%%%%%%%%%%%%%%%%%%
%% 
%%
\begin{eqnarray}
\label{Jnm}
&& \mathcal J^{n}_{m'}(\Psi)=  
\sum_{q\in \C}\langle w', E^{(m')}_{W} \Big(v'_{1}, y_1; \ldots; 
 v'_{m'}, y_{m'};  
\nn
&&\quad\quad\quad 
P_{q}( \Psi(v'_{m'+1}, y_{m'+1}; \ldots; v'_{m'+n}, y_{m'+n}) 
)\Big)\rangle, 
%%
 %\nn
 %&&
%%
\end{eqnarray}
is absolutely convergent when 
\begin{eqnarray}
\label{granizy2}
y_{i'}\ne y_{j'}, \quad i'\ne j', \quad 
\nn
|y_{i'}|>|y_{k''}|>0, 
\end{eqnarray}
 for $i'=1, \dots, m'$, and $k''=m'+1, \dots, n +m'$, and the sum can be analytically extended to a
rational function 
in $(y_{1}, \ldots, 
y_{n+m'})$ with the only possible poles at 
$y_{i'}=y_{j'}$, of orders less than or equal to 
$N^{n}_{m'}(v'_{i'}, v'_{j'})$, for $i'$,  $j'=1, \dots, n$, $i'\ne j'$.  
%%
%%\end{enumerate}
%%

%%
2'') 
Thus, for the product \eqref{Z2n_pt_epsss} we obtain 
  $(v''_{1}, \dots, v''_{k+n +m+m'})\in V$,  
and 
%%redfining 
%%
$(z_1, \ldots $ , $ z_{k+n+m+m'} )\in \C$, 
we find 
positive integers  
%%
%\[
%%
$N^{k+n}_{m+m'}(v'_{i}, v'_{j})$,  
%%
%\]
%%
 depending only on $v'_{i}$ and 
$v''_{j}$, for $i''$, $j''=1, \dots, k+n$, $i''\ne j''$, such that for arbitrary $w'\in W'$. 
%% 
%%
%%%%%%%%%%%%%%%%%%%%%%%%%%%%%%%%%%%%%%%%%%%%%%%%%%%%%%%%%%%%%%%%%%%%%%%%%%%%%%%%
First we note 
%%
%%%%%%%%%%%%%%%%%%%%%%%%%%%%%%%%%%%%%%%%%%%%%%%%%%%%%%%%%%%%%%%%%%%%%%%%%55
\begin{lemma}
\label{obvlem}
%%
%%%%%%%%%%%%%%%%%%%%%%%%%%%%%%%%%%%%%%%%%%%%%%%%%%%%%%%%%%%%%%%%%%%%%%%%%%%%%%
\begin{eqnarray*}
%%
%\label{Jnm00 }
%%
&&
\sum_{q\in \C}\langle w', E^{(m+m')}_{W} \Big(
v''_{1}, z_{1}; \ldots; 
v''_{m+m'}, z_{m+m'};   
\nn
&&
 \qquad \qquad  \qquad  
P_{q}\Big( \Theta(v''_{m+m'+1}, z_{m+m'+1}; \ldots; v''_{m+m'+k+n}, z_{m+m'+k+n}  
\Big) \Big) \rangle
%%
%
%&&
\nn
&&
%%
%\qquad \qquad 
%%
%\qquad 
%%
=
\sum_{l \in \Z }  \epsilon^l 
 \sum_{u\in V_l } 
\langle w',  E^{(m)}_{W} \Big(
v_{k+1}, x_{k+1}; \ldots;
v_{k+m}, x_{k+m}; 
\nn
&& 
\qquad \qquad  \qquad \qquad  \qquad \qquad  
P_{q} \Big(  
  Y^{W}_{WV}\left(  
\Phi (v_{1}, x_{1};  \ldots; v_{k}, x_{k}), \zeta_1\right)\; u \Big) \Big)\rangle  
\nn
& &
\qquad   \qquad 
 \langle w',  E^{(m')}_{W} \Big(
v'_{n+1}, y_{n+1}; \ldots; 
v'_{n+m'}, y_{n+m'};  
\nn
&&
  \qquad \qquad \qquad \qquad  \qquad \qquad  
P_{q}\Big(    Y^{W}_{WV}\left( 
\Psi
(v'_{1}, y_{1}; \ldots; v'_{n}, y_{n}) , \zeta_{2}\right) \; \overline{u} \Big) \Big) \rangle.    
%%
%%
%%
%%%%%%%%%%%%%%%%%%%%%%%%%%%%%%%%%%%%%%%%%%%%%%%%%%%%%%%%%%%%%%%%%%%%%%%%%%%
%%
\end{eqnarray*}
\end{lemma}
%%%%
%%%%%%%%%%%%%%%%%%%%%%%%%%%%%%%%%%%%%%%%%%%%%%%%%%%%%%%%%%%%%%%%%%%%%%%%%%%%%%%%%%%%%%%%
%%%%%%%%%%%%%%%%%%%%%%%%%%%%%%%%%%%%%%%%%%%%%%%%%%%%%%%%%%%%%%%%%%%%%%%%%%%%%%%%%%%%%%%%
\begin{proof}
Consider
%%
%%%%%%%%%%%%%%%%%%%%%%%%%%%%%%%%%%%%%%%%%%%%%%%%%%%%%%%%%%%%%%%%%%%%%%%%%%%%%%%%%%%%%5=
\begin{eqnarray*}
&& 
\sum_{l \in \Z }  \epsilon^l 
 \sum_{u\in V_l } 
\langle w',  E^{(m+m')}_{W} \Big(
v''_{1}, z_1; \ldots;
v''_{m+m'}, z_{m+m'}; 
%%v_{1}\otimes \ldots\otimes v_{m};   
\nn
&& 
\qquad \qquad  P_{q} \Big(  
  Y^{W}_{WV}\left(  
\Phi (v''_{m+m'+1}, z_{m+m'+1};  \ldots; v''_{m+m' +k}, z_{m+m'+k}), \zeta_1\right)\; u \Big) \Big)\rangle  
\nn
& &
\qquad   \qquad 
 \langle w',  E^{(m+m')}_{W} \Big(
v''_{1}, z_1; \ldots;
v''_{m+m'}, z_{m+m'}; 
\nn
&&
  \qquad \qquad P_{q}\Big(    Y^{W}_{WV}\left( 
\Psi
(v''_{m+m'+k+1}, z_{m+m'+k+1}; \ldots; \right. 
\nn
&&
      \qquad \qquad  \qquad \qquad    \qquad \qquad \qquad \qquad \left. 
 v''_{m+m'+k+n}, z_{m+m'+k+n}) , \zeta_{2}\right) \; \overline{u} \Big) \Big) \rangle    
%%
%%%%%%%%%%%%%%%%%%%%%%%%%%%%%%%%%%%%%%%%%%%%%%%%%%%%%%%%%%%%%%%%%%%%%%%%%%%%%%%%%%%%%%%%%%%%%
%%%%%%%%%%%%%%%%%%%%%%%%%%%%%%%%%%%%%%%%%%%%%%%%%%%%%%%%%%%%%%%%%%%%%%%%%%%%%%%%%%%%%%%%%%%%%
%%
\end{eqnarray*}
\begin{eqnarray*}
&&
= \sum_{q\in \C} 
\sum_{l \in \Z }  \epsilon^l 
 \sum_{u\in V_l } 
  \langle w', E^{(m+m')}_{W} \Big(
v''_{1}, z_1; \ldots; %\otimes \ldots\otimes 
v''_{m+m'}, z_{m+m'};  
\nn
&&
\qquad   
  P_{q}\Big( e^{\zeta_1 L_W{(-1)} }\; Y_{W} %_{WV} 
\left( u, -\zeta_1 ) 
\; \Phi (v''_{m+m'+1}, z_{m+m'+1};  \ldots; v''_{m+m'+ k }, z_{m+m'+k}) %, -\zeta_1) u%, \zeta_1
\right) %\; u
 \Big) \rangle 
\nn
& & 
\quad   %\qquad 
\langle w',  E^{(m+m')}_{W} \Big(
%%v'_{1}\otimes \ldots\otimes v'_{m'};
v''_{1}, z_1; \ldots;
v''_{m+m'}, z_{m+m'}; 
\nn
&&
  P_{q}\Big( e^{\zeta_2 L_W{(-1)}}\; Y_{W} %_{WV}  
\left(  \overline{u}, -\zeta_2  )\;  
\Psi
(v''_{m+m'+k+1}, z_{m+m'+k+1}; \ldots; \right. 
\nn
&&
   \qquad \qquad     \qquad \qquad  \qquad \qquad    \qquad \qquad \qquad \qquad \left.
 v''_{m+m'+k+n}, z_{m+m'+k+n})  %, -\zeta_{2}) \overline{u} 
\right) %\; \overline{u} 
\Big)\rangle.  %\right|   
%\nn
%%%%%%%%%%%%%%%%%%%%%%%%%%%%%%%%%%%%%%%%%%%%%%%%%%%%%%%%%%%%%%%%%%%%%%%%%%
%%%%%%%%%%%%%%%%%%%%%%%%%%%%%%%%%%%%%%%%%%%%%%%%%%%%%%%%%%%%%%%%%%%%%%%%%%
%%
\end{eqnarray*}
%%
%
%%
%% tupot
%%
The action of exponentials $e^{\zeta_a L_W{(-1)} }$, $a=1$, $2$, of the differential operator $L_W{(-1)}$,   
and $W$-module vertex operators $Y_{W} \left( u, -\zeta_1 \right)$, % and 
$Y_{W} \left( u, -\zeta_2 \right)$ 
 shifts the grading index $q$ of $W_q$-subspaces by $\alpha \in \C$ which can be later rescaled 
to $q$.  
 Thus, we can rewrite the last expression as 
%%
%%%%%%%%%%%%%%%%%%%%%%%%%%%%%%%%%%%%%%%%%%%%%%%%%%%%%%%%%%%%%%%%%%%%%%%%%%%%%%%%%%%%%5
\begin{eqnarray*}
&&
= \sum_{q\in \C}  %\sum_{u\in V } 
\sum_{l \in \Z }  \epsilon^l 
 \sum_{u\in V_l } 
 \langle w', E^{(m+m')}_{W} \Big(
v''_{1}, z_1; \ldots; 
v''_{m+m'}, z_{m+m'};  
\nn
&&
\qquad   
  e^{\zeta_1 L{_W(-1)} }\; Y_{W} 
\left
( u, -\zeta_1 \big) 
\;P_{q+\alpha}\Big(   \Phi(v''_{m+m'+1}, z_{m+m'+1};  \ldots; v''_{m+m'+ k }, z_{m+m'+k}) 
\right) %\; u
 \Big) \rangle
\nn
& & 
\langle w',  E^{(m+m')}_{W} \Big(
%%v'_{1}\otimes \ldots\otimes v'_{m'};
v''_{1}, z_1; \ldots;
v''_{m+m'}, z_{m+m'}; 
\nn
&&
\quad   %\qquad 
  e^{\zeta_2 L_W{(-1)} }\; Y_{W} 
\left(  \overline{u}, -\zeta_2  \Big)\;   \right. 
%%
%\nn
%&&
%% 
%\qquad \qquad
%%
\nn
&&
\left. 
 \qquad P_{q+\alpha}\Big(  \Psi
(v''_{m+m'+k+1}, z_{m+m'+k+1}; \ldots; v''_{m+m'+k+n}, z_{m+m'+k+n})  
\right)
\Big)\rangle 
%\nn
\end{eqnarray*}
%%%%%%%%%%%%%%%%%%%%%%%%%%%%%%%%%%%%%%%%%%%%%%%%%%%%%%%%%%%%%%%%%%%%%%%%%%%%%%%%%%5
%%%%%%%%%%%%%%%%%%%%%%%%%%%%%%%%%%%%%%%%%%%%%%%%%%%%%%%%%%%%%%%%%%%%%%%%%%%%%%%%%%%
\begin{eqnarray*}
&&
= \sum_{q\in \C}  
 \sum_{l \in \Z }  \epsilon^l 
 \sum_{u\in V_l } 
\langle w', E^{(m+m')}_{W} \Big(
v''_{1}, z_1; \ldots; 
v''_{m+m'}, z_{m+m'};  
\nn
&&
Y^{W}_{WV} 
\left(
 P_{q+\alpha}\Big(   \Phi (v''_{m+m'+1}, z_{m+m'+1};  \ldots; v''_{m+m'+ k }, z_{m+m'+k}) %, -\zeta_1) , \zeta_1 
\right)
\Big), \zeta_1 \Big)\; u
  \rangle 
\nn
&&
\langle w',  E^{(m+m')}_{W} \Big(
%%v'_{1}\otimes \ldots\otimes v'_{m'};
v''_{1}, z_1; \ldots;
%%
%%\otimes \ldots\otimes 
%%
v''_{m+m'}, z_{m+m'}; 
%%
%\nn
%&&
%%
\nn
& & 
Y^{W}_{WV}   
\left(  
 P_{q+\alpha}\Big(  \Psi
(v''_{m+m'+k+1}, z_{m+m'+k+1}; \ldots; v''_{m+m'+k+n}, z_{m+m'+k+n})  , -\zeta_{2}) \;\overline{u} 
\right) 
\Big)\rangle    
%\nn
%%
\end{eqnarray*}
%%%%%%%%%%%%%%%%%%%%%%%%%%%%%%%%%%%%%%%%%%%%%%%%%%%%%%%%%%%%%%%%%%%%%%%%%%%%%%%%%%5
%%
\begin{eqnarray*}
%%%%%%%%%%%%%%%%%%%%%%%%%%%%%%%%%%%%%%%%%%%%%%%%%%%%%%%%%%%%%%%%%%%%%%%%%%%%%%%%
&&
= \sum_{q\in \C} 
 \sum_{\widetilde{w} \in W }   
 \langle w', E^{(m+m')}_{W} \Big(
v''_{1}, z_1; \ldots; 
v''_{m+m'}, z_{m+m'};  \widetilde{w} \Big) \rangle  
\nn
&&
\qquad   
\sum_{l \in \Z }  \epsilon^l 
 \sum_{u\in V_l } 
\langle w',  Y^{W}_{WV}  
\left( 
\;P_{q+\alpha}\Big(   \Phi (v''_{m+m'+1}, z_{m+m'+1};  \ldots; v''_{m+m'+ k }, z_{m+m'+k}) , -\zeta_1)\; u 
\right) 
 \Big)  \rangle 
\nn
& & 
\quad   %\qquad 
\langle \widetilde{w}', E^{(m+m')}_{W} \Big(
v''_{1}, z_1; \ldots; 
v''_{m+m'}, z_{m+m'};  \widetilde{w} \Big) \rangle  
\nn
&&
\langle w', Y^{W}_{WV}   
\left(  
P_{q+\alpha}\Big(  \Psi
(v''_{m+m'+k+1}, z_{m+m'+k+1}; \ldots; v''_{m+m'+k+n}, z_{m+m'+k+n})  , -\zeta_{2}) \;\overline{u}  
\right) 
\Big)\rangle   
%\nn
%%
%%%%%%%%%%%%%%%%%%%%%%%%%%%%%%%%%%%%%%%%%%%%%%%%%%%%%%%%%%%%%%%%%%%%%%%%%%%
%%
\end{eqnarray*}
%%
%%%%%%%%%%%%%%%%%%%%%%%%%%%%%%%%%%%%%%%%%%%%%%%%%%%%%%%%%%%%%%%%%%%%%%%%%%%%%%%%%555
\begin{eqnarray*}
%%
%%%%%%%%%%%%%%%%%%%%%%%%%%%%%%%%%%%%%%%%%%%%%%%%%%%%%%%%%%%%%%%%%%%%%%%%%%%%%%%%
%%%%%%%%%%%%%%%%%%%%%%%%%%%%%%%%%%%%%%%%%%%%%%%%%%%%%%%%%%%%%%%%%%%%%%%%%%%%%%%%
&&
= \sum_{q\in \C}  
 \langle w', E^{(m+m')}_{W} \Big(
v''_{1}, z_1; \ldots; %\otimes \ldots\otimes 
v''_{m+m'}, z_{m+m'};  %w  \rangle  
\nn
&&
\qquad   
 P_{q+\alpha}\Big(   \Theta (v''_{m+m'+1}, z_{m+m'+1};  \ldots; v''_{m+m'+ k }, z_{m+m'+k}; % ) , -\zeta_1) u 
\nn
& & 
\qquad   \qquad 
v''_{m+m'+k+1}, z_{m+m'+k+1}; \ldots; v''_{m+m'+k+n}, z_{m+m'+k+n}) \Big)  
% 
%%
%\Big)
\rangle. 
%\nn
%%
%%%%%%%%%%%%%%%%%%%%%%%%%%%%%%%%%%%%%%%%%%%%%%%%%%%%%%%%%%%%%%%%%%%%%%%%%%%
%%
\end{eqnarray*}
%%
%%%%%%%%%%%%%%%%%%%%%%%%%%%%%%%%%%%%%%%%%%%%%%%%%%%%%%%%%%%%%%%%%%%%%%%%%%%%%%%%%%%%%%%%%%%%%%%%%
%%%%%%%%%%%%%%%%%%%%%%%%%%%%%%%%%%%%%%%%%%%%%%%%%%%%%%%%%%%%%%%%%%%%%%%%%%%%%%%%%%%%%%%%%%%%%%%%%
Now note that, according to Proposition \ref{pupa}, as an element of $\W_{z_1, \ldots,  z_{k+n+m+m'}}$
%%
%%%%
\begin{eqnarray}
\label{svoloch}
&& \langle w', E^{(m+m')}_{W} \Big(
v''_{1}, z_1; \ldots; %\otimes \ldots\otimes 
v''_{m+m'}, z_{m+m'};  %w  \rangle  
\nn
&&
\qquad   
 P_{q+\alpha}\Big(   \Theta (v''_{m+m'+1}, z_{m+m'+1};  \ldots; v''_{m+m'+ k }, z_{m+m'+k}; % ) , -\zeta_1) u 
\nn
& & 
\qquad   \qquad 
v''_{m+m'+k+1}, z_{m+m'+k+1}; \ldots; v''_{m+m'+k+n}, z_{m+m'+k+n}) \Big) 
\rangle,  
\end{eqnarray}
is invariant with respect to the action of $\sigma \in S_{k+n+m+m'}$. 
 Thus we are able to use this invariance to show that \eqref{svoloch} is reduced to 
%%
%%%%%%%%%%%%%%%%%%%%%%%%%%%%%%%%%%%%%%%%%%%%%%%%%%%%%%%%%%%%%%%%%%%%%%%%%%%%%%%%%%%%%%%%%%
%%%%%%%%%%%%%%%%%%%%%%%%%%%%%%%%%%%%%%%%%%%%%%%%%%%%%%%%%%%%%%%%%%%%%%%%%%%%%%%%%%%%%%%%%%
%%%%
\begin{eqnarray*}
&& \langle w', E^{(m+m')}_{W} \Big(
v''_{k+1}, z_{k+1}; \ldots; v''_{k+1+m}, z_{k+1+m};  %\otimes \ldots\otimes 
v''_{n+1}, z_{n+1}; \ldots; v''_{n+1+m'}, z_{n+1+m'};   %w  \rangle  
\nn
&&
\qquad   
 P_{q+\alpha}\Big(   \Theta (v''_{1}, z_{1};  \ldots; v''_{k}, z_{k}; 
%
%%
%\qquad   \qquad 
%%
%  
%% 
v''_{k+1}, z_{k+1}; \ldots; v''_{k+n}, z_{k+n}) \Big) \Big)  
\rangle  
%%
%%%%%%%%%%%%%%%%%%%%%%%%%%%%%%%%%%%%%%%%%%%%%%%%%%%%%%%%%%%%%%%%%%%%%%%%%%%%%%%%%%%%%%%%%
%%%%%%%%%%%%%%%%%%%%%%%%%%%%%%%%%%%%%%%%%%%%%%%%%%%%%%%%%%%%%%%%%%%%%%%%%%%%%%%%%%%%%%%%%
\nn
&&
=\langle w', E^{(m+m')}_{W} \Big(
v_{k+1}, x_{k+1};  \ldots; v_{k+1+m}, x_{k+1+m}; 
v'_{n+1}, y_{n+1};  \ldots; v'_{n+1+m'}, y_{n+1+m'};  
\nn
&&
\qquad   
 P_{q+\alpha}\Big(   \Theta (v_{1}, x_{1};  \ldots; v_{k}, x_{k}; 
v'_{1}, y_{1}; \ldots; v'_{n}, y_{n}) \Big) 
\rangle. 
\end{eqnarray*}
%%
%%%%%%%%%%%%%%%%%%%%%%%%%%%%%%%%%%%%%%%%%%%%%%%%%%%%%%%%%%%%%%%%%%%%%%%%%%%%%%%%%%%%%%
%%%%%%%%%%%%%%%%%%%%%%%%%%%%%%%%%%%%%%%%%%%%%%%%%%%%%%%%%%%%%%%%%%%%%%%%%%%%%%%%%%%%%
Similarly, since 
\begin{eqnarray*}
&& 
\langle w',  E^{(m)}_{W} \Big(
v''_{1}, z_1; \ldots;
v''_{m+m'}, z_{m+m'}; 
\nn
&& 
\qquad \qquad  P_{q} \Big(  
  Y^{W}_{WV}\left(  
\F (v''_{m+m'+1}, z_{m+m'+1};  \ldots; v''_{m+m' +k}, z_{m+m'+k}), \zeta_1\right)\; u \Big) \Big)\rangle,  
\end{eqnarray*} 
%%
%and 
%%%%%%%%%%%%%%%%%%%%%%%%%%%%%%%%%%%%%%%%%%%%%%%%%%%%%%%%%%%%%%%%%%%%%%%%%%%%%%%%%
%%
\begin{eqnarray*}
& &
%%
 %\qquad   \qquad 
%%
%%
 \langle w',  E^{(m')}_{W} \Big(
v''_{1}, z_1; \ldots;
v''_{m+m'}, z_{m+m'}; 
\nn
&&
  \qquad 
P_{q}\Big(    Y^{W}_{WV}\left( 
\F
(v''_{m+m'+k+1}, z_{m+m'+k+1}; \ldots; v''_{m+m'+k+n}, z_{m+m'+k+n}) , \zeta_{2}\right) \; \overline{u} \Big) \Big) \rangle.    
\end{eqnarray*}
%%
%%%%%%%%%%%%%%%%%%%%%%%%%%%%%%%%%%%%%%%%%%%%%%%%%%%%%%%%%%%%%%%%%%%%%%%%%%%%%%%%%%%%%%
correspond to elements of $\W_{
z_1, \ldots, 
z_{m+m'+k}
}$ and 
$\W_{z_{m+m'+k+1}, \ldots, z_{m+m'+k+n}}$,    
we use Proposition \ref{pupa} again and obtain 
%%
%%%%%%%%%%%%%%%%%%%%%%%%%%%%%%%%%%%%%%%%%%%%%%%%%%%%%%%%%%%%%%%%%%%%%%%%%%%%%%%%%%%%%%%
%%
\begin{eqnarray*}
&& 
\langle w',  E^{(m)}_{W} \Big(
v_{k+1}, x_{k+1}; \ldots; 
v_{k+m}, x_{k+m};  
 P_{q} \Big(  
  Y^{W}_{WV}\left(  
\F (v_{1}, x_{1};  \ldots; v_{k}, x_{k}), \zeta_1\right)\; u \Big) \Big)\rangle 
\end{eqnarray*} 
%%
%and 
%%%%%%%%%%%%%%%%%%%%%%%%%%%%%%%%%%%%%%%%%%%%%%%%%%%%%%%%%%%%%%%%%%%%%%%%%%%%%%%%%
%%%%%%%%%%%%%%%%%%%%%%%%%%%%%%%%%%%%%%%%%%%%%%%%%%%%%%%%%%%%%%%%%%%%%%%%%%%%%%%%%
%%
\begin{eqnarray*}
& &
%%
 %\qquad   \qquad 
%%
%%
 \langle w',  E^{(m')}_{W} \Big(
v'_{n+1}, y_{n+1}; \ldots;
v'_{n+m'}, y_{n+m'}; 
 P_{q}\Big(    Y^{W}_{WV}\left( 
\F
(v'_{1}, y_{1}; \ldots; v'_{n}, y_{n}) , \zeta_{2}\right) \; \overline{u} \Big) \Big) \rangle,     
\end{eqnarray*}
correspondingly. 
Thus, the assertion of Lemma follows. 
\end{proof}
%%
%%%%%%%%%%%%%%%%%%%%%%%%%%%%%%%%%%%%%%%%%%%%%%%%%%%%%%%%%%%%%%%%%%%%%%%%%%%%%%%%%%%%%%
%%
%when 
%%
Under conditions
\begin{eqnarray}
\label{granizy2}
z_{i''}\ne z_{j''}, \quad i''\ne j'', \quad 
\nn
|z_{i''}|>|z_{k'''}|>0, 
\end{eqnarray}
 for $i''=1, \dots, m+m'$, and $k'''=m+m'+1, \dots, m+m'+ k+n$, 
let us introduce 
\begin{eqnarray}
\label{perda}
&&
 \mathcal J^{k+n}_{m+m'}(\Theta) = \sum_{q\in \C}
\langle w', E^{(m+m')}_{W} \Big(
v''_{1}, z_1; \ldots; 
v''_{m+m'}, z_{m+m'};  %\right. 
\nn
&&
\left. 
\qquad %\qquad 
 P_{q}\Big( \Theta( v''_{m+m'+1}, z_{m+m'+1}; \ldots; v''_{m+m'+k+n}, z_{m+m'+k+n})
; \epsilon \right)\Big)\rangle.  
\end{eqnarray}
Using Lemma \ref{obvlem} we obtain 
%% 
%%%%%%%%%%%%%%%%%%%%%%%%%%%%%%%%%%%%%%%%%%%%%%%%%%%%%%%%%%%%%%%%%%%%%%%%%%%%%%
%%%%%%%%%%%%%%%%%%%%%%%%%%%%%%%%%%%%%%%%%%%%%%%%%%%%%%%%%%%%%%%%%%%%%%%%%%%%%%
\begin{eqnarray*}
&&
|\mathcal J^{k+n}_{m+m'}(\Theta) | 
\nn
&&
= \left| \sum_{q\in \C}\langle w', E^{(m+m')}_{W} \Big(
v''_{1}, z_1; \ldots; %\otimes \ldots\otimes 
v''_{m+m'}, z_{m+m'};  \right. 
\nn
&&
\left. 
\qquad P_{q}\Big( \Theta( v''_{m+m'+1}, z_{m+m'+1}; \ldots; v''_{m+m'+k+n}, z_{m+m'+k+n})
; \epsilon)\Big)\rangle \right|
\end{eqnarray*}
%%%%%%%%%%%%%%%%%%%%%%%%%%%%%%%%%%%%%%%%%%%%%%%%%%%%%%%%%%%%%%%%%%%55
\begin{eqnarray*}
%\nn
&&
= \left| \sum_{q\in \C} 
\sum_{l \in \Z }  \epsilon^l 
 \sum_{u\in V_l } 
\langle w',  E^{(m)}_{W} \Big(
v_{k+1}, x_{k+1}; \ldots;
v_{k+m}, x_{k+m}; 
\right. 
\nn
&& 
\qquad \qquad  \qquad \qquad \qquad \qquad 
P_{q} \Big(  
  Y^{W}_{WV}\left(  
\Phi (v_{1}, x_{1};  \ldots; v_{k}, x_{k}), \zeta_1\right)\; u \Big) \Big)\rangle  
\nn
& &
%%
%\left. 
\qquad   \qquad 
 \langle w',  E^{(m')}_{W} \Big(
v'_{n+1}, y_{n+1}; \ldots; 
v'_{n+m'}, y_{n+m'};  
\nn
&&
\left. 
\qquad \qquad 
  \qquad \qquad \qquad \qquad 
P_{q}\Big(    Y^{W}_{WV}\left( 
\Psi
(v'_{1}, y_{1}; \ldots; v'_{n}, y_{n}) , \zeta_{2}\right) \; \overline{u} \Big) \Big) \rangle \right|
%%
%%
%%
%%%%%%%%%%%%%%%%%%%%%%%%%%%%%%%%%%%%%%%%%%%%%%%%%%%%%%%%%%%%%%%%%%%%%%%%%%%%%%%%%%%%%%
%%
 %= 
%%
\nn
&&
\le \left|
\mathcal J^{k}_{m}(\F) \right| \; \left|  \mathcal J^{n}_{m'}(\F)\right|, 
\end{eqnarray*}
%%
%%%%%%%%%%%%%%%%%%%%%%%%%%%%%%%%%%%%%%%%%%%%%%%%%%%%%%%%%%%%%%%%%%%%%%%%%%%%%%%%%%%
%%%
where we have used 
the invariance of \eqref{Z2n_pt_epsss} with respect to 
$\sigma \in S_{m+m'+k+n}$. 
According to Definitions \ref{composabilitydef} 
$\mathcal J^{k}_{m}(\Phi)$ and $\mathcal J^{n}_{m'}(\Psi)$ in the last expression 
are absolute convergent. 
%%
%and  
%
%%
Thus, we infer that  $\mathcal J^{k+n}_{m+m'}(\Theta)$ 
%% 
%%%%%%%%%%%%%%%%%%%%%%%%%%%%%%%%%%%%%%%%%%%%%%%%%%%%%%%%%%%%%%%%%%%%%%%%%%%%%%
%%
is absolutely convergent, and 
%%   
%%
%%%%%%%%%%%%%%%%%%%%%%%%%%%%%%%%%%%%%%%%%%%%%%%%%%%%%%%%%%%%%%%%%%%%5
%%
the sum \eqref{Inmdvadva} 
 is analytically extendable to a  rational function  
in $(z_{1}, \dots, z_{k+n+m+m'})$ with the only possible poles at 
$x_i=x_j$, $y_{i'}=y_{j'}$, and 
at $x_i=y_{j'}$, i.e., the only possible poles at 
$z_{i''}=z_{j''}$, of orders less than or equal to 
$N^{k+n}_{m+m'}(v''_{i''}, v''_{j''})$,  
for $i''$, $j''=1, \dots, k'''$, $i''\ne j''$.  
\end{proof}
%%

%%%%%%%%%%%%%%%%%%%%%%%%%%%%%%%%%%%%%%%%%%%%%%%%%%%%%%%%%%%%%%%%%%%%%%%%%%%%%%%%%%%%%%%%%%%%%%%
%%%%%%%%%%%%%%%%%%%%%%%%%%%%%%%%%%%%%%%%%%%%%%%%%%%%%%%%%%%%%%%%%%%%%%%%%%%%%%%%%%%%%%%%%%%%%%%
\subsection{Proof of Proposition \ref{tosya}}
%%
%%%%%%%%%%%%%%%%%%%%%%%%%%%%%%%%%%%%%%%%%%%%%%%%%%%%%%%%%%%%%%%%%%%%%%%%%%%%%%%%%%%%%%%%%%%%%%5
\begin{proof}
%%
%% 
%%%%%%%%%%%%%%%%%%%%%%%%%%%%%%%%%%%%%%%%%%%%%%%%%%%%%%%%%%%%%%%%%%%%%%%%%%%%%%%%%%%%%%%%%%%%%%%%%%%%%%%%%
%%
For a vertex operator $Y_{V, W}(v,z)$ let us introduce a notation 
\[
\omega_{V, W}=Y_{V, W}(v,z)\; dz^{{\rm wt} v}.
\] 
Let us use notations \eqref{zsto} and \eqref{notari}. 
According to \eqref{deltaproduct}, the action of 
$\delta_{m + m'-t}^{k + n-r}$ on $\widehat{R} \Theta(
v_1, x_1; \ldots; v_k, x_k; v'_1, y_1; \ldots; v'_k, y_n
%%
%%v_1', z_; \ldots; v'_{k+n'-r}, z_{k+n'-r}
%%
; \epsilon)$
is given by 
%% 
%%%%%%%%%%%%%%%%%%%%%%%%%%%%%%%%%%%%%%%%%%%%%%%%%%%%%%%%%%%%%%%%%%%%%%%%%%%%%%%%%%%%%%%%
\begin{eqnarray*}
&&   
\langle w', 
 \delta_{m + m'-t}^{k + n-r} \widehat{R} \; \Theta( 
v_1, x_1; \ldots; v_k, x_k; v'_1, y_1; \ldots; v'_n, y_n
%v_1', z_; \ldots; v'_{k+n'-r}, z_{k+n'-r}
; \epsilon) \rangle 
\nn
%%
%%%%%%%%%%%%%%%%%%%%%%%%%%%%%%%%%%%%%%%%%%%%%%%%%%%%%%%%%%%%%%%%%%%%%%%%%%%%%%%%%%%%%%%%%%%%%%
&& \quad =
\langle w', \sum_{i=1}^{k%+n
}(-1)^{i} \; 
\widehat{R} \; \Theta ( \widetilde{v}_1, z_1; \ldots;  \widetilde{v}_{i-1}, z_{i-1}; \;  \omega_V (\widetilde{v}_i, z_i  
 - z_{i+1}) 
 \widetilde{v}_{i+1}, z_{i+1}; \; \widetilde{v}_{i+2}, z_{i+2}; 
\nn
&& \qquad \qquad \qquad 
\ldots;  \widetilde{v}_k, z_k;  \widetilde{v}_{k+1}, z_{k+1}; \ldots; \widetilde{v}_{k+n}, z_{k+n}; \epsilon ) \rangle   
%%
%\nn
%&&
%\nn%
\end{eqnarray*}
%%
%%%%%%%%%%%%%%%%%%%%%%%%%%%%%%%%%%%%%%%%%%%%%%%%%%%%%%%%%%%%%%%%%%%%%%%%%%
\begin{eqnarray*}
%%%%%%%%%%%%%%%%
&& \qquad + 
\sum_{i=1}^{n-r}(-1)^{i} \; \langle w', 
\Theta \left( \widetilde{v}_1, z_1; \ldots; \widetilde{v}_k, z_k;  
\widetilde{v}_{k+1}, z_{k+1}; \ldots; \widetilde{v}_{k+i-1}, z_{k+i-1}; 
\right. 
\nn
&&\qquad \qquad \qquad  
  \omega_V \left(\widetilde{v}_{k+i}, z_{k+i}  
 - z_{k+i+1} ) \; 
 \widetilde{v}_{k+i+1}, z_{k+i+1}; \right.
\nn
&&
\left. 
\qquad \qquad \qquad \qquad \qquad \qquad \widetilde{v}_{k+i+2}, z_{k+i+2}; \ldots; \widetilde{v}_{k+n-r}, z_{k+n-r}; \epsilon \right) \rangle  
%%
%\nn
%%
\end{eqnarray*}
%%
%%%%%%%%%%%%%%%%%%%%%%%%%%%%%%%%%%%%%%%%%%%%%%%%%%%%%%%%%%%%%%%%%%%%%%%%%%
\begin{eqnarray*}
&& \qquad +  \langle w', 
 \omega_W \left(\widetilde{v}_1, z_1   
  \right) \; \Theta (\widetilde{v}_2, z_2
; \ldots; \widetilde{v}_{k}, z_k; \widetilde{v}_{k+1}, z_{k+1}; \ldots; \widetilde{v}_{k+n-r}, z_{k+n-r}; \epsilon   
)  \rangle  
%\nn
%%
%\nonumber
\end{eqnarray*}
%%
%%%%%%%%%%%%%%%%%%%%%%%%%%%%%%%%%%%%%%%%%%%%%%%%%%%%%%%%%%%%%%%%%%%%%%%%%%
\begin{eqnarray*}
 & &\qquad +  \langle w, (-1)^{k+n+1-r}    
 \omega_W(\widetilde{v}_{k+n-r+1}, z_{k+n-r+1}  
) 
\;
\nn
&&
 \qquad \qquad \qquad \F(\widetilde{v}_1, z_1
; \ldots; \widetilde{v}_k, z_k; \widetilde{v}_{k+1}, z_{k+1}; \ldots; \widetilde{v}_{k+n-r}, z_{k+n-r}; \epsilon ) \rangle 
%%
%%\nn
%&&
%\nn
%%
\end{eqnarray*}
%%
%%%%%%%%%%%%%%%%%%%%%%%%%%%%%%%%%%%%%%%%%%%%%%%%%%%%%%%%%%%%%%%%%%%%%%%%%%
\begin{eqnarray*}
%%
%%%%%%%%%%%%%%%%%%%%%%%%%%%%%%%%%%%%%%%%%%%%%%%%%%%%%%%%%%%%%%%%%%%%%%%%
%%%%%%%%%%%%%%%%%%%%%%%%%%%%%%%%%%%%%%%%%%%%%%%%%%%%%%%%%%%%%%%%%%%%%%%%
&& \quad = 
\sum_{l \in \Z }  \epsilon^l 
\langle w', \sum_{i=1}^{k}(-1)^{i} \; Y^W_{VW}(  
\Theta ( \widetilde{v}_1, z_1; \ldots;  \widetilde{v}_{i-1}, z_{i-1}; \; \omega_V (\widetilde{v}_i, z_i  
 - z_{i+1}) 
 \widetilde{v}_{i+1}, z_{i+1}; \; 
\nn
&& \qquad \qquad \qquad  
\widetilde{v}_{i+2}, z_{i+2};  \ldots;  \widetilde{v}_k, z_k), \zeta_1) u \rangle 
\nn
&&
\qquad \qquad \qquad \qquad \qquad \qquad
 \langle w', Y^W_{VW}( \Theta(\widetilde{v}_{k+1}, z_{k+1}; \ldots; \widetilde{v}_{k+n-r}, z_{k+n-r}), \zeta_2) 
\overline{u} \rangle  
%%
%%%%%%%%%%%%%%%%%%%%%%%%%%%%%%%%%%%%%%%%%%%%%%%%%%%%%%%%%%%%%%%%%%%%%%%
%\nn
\end{eqnarray*}
%%
%%%%%%%%%%%%%%%%%%%%%%%%%%%%%%%%%%%%%%%%%%%%%%%%%%%%%%%%%%%%%%%%%%%%%%%%%%
\begin{eqnarray*}
&& \qquad + 
\sum_{l \in \Z }  \epsilon^l  \sum_{i=1}^{n-r}(-1)^{i} \; \langle w',  Y^W_{VW}(
\Phi \left( \widetilde{v}_1, z_1; \ldots; \widetilde{v}_k, z_k), \zeta_1\right) u \rangle 
\nn
&&
\qquad \qquad \qquad 
   \langle w',  
Y^W_{VW}( \Psi(\widetilde{v}_{k+1}, z_{k+1}; \ldots; \widetilde{v}_{k+i-1}, z_{k+i-1};  
\nn
&& 
\qquad \qquad \qquad  \qquad \omega_V \left(\widetilde{v}_i, z_{k+i}  
 - z_{k+i+1}\right) \; 
 \widetilde{v}_{k+i+1}, z_{k+i+1}; \widetilde{v}_{k+i+2}, z_{k+i+2}; 
\nn
&&
\qquad \qquad \qquad  \qquad \qquad  \ldots; \widetilde{v}_{k+n-r}, z_{k+n-r} ), \zeta_2) \overline{u} \rangle  
%%
%%
%\nn
%&&
%\nn
\end{eqnarray*}
%%
%%%%%%%%%%%%%%%%%%%%%%%%%%%%%%%%%%%%%%%%%%%%%%%%%%%%%%%%%%%%%%%%%%%%%%%%%%
\begin{eqnarray*}
%%%%%%%%%%%%%%%%%%%%%%%%%%%%%%%%%%%%%%%%%%%%%%%%%%%%%%%%%%%%%%%%%%%%%%%%%%%%
%%
&& \qquad +  \sum_{l \in \Z }  \epsilon^l \langle w', Y^W_{VW}(
 \omega_W \left(\widetilde{v}_1, z_1 \right) \; \Phi (\widetilde{v}_2, z_2 ; \ldots; \widetilde{v}_{k}, z_k), \zeta_1) u \rangle 
\nn
&&
 \qquad \qquad \langle w', Y^W_{VW}( \Psi( \widetilde{v}_{k+1}, z_{k+1}; \ldots; \widetilde{v}_{k+n-r}, z_{k+n-r} ), \zeta_2) \overline{u} \rangle   
%\nn
\end{eqnarray*}
%%
%%%%%%%%%%%%%%%%%%%%%%%%%%%%%%%%%%%%%%%%%%%%%%%%%%%%%%%%%%%%%%%%%%%%%%%%%%
\begin{eqnarray*}
&&
\qquad +  \sum_{l \in \Z }  \epsilon^l  \langle w', Y^W_{VW}( (-1)^{k+1}  
 \omega_W \left(\widetilde{v}_{k+1}, z_{k+1} 
  \right) \;  \Phi (\widetilde{v}_1, z_1 ; \ldots; \widetilde{v}_{k}, z_k), \zeta_1) u\rangle  
\nn
&& \qquad \qquad \qquad 
\langle w', Y^W_{VW}(
\Psi( \widetilde{v}_{k+2}, z_{k+2}; \ldots; \widetilde{v}_{k+n-r}, z_{k+n-r}), \zeta_2) \overline{u} \rangle  
%%
%\nn
\end{eqnarray*}
%%
%%%%%%%%%%%%%%%%%%%%%%%%%%%%%%%%%%%%%%%%%%%%%%%%%%%%%%%%%%%%%%%%%%%%%%%%%%
\begin{eqnarray*}
&&
\qquad - 
%%%%%%%%%%%%%%%%%%%%%%%%%%%%%%%%%%%%%%%%%%%%%%%%%%%%%5
\sum_{l \in \Z }  \epsilon^l \langle w', (-1)^{k+1}   \langle w',  Y^W_{VW}(  
 \omega_W \left(\widetilde{v}_{k+1}, z_{k+1} 
  \right) \; \Phi (\widetilde{v}_1, z_1 ; \ldots; \widetilde{v}_{k}, z_k), \zeta_1) u \rangle  
\nn
&&
\qquad \qquad  \qquad 
\langle w',  Y^W_{VW}( \Psi(\widetilde{v}_{k+2}, z_{k+2}; \ldots; \widetilde{v}_{k+n-r}, z_{k+n-r}),  \zeta_2)  \overline{u} 
\rangle  
%%
%%%%%%%%%%%%%%%%%%%%%%%%%%%%%%%%%%%%%%%%%%%%%%%%%%%%%%%
%%
%%%%%%%%%%%%%%%%%%%%%%%%%%%%%%%%%%%%%%%%%%%%%%%%%%%%%%5
%\nn
%&&
\nn
&&
\qquad +  \sum_{l \in \Z }  \epsilon^l  \langle w', Y^W_{VW}(  
 \Phi(\widetilde{v}_1, z_1; \ldots; \widetilde{v}_k, z_k), \zeta_1)  u\rangle  
\nn
&&
\qquad \qquad 
 \langle w', Y^W_{VW}(  
\omega_W(\widetilde{v}_{k+n-r+1}, z_{k+n-r+1})\; 
\nn
&&
 \qquad \qquad \qquad  \qquad  \Psi(\widetilde{v}_{k+1}, z_{k+1}; \ldots; \widetilde{v}_{k+n-r}, z_{k+n-r}  ), \zeta_2)  \overline{u}\rangle 
%%
%%%%
\nn
&& \qquad 
%%%%%%%%5%%%%%%%%%%%%%%%%%%%%%%%%%%%%%%%%%%%%%%%%%%%%%%%%%%%%%%%
- \sum_{l \in \Z }  \epsilon^l \langle w', Y^W_{VW}(  
 \Phi(\widetilde{v}_1, z_1; \ldots; \widetilde{v}_k, z_k), \zeta_1) \rangle  
\nn
&& 
\qquad \qquad \langle w',  Y^W_{VW}(
   \omega_W(\widetilde{v}_{k+n-r+1}, z_{k+n-r+1})
\nn
&&
 \qquad \qquad \qquad  \qquad  \Psi( \widetilde{v}_{k+1}, z_{k+1}; \ldots; \widetilde{v}_{k+n-r}, z_{k+n-r}  ), \zeta_2) \rangle 
%%
%\nn
%&&
%\nn
%%
\end{eqnarray*}
%%
%%%%%%%%%%%%%%%%%%%%%%%%%%%%%%%%%%%%%%%%%%%%%%%%%%%%%%%%%%%%%%%%%%%%%%%%%%
\begin{eqnarray*}
%%
%%%%%%%%%%%%%%%%%%%%%%%%%%%%%%%%%%%%%%%%%%%%%%%%%%%%%%%%%%%%%%%%%%%%%%%%%%
%%%%%%%%%%%%%%%%%%%%%%%%%%%%%%%%%%%%%%%%%%%%%%%%%%%%%%%%%%%%%%%%%%%%%%%%%%
&& \quad = 
\sum_{l \in \Z }  \epsilon^l 
\langle w',  \; Y^W_{VW}(  
\delta^k_m\Phi ( \widetilde{v}_1, z_1; \ldots;  \widetilde{v}_k, z_k), \zeta_1) u \rangle 
\nn
&& 
\qquad \qquad \qquad \qquad \langle w', Y^W_{VW}( \Psi(\widetilde{v}_{k+1}, z_{k+1}; \ldots; \widetilde{v}_{k+n-r}, z_{k+n-r}), \zeta_2) 
\overline{u} \rangle  
%%
%%%%%%%%%%%%%%%%%%%%%%%%%%%%%%%%%%%%%%%%%%%%%%%%%%%%%%%%%%%%%%%%%%%%%%%
\nn
&& \qquad + (-1)^k 
\sum_{l \in \Z }  \epsilon^l   \langle w',  Y^W_{VW}(
\Phi \left( \widetilde{v}_1, z_1; \ldots; \widetilde{v}_k, z_k), \zeta_1 \right) u \rangle   
\nn
&&
\qquad \qquad \qquad 
  \langle w',  Y^W_{VW}( \delta^{n-r}_{m'-t}  \Psi(\widetilde{v}_{k+1}, z_{k+1}; \ldots;  \widetilde{v}_{k+n-r}, z_{k+n-r} ), \zeta_2 ) \overline{u} \rangle  
%%
%%
%\nn
%&&
%\nn
%%
%&& 
%%
\end{eqnarray*}
%%
%%%%%%%%%%%%%%%%%%%%%%%%%%%%%%%%%%%%%%%%%%%%%%%%%%%%%%%%%%%%%%%%%%%%%%%%%%
\begin{eqnarray*}
&& \quad = 
\langle w',   
\delta^k_m\Phi ( \widetilde{v}_1, z_1; \ldots;  \widetilde{v}_k, z_k) \cdot 
%%
%\nn
%&& 
%\qquad \qquad \qquad \qquad 
\langle w',  \Psi(\widetilde{v}_{k+1}, z_{k+1}; \ldots; \widetilde{v}_{k+n-r}, z_{k+n-r}) \rangle   
%%
%%%%%%%%%%%%%%%%%%%%%%%%%%%%%%%%%%%%%%%%%%%%%%%%%%%%%%%%%%%%%%%%%%%%%%%
\nn
&& \qquad + (-1)^k
   \langle w',  
 \Phi \left( \widetilde{v}_1, z_1; \ldots; \widetilde{v}_k, z_k\right) \cdot_{\epsilon}  
     \delta^{n-r}_{m'-t}  \Psi(\widetilde{v}_{k+1}, z_{k+1}; \ldots;  \widetilde{v}_{k+n-r}, z_{k+n-r} ) \rangle,  
\end{eqnarray*}
%%
%%%%%%%%%%%%%%%%%%%%%%%%%%%%%%%%%%%%%%%%%%%%%%%%%%%%%%%%%%%%%%%%%%%%%%%%%%%5
%%%%%%%%%%%%%%%%%%%%%%%%%%%%%%%%%%%%%%%%%%%%%%%%%%%%%%%%%%%%%%%%%%%%%%%%%%%%a
since, 
\begin{eqnarray*}
&& \sum_{l \in \Z }  \epsilon^l \langle w', (-1)^{k+1}   Y^W_{VW}(  
 \omega_W \left(\widetilde{v}_{k+1}, z_{k+1} 
  \right) \; \Phi (\widetilde{v}_1, z_1 ; \ldots; \widetilde{v}_{k}, z_k), \zeta_1) u \rangle  
\nn
&&
\qquad \qquad \langle w',  Y^W_{VW}( \Psi(\widetilde{v}_{k+2}, z_{k+2}; \ldots; \widetilde{v}_{k+n-r}, z_{k+n-r}),  \zeta_2) \overline{u} 
\rangle  
%\nn
%&&
\end{eqnarray*}
%%%%%%%%%%%%%%%%%%%%%%%%%%%%%%%%%%%%%%%%%%%%%%%%%%%%%%%%%%%%%%%%%%%%%%%%%%%%%%%%%%%%%%
\begin{eqnarray*} 
%\nn
&&
=\sum_{l \in \Z }  \epsilon^l \langle w', (-1)^{k+1} e^{\zeta_1 L_W{(-1)}}   Y_W(u, -\zeta_1)  \;  
 \omega_W \left(\widetilde{v}_{k+1}, z_{k+1}  
  \right) \; \Phi (\widetilde{v}_1, z_1 ; \ldots; \widetilde{v}_{k}, z_k) \rangle  
\nn
&&
\qquad \qquad \langle w',  Y^W_{VW}( \Psi(\widetilde{v}_{k+2}, z_{k+2}; \ldots; \widetilde{v}_{k+n-r}, z_{k+n-r}),  \zeta_2) \overline{u}
\rangle  
%%
%\nn
%&&
%%%%%%%%%%%%%%%%%%%%%%%%%%%%%%%%%%%%%%%%%%%%%%%%%%%%%%%%%%%%%%%%%%%%%%%%%%%%%%%%%%%%%%
%\nn
\end{eqnarray*}
%%%%%%%%%%%%%%%%%%%%%%%%%%%%%%%%%%%%%%%%%%%%%%%%%%%%%%%%%%%%%%%%%%%%%%%%%%%%%%%%%%%%%%
\begin{eqnarray*} 
&&
=\sum_{l \in \Z }  \epsilon^l \langle w', (-1)^{k+1} e^{\zeta_1 L_W{(-1)}}  \omega_W \left(\widetilde{v}_{k+1}, z_{k+1}  \right) 
 Y_W(u, -\zeta_1)  \;  
  \; \Phi (\widetilde{v}_1, z_1 ; \ldots; \widetilde{v}_{k}, z_k) \rangle  
\nn
&&
\qquad \qquad \langle w',  Y^W_{VW}( \Psi(\widetilde{v}_{k+2}, z_{k+2}; \ldots; \widetilde{v}_{k+n-r}, z_{k+n-r}),  \zeta_2) \overline{u}
\rangle  
%%
%\nn
%&&
\end{eqnarray*}
%%%%%%%%%%%%%%%%%%%%%%%%%%%%%%%%%%%%%%%%%%%%%%%%%%%%%%%%%%%%%%%%%%%%%%%%%%%%%%%%%%%%%%
\begin{eqnarray*} 
%%%%%%%%%%%%%%%%%%%%%%%%%%%%%%%%%%%%%%%%%%%%%%%%%%%%%%%%%%%%%%%%%%%%%%%%%%%%%%%%%%%%%%
\nn
&&
= \sum\limits_{v\in V} 
%%
%\sum_{l \in \Z }  \epsilon^l 
\langle w', (-1)^{k+1} \; 
\omega_W \left(\widetilde{v}_{k+1}, z_{k+1} +\zeta_1  \right)\;  e^{\zeta_1 L_W{(-1)}}  
 Y_W(u, -\zeta_1)  \;  
  \; \Phi (\widetilde{v}_1, z_1 ; \ldots; \widetilde{v}_{k}, z_k) \rangle  
\nn
&&
\qquad \qquad \langle w',  Y^W_{VW}( \Psi(\widetilde{v}_{k+2}, z_{k+2}; \ldots; \widetilde{v}_{k+n-r}, z_{k+n-r}),  \zeta_2) \overline{u}
\rangle  
%%
%\nn
%&&
\end{eqnarray*}
%%%%%%%%%%%%%%%%%%%%%%%%%%%%%%%%%%%%%%%%%%%%%%%%%%%%%%%%%%%%%%%%%%%%%%%%%%%%%%%%%%%%%%
\begin{eqnarray*} 
%%%%%%%%%%%%%%%%%%%%%%%%%%%%%%%%%%%%%%%%%%%%%%%%%%%%%%%%%%%%%%%%%%%%%%%%%%%%%%%%%%%%%%
%\nn
&&
=\sum\limits_{v\in V} 
 \sum_{u\in V_l } 
\sum_{l \in \Z }  \epsilon^l 
 \sum_{u\in V_l } 
\langle v', (-1)^{k+1} \; \omega_W \left(\widetilde{v}_{k+1}, z_{k+1}+\zeta_1  \right) w \rangle  
\nn
&&
\qquad \qquad \langle w',  e^{\zeta_1 L_W{(-1)}}  
 Y_W(u, -\zeta_1)  \;  
  \; \Phi (\widetilde{v}_1, z_1 ; \ldots; \widetilde{v}_{k}, z_k) \rangle  
\nn
&&
\qquad \qquad \langle w',  Y^W_{VW}( \Psi(\widetilde{v}_{k+2}, z_{k+2}; \ldots; \widetilde{v}_{k+n-r}, z_{k+n-r}),  \zeta_2) \overline{u}
\rangle  
%%
%\nn
%&&
\end{eqnarray*}
%%%%%%%%%%%%%%%%%%%%%%%%%%%%%%%%%%%%%%%%%%%%%%%%%%%%%%%%%%%%%%%%%%%%%%%%%%%%%%%%%%%%%%
\begin{eqnarray*} 
%%%%%%%%%%%%%%%%%%%%%%%%%%%%%%%%%%%%%%%%%%%%%%%%%%%%%%%%%%%%%%%%%%%%%%%%%%%%%%%%%%%%%%
%\nn
&&
= 
\sum_{l \in \Z }  \epsilon^l 
%%
%\nn
%&&
%%
%\qqaud \qquad 
 \langle w',  e^{\zeta_1 L_W{(-1)}}  
 Y_W(u, -\zeta_1)  \;  
  \; \Phi (\widetilde{v}_1, z_1 ; \ldots; \widetilde{v}_{k}, z_k) \rangle  
\nn
&&
\qquad \qquad \sum\limits_{v\in V} \langle v', (-1)^{k+1} \; \omega_W \left(\widetilde{v}_{k+1}, z_{k+1}+\zeta_1  \right) w \rangle   
\nn
&&
\langle w',  Y^W_{VW}( \Psi(\widetilde{v}_{k+2}, z_{k+2}; \ldots;
\nn
&&
 \qquad \qquad \qquad \qquad \widetilde{v}_{k+n-r}, z_{k+n-r}),  \zeta_2) \overline{u}
\rangle  
%%
%\nn
%&&
\end{eqnarray*}
%%%%%%%%%%%%%%%%%%%%%%%%%%%%%%%%%%%%%%%%%%%%%%%%%%%%%%%%%%%%%%%%%%%%%%%%%%%%%%%%%%%%%%
\begin{eqnarray*} 
%%%%%%%%%%%%%%%%%%%%%%%%%%%%%%%%%%%%%%%%%%%%%%%%%%%%%%%%%%%%%%%%%%%%%%%%%%%%%%%%%%%%%%
%\nn
&&
= 
\sum_{l \in \Z }  \epsilon^l 
%%
%\nn
%&&
%%
%\qqaud \qquad 
 \langle w',  
 Y^{W}_{VW}(  
   \Phi (\widetilde{v}_1, z_1 ; \ldots; \widetilde{v}_{k}, z_k) , \zeta_1) u \; \rangle  
\nn
&&
\qquad \qquad 
 \langle w', (-1)^{k+1} \; \omega_W \left(\widetilde{v}_{k+1}, z_{k+1}+\zeta_1  \right) 
\;
\nn
&&
 \qquad \qquad  Y^W_{VW}( \Psi(\widetilde{v}_{k+2}, z_{k+2}; \ldots; \widetilde{v}_{k+n-r}, z_{k+n-r}),  \zeta_2) \overline{u}
\rangle  
%%
%\nn
%&&
%%%%%%%%%%%%%%%%%%%%%%%%%%%%%%%%%%%%%%%%%%%%%%%%%%%%%%%%%%%%%%%%%%%%%%%%%%%%%%%%%%%%%
%\nn
\end{eqnarray*}
%%%%%%%%%%%%%%%%%%%%%%%%%%%%%%%%%%%%%%%%%%%%%%%%%%%%%%%%%%%%%%%%%%%%%%%%%%%%%%%%%%%%%%
\begin{eqnarray*} 
&&
= 
\sum_{l \in \Z }  \epsilon^l 
%%
%\nn
%&&
%%
%\qqaud \qquad 
 \langle w',  
 Y^{W}_{VW}(  
   \Phi (\widetilde{v}_1, z_1 ; \ldots; \widetilde{v}_{k}, z_k) , \zeta_1) u \; \rangle  
\nn
&&
\qquad \qquad 
 \langle w', (-1)^{k+1} \; \omega_W \left(\widetilde{v}_{k+1}, z_{k+1}+\zeta_1  \right) 
\nn
&&
\qquad \qquad 
\; e^{\zeta_2 L_W{(-1)}}    Y_W(\overline{u}, -\zeta_2) \; \Psi(\widetilde{v}_{k+2}, z_{k+2}; \ldots; \widetilde{v}_{k+n-r}, z_{k+n-r}) 
\rangle  
%%
%\nn
%&&
\end{eqnarray*}
%%%%%%%%%%%%%%%%%%%%%%%%%%%%%%%%%%%%%%%%%%%%%%%%%%%%%%%%%%%%%%%%%%%%%%%%%%%%%%%%%%%%%%
\begin{eqnarray*} 
%%%%%%%%%%%%%%%%%%%%%%%%%%%%%%%%%%%%%%%%%%%%%%%%%%%%%%%%%%%%%%%%%%%%%%%%%%%%%%%%%%%%%%
%\nn
&&
= 
\sum_{l \in \Z }  \epsilon^l 
%%
%\nn
%&&
%%
%\qqaud \qquad 
 \langle w',  
 Y^{W}_{VW}(  
   \Phi (\widetilde{v}_1, z_1 ; \ldots; \widetilde{v}_{k}, z_k) , \zeta_1) u \; \rangle  
\nn
&&
\qquad 
 \langle w', (-1)^{k+1} \; 
\; e^{\zeta_2 L_W{(-1)}} \;   Y_W(\overline{u}, -\zeta_2)
\;
\omega_W \left(\widetilde{v}_{k+1}, z_{k+1}+\zeta_1-\zeta_2 \right)  
\nn
&&
\qquad \qquad \; \Psi(\widetilde{v}_{k+2}, z_{k+2}; \ldots; \widetilde{v}_{k+n-r}, z_{k+n-r}) 
\rangle  
%%
%\nn
%&&
%%
\end{eqnarray*}
%%%%%%%%%%%%%%%%%%%%%%%%%%%%%%%%%%%%%%%%%%%%%%%%%%%%%%%%%%%%%%%%%%%%%%%%%%%%%%%%%%%%%%
\begin{eqnarray*} 
%%%%%%%%%%%%%%%%%%%%%%%%%%%%%%%%%%%%%%%%%%%%%%%%%%%%%%%%%%%%%%%%%%%%%%%%%%%%%%%%%%%%%%
%\nn
&& 
\qquad \qquad 
= \sum_{l \in \Z }  \epsilon^l \langle w', Y^W_{VW}(  
 \Phi(\widetilde{v}_1, z_1; \ldots; \widetilde{v}_k, z_k), \zeta_1) u \rangle  
\nn
&& 
\qquad \qquad \langle w',  Y^W_{VW}(
   \omega_W(\widetilde{v}_{k+1}, z_{k+1}) \; \Psi( \widetilde{v}_{k+2}, z_{k+2}; \ldots; \widetilde{v}_{k+n-r}, z_{k+n-r}  ), \zeta_2) \overline{u} \rangle,   
\end{eqnarray*}
%%
%%%%%%%%%%%%%%%%%%%%%%%%%%%%%%%%%%%%%%%%%%%%%%%%%%%%%%%%%%%%%%%%%%%%%%%%%%%%%%%%%%%%%%%%%%%%
%%
due to locality \eqref{porosyataw} of vertex opertors, and arbitrarness of $\widetilde{v}_{k+1}\in V$ and $z_{k+1}$, 
we can always put
\[
\omega_W \left(\widetilde{v}_{k+1}, z_{k+1}+\zeta_1-\zeta_2 \right)  =\omega_W(\widetilde{v}_{k+2}, z_{k+2}), 
\]
for $\widetilde{v}_{k+1}=\widetilde{v}_{k+2}$, $z_{k+2}= z_{k+1}+\zeta_2-\zeta_1$.  
%%
%, 
%%
The statement of the proposition for $\delta^2_{ex}$ \eqref{halfdelta} can be checked accordingly. 
\end{proof}
%%
%%%%%%%%%%%%%%%%%%%%%%%%%%%%%%%%%%%%%%%%%%%%%%%%%%%%%%%%%%%%%%%%%%%%%%%%%%%%%%%%%%%%%%%%%
%%
%%%%%%%%%%%%%%%%%%%%%%%%%%%%%%%%%%%%%%%%%%%%%%%%%%%%%%%%%%%%%%%%%%%%%%%%%%%%%%%%%%%%%%%%%%%%%%%%
%%%%%%%%%%%%%%%%%%%%%%%%%%%%%%%%%%%%%%%%%%%%%%%%%%%%%%%%%%%%%%%%%%%%%%%%%%%%%%%%%%%%%%%%%%%%%%%%

%%
\end{document}